\documentclass[12pt,american]{amsart}
\usepackage[T1]{fontenc}
\usepackage[latin9]{inputenc}
\usepackage{geometry}
\usepackage{babel}
\usepackage{mathrsfs}
\usepackage{amstext}
\usepackage{amsthm}
\usepackage{amssymb}
\usepackage{color}

\makeatletter
\numberwithin{equation}{section}
\numberwithin{figure}{section}
\theoremstyle{plain}
\newtheorem{thm}{\protect\theoremname}[section]
\theoremstyle{remark}
\newtheorem{rem}[thm]{\protect\remarkname}
\theoremstyle{plain}
\newtheorem{lem}[thm]{\protect\lemmaname}
\theoremstyle{definition}
\newtheorem{defn}[thm]{\protect\definitionname}
\theoremstyle{plain}
\newtheorem{prop}[thm]{\protect\propositionname}
\theoremstyle{plain}
\newtheorem{cor}[thm]{\protect\corollaryname}

\usepackage[svgnames]{xcolor}
\usepackage[bookmarksnumbered=true]{hyperref} 
\hypersetup{
	colorlinks = true,
	linkcolor = Blue,
	anchorcolor = blue,
	citecolor = Green,
	filecolor = blue,
	urlcolor = FireBrick
}
\numberwithin{equation}{section}
\usepackage{bbm}

\makeatother

\providecommand{\corollaryname}{Corollary}
\providecommand{\definitionname}{Definition}
\providecommand{\lemmaname}{Lemma}
\providecommand{\propositionname}{Proposition}
\providecommand{\remarkname}{Remark}
\providecommand{\theoremname}{Theorem}

\begin{document}

\title[On Landis Conjecture for the Fractional Schr\"{o}dinger Equation]{On Landis Conjecture for the Fractional Schr\"{o}dinger Equation}

\author{Pu-Zhao Kow}

\address{Department of Mathematics and Statistics, P.O. Box 35 (MaD), FI-40014 University of Jyv\"{a}skyl\"{a}, Finland.}

\email{\href{mailto:pu-zhao.pz.kow@jyu.fi}{pu-zhao.pz.kow@jyu.fi}}
\begin{abstract}
In this paper, we study a Landis-type conjecture for the general fractional Schr\"{o}dinger equation $((-P)^{s}+q)u=0$. As a byproduct, we also proved the additivity and boundedness of the linear operator $(-P)^{s}$ for non-smooth coefficents. For differentiable potentials $q$, if a solution decays at a rate $\exp(-|x|^{1+})$, then the solution vanishes identically. For non-differentiable potentials $q$, if a solution decays at a rate $\exp(-|x|^{\frac{4s}{4s-1}+})$, then the solution must again be trivial. The proof relies on delicate Carleman estimates. This study is an extension of the work by R\"{u}land-Wang (2019).
\end{abstract}

\maketitle

\maketitle

\section{Introduction}

In this work, we study a Landis-type conjecture for the fractional Schr\"{o}dinger equation 
\begin{equation}
((-P)^{s}+q)u=0\quad\text{in }\mathbb{R}^{n}, \quad \text{where} \quad P=\sum_{j,k=1}^{n}\partial_{j}a_{jk}(x)\partial_{k} \label{eq:Sch}
\end{equation}
with $s\in(0,1)$ and $|q(x)|\le1$. Here, the operator $(-P)^s$ is defined as 
\begin{equation}
(-P)^{s}u(x) := \int_{0}^{\infty} \lambda^{s} \, \mathsf{d}E_{\lambda} = \frac{1}{\Gamma(-s)} \int_{0}^{\infty}(e^{tP}-1)u(x)\,\frac{\mathsf{d}t}{t^{1+s}}  \label{eq:definition-fractional}
\end{equation}
for all
\begin{equation*}
u \in {\rm dom}\,((-P)^{s}) := \bigg\{ u \in L^{2}(\mathbb{R}^{n}) : \int_{0}^{\infty} \lambda^{2s} \, \mathsf{d} \| E_{\lambda}u \|^{2} < \infty \bigg\}
\end{equation*}
where $\{E_{\lambda}\}$ is the spectral resolution of $-P$ (each $\{E_{\lambda}\}$ is a projection in $L^{2}(\mathbb{R}^{n})$) and  $\{e^{tP}\}_{t\ge0}$ is the heat-diffusion semigroup generated by $-P$, see e.g. \cite{GLX17FractionalCalderon,ST10ExtensionProblem}. 

The Landis conjecture was proposed by E.M. Landis in the 60's \cite{KL88}. He conjectured the following statement: Let $|q(x)|\le1$ and let $u$ be a solution to  \eqref{eq:Sch} with $P=\Delta$ and $s=1$. If $|u(x)|\le C_{0}$ and $|u(x)|\le\exp(-C|x|^{1+})$, then $u\equiv0$. However, this statement is false: In \cite{Mes92Landis}, Meshkov constructed a (complex-valued) potential $q$ and a (complex-valued) nontrivial $u$ with $|u(x)|\le C\exp(-C|x|^{\frac{4}{3}})$. In the same literature, he also showed that if $|u(x)|\le C\exp(-C|x|^{\frac{4}{3}+})$, then $u\equiv0$. In other words, the exponent $\frac{4}{3}+$ is optimal. In \cite{BK05quantitativeLandis}, Bourgain and Kenig derived a quantitative form of Meshkov's result, which is based on the Carleman method; their result then extended by Davey in \cite{Dav14MagneticSch}, including the drift term. Following, in \cite{LW14Landis}, Lin and Wang further extend Davey's result by replacing $\Delta$ by $P$. 

The results mentioned above allowing \emph{complex-valued} solutions. It is also interesting to study the real-version of Landis conjecture, which proposed by Kenig in \cite[Question~1]{Ken06RealLandis}. The case when $n=1$ and $n=2$ were resolved in \cite{Ros21realLandis} and \cite{LMNN20realLandis}, respectively. To the best of the author's knowledge, the real-version of Landis conjecture is still open for $n \ge 3$. Here we also refer some related works \cite{Dav20realLandis,DKW17realLandis,DKW20realLandis,DW20LandisConjecturePlane,KSW15realLandis}.

In \cite{RW19Landis}, R\"{u}land and Wang consider the Landis conjecture of the fractional Schr\"{o}dinger equation \eqref{eq:Sch} with $P=\Delta$ and $0<s<1$. For the case when $s=1/2$, in \cite{Cas20SharpExponentialDecay}, we remark that Cassano proved the Landis conjecture for the Dirac equation. In some sense, the Dirac operator is the square root of the Laplacian operator, that is, the phenomena are similar when $s=1/2$. 

\subsection{Main results}

We assume that the second order elliptic operator $P$ satisfies the elliptic condition 
\begin{equation}
	\lambda|\xi|^{2}\le\sum_{j,k=1}^{n}a_{jk}(x)\xi_j\xi_{k}\le\lambda^{-1}|\xi|^{2}\quad\text{for some constant }0<\lambda\le 1.\label{eq:unif-ellip}
\end{equation}
Assume that $a_{jk}=a_{kj} \in \mathcal{C}^{0,1}(\mathbb{R}^{n})$ for all $1\le j,k \le n$, and satisfy 
\begin{equation}
	\max_{1\le j,k\le n}\sup_{|x|\ge1}|a_{jk}(x)-\delta_{jk}(x)|+\max_{1\le j,k\le n}\sup_{|x|\ge1}|x||\nabla a_{jk}(x)|\le\epsilon\label{eq:decay}
\end{equation}
for some sufficiently small $\epsilon>0$ and
\begin{equation}
	\max_{1\le j,k\le n}\sup_{|x|\ge1}|\nabla^{2}a_{jk}(x)|\le C\label{eq:2der-bdd}
\end{equation}
for some positive constant $C$. 

In this paper, we prove the following Landis-type conjecture for the fractional Schr\"{o}dinger equations. 
\begin{thm}
\label{thm:result1}Let $s\in(0,1)$ and assume that $u\in {\rm dom}\,((-P)^{s})$ is a solution to \eqref{eq:Sch} with \eqref{eq:unif-ellip}, \eqref{eq:decay} and \eqref{eq:2der-bdd}. We assume that the potential $q\in\mathcal{C}^{1}(\mathbb{R}^{n})$
satisfies $|q(x)|\le1$ and

\[
|x||\nabla q(x)|\le1.
\]
If $u$ further satisfies 
\[
\int_{\mathbb{R}^{n}}e^{|x|^{\alpha}}|u|^{2}\,\mathsf{d}x\le C<\infty\quad\text{for some }\alpha>1,
\]
then $u\equiv0$. 
\end{thm}

We also have the following result for non-differentiable potential
$q$. 
\begin{thm}
\label{thm:result2}Let $s\in(1/4,1)$ and assume that $u\in {\rm dom}\,((-P)^{s})$
is a solution to \eqref{eq:Sch} with \eqref{eq:unif-ellip}, \eqref{eq:decay}, and \eqref{eq:2der-bdd}. Now we assume that the potential $q$ satisfies
$|q(x)|\le1$. If $u$ satisfies 

\[
\int_{\mathbb{R}^{n}}e^{|x|^{\alpha}}|u|^{2}\,\mathsf{d}x\le C<\infty\quad\text{for some }\alpha>\frac{4s}{4s-1},
\]
then $u\equiv0$. 
\end{thm}

\begin{rem}
When $s=\frac 12$, Theorem \ref{thm:result1} and Theorem
\ref{thm:result2} still hold without \eqref{eq:2der-bdd}.
\end{rem}

\begin{rem}
We prove Theorem~\ref{thm:result2} using the splitting arguments in \cite{RW19Landis}. Therefore, due to the sub-ellipticity nature, we the same restriction $s\in(\frac{1}{4},1)$. We also see that, as $s\rightarrow1$, the exponent $\frac{4s}{4s-1}$ in Theorem~\ref{thm:result2} tends to $\frac{4}{3}$, which is the optimal exponent for the classical Schr\"{o}dinger equation.
\end{rem}

\begin{rem}
	The condition \eqref{eq:decay} allows small perturbations of Laplacian only, which works as a sufficient condition in deriving Carleman estimate. In \cite{GR19unique}, they also imposed similar assumption to prove the strong unique continuation property for \eqref{eq:Sch}. In contrast to the works \cite{DKW17realLandis,Ros21realLandis}, which studied the \emph{real-version} of Landis conjecture, such condition is not needed, since their proofs did not involve any Carleman estimate.
\end{rem}

\subsection{Main ideas}

The main method of proving Theorem~\ref{thm:result1} and \ref{thm:result2} is Carleman estimates. However, due to the non-locality of $(-P)^s$, the techniques here are much complicated than those for the classical case, i.e., $s=1$. One of the major tricks is to localize $(-P)^s$, which is motivated by Caffarelli-Silvestre's fundamental work \cite{CS07extension}. Here we will use the Caffarelli-Silvestre type extension of $(-P)^s$ proved in \cite{ST10ExtensionProblem,Sti10Fractional}. After localizing $(-P)^s$, we will derive a Carleman estimate on $\mathbb{R}_{+}^{n+1}$ mimicking the one proved in \cite{RS20Calderon}. This Carleman estimate enables passing of the boundary decay to the bulk decay. 

\subsection{\label{sec:regularity}Main difficulties: Regularity of $(-P)^{s}$}

Using the Fourier transform, it is easy to see that
\[
(-\Delta)^{\alpha}(-\Delta)^{\beta}=(-\Delta)^{\alpha+\beta} \quad \text{and} \quad (-\Delta)^{s} \in \mathcal{L}(\dot{H}^{\beta+s}(\mathbb{R}^{n}),\dot{H}^{\beta-s}(\mathbb{R}^{n})).
\]
However, extension of these properties to $(-P)^{s}$ is not trivial. We establish the additivity property of $(-P)^{s}$ by introducing the Balakrishnan definition of $(-P)^{s}$, which is equivalent to \eqref{eq:definition-fractional}, see e.g. \cite{MS01FractionalOperators} or \cite[Section IX.11]{Yos80FunctionalAnalysis}. The continuity of $(-P)^{s}:H^{2s}(\mathbb{R}^{n})\rightarrow L^{2}(\mathbb{R}^{n})$ can be also obtained by the Balakrishnan operator, as well as the interpolation of the single operator $-P$. Here, we shall not interpolate on the family of the operator $(-P)^{s}$, see also \cite{GM14AnalyticFamiliesMultilinear} for the interpolation theory of the analytic familiy of multilinear operators. 

\begin{rem}
In \cite{See67ComplexPowerEllipticOperator}, R.T. Seeley showed that the operator $(-P)^{s}$ is a pseudo-differential operator of order $2s$ if $a_{jk}$ are smooth. In this case, we can apply the theory of pseudo-differential operator, see e.g. \cite{Tay74PseudoDifferentialOperator}. As a byproduct, we loosened the smoothness hypothesis that required by theories of the pseudo-differential operator. Moreover, the boundary value theories for the fractional Laplacian have been elaborated in recent years, see e.g. \cite{Gru14MuTransmissionFractionalElliptic,Gru15FractionalLaplacianDomains,Gru16IntegrationByPartsPohozaev,Gru16RegularitySpectralFractional,Gru20ExactGreenFormulaFractionalLaplacian}. In \cite{Gru20ExactGreenFormulaFractionalLaplacian}, Grubb calculated  the first few terms in the symbol of $(-P)^{s}$.\footnote{I would like to thank Prof Gerd Grubb for bringing these issues to my attention and for pointing out several related references.} 
\end{rem}

\subsection{Main difficulties: Carleman estimates}

In \cite{RW19Landis}, they proved their Carleman estimates by estimating a certain commutator term, see \cite[(31)--(33)]{RW19Landis}. In our case, we shall approximate $P$ by $\Delta$. However, we face difficulties while controlling the remainder terms. Here, we solve this problem using the ideas in \cite{Reg97StrongUniquenessSecondOrderElliptic}. It is also interesting to mention that the terms of second derivative in the Carleman estimate should be $\tilde{\nabla}(\nabla\tilde{u})$ rather than $\tilde{\nabla}^{2}\tilde{u}$, where $\tilde{\nabla}=(\nabla,\partial_{n+1})$ is the gradient operator on $\mathbb{R}^{n+1}$, and $\tilde{u}$ is the Caffarelli-Silvestre type extension of $u$.

\subsection{Organization of the paper}

In Section~\ref{sec:definition}, we localize the operator $(-P)^{s}$ and solve the problems described in Paragraph~\ref{sec:regularity}. Following, in Section~\ref{sec:decay}, we show that the decay of $u$ implies the decay of the Caffarelli-Silvestre type extension $\tilde{u}$ of $u$. Then, we derive some delicate Carleman estimates on $\mathbb{R}_{+}^{n}$ in Section~\ref{sec:Carleman}. Finally, we prove Theorem~\ref{thm:result1} and Theorem~\ref{thm:result2} in Section~\ref{sec:result}. 

\section{\label{sec:definition}Caffarelli-Silvestre type Extension}

Let $\mathbb{R}_{+}^{n+1}=\mathbb{R}^{n}\times\mathbb{R}_{+}=\{(x',x_{n+1}):x_{n+1} > 0\}$, and we write $x=(x',x_{n+1})$ with $x'\in\mathbb{R}^{n}$ and $x_{n+1}\in\mathbb{R}_{+}$. We also denote $\nabla'=(\partial_{1},\cdots\partial_{n})$ and $\nabla = (\nabla' , \partial_{n+1})$. For $x_{0}\in\mathbb{R}^{n}\times\{0\}$, we denote the half balls in $\mathbb{R}_{+}^{n+1}$ and $\mathbb{R}^{n}\times\{0\}$ by 
\begin{align*}
B_{r}^{+}(x_{0})&:=\{x\in\mathbb{R}_{+}^{n+1}:|x-x_{0}|\le r\}, \\ 
B_{r}'(x_{0})&:=\{(x',0) \in \mathbb{R}^{n}\times\{0\} : |(x',0) - x_{0}| \le r\},
\end{align*}
$B_{r}^{+}(0)=B_{r}^{+}$, and $B_{r}'(0)=B_{r}'$. We define the annulus 
\begin{align*}
A_{r,R}^{+} &:= \{x\in\mathbb{R}_{+}^{n+1}:r\le|x|\le R\} \\ 
A_{r,R}' &:= \{(x',0)\in\mathbb{R}^{n}\times\{0\}:r\le|(x',0)|\le R\}.
\end{align*}
We consider the following Sobolev spaces: 
\begin{align*}
L^{2}(D,x_{n+1}^{1-2s}) & :=\bigg\{ v:D\rightarrow\mathbb{R}:\int_{D}x_{n+1}^{1-2s}|v|^{2}\,\mathsf{d}x<\infty\bigg\},\\
\dot{H}^{1}(D,x_{n+1}^{1-2s}) & :=\bigg\{ v:D\rightarrow\mathbb{R}:\int_{D}x_{n+1}^{1-2s}|\nabla v|^{2}\,\mathsf{d}x<\infty\bigg\},\\
H^{1}(D,x_{n+1}^{1-2s}) & :=\bigg\{ v:D\rightarrow\mathbb{R}:\int_{D}x_{n+1}^{1-2s}(|v|^{2}+|\nabla v|^{2})\,\mathsf{d}x<\infty\bigg\},
\end{align*}
where $D$ is a relative open set in $\overline{\mathbb{R}_{+}^{n+1}}$.

For $s\in(0,1)$, let $\tilde{u}$ be a solution to the following degenerate elliptic equation: 
\begin{align}
\bigg[\partial_{n+1}x_{n+1}^{1-2s}\partial_{n+1}+x_{n+1}^{1-2s}P\bigg]\tilde{u} & =0\quad\text{in }\mathbb{R}_{+}^{n+1},\label{eq:Sch-ext}\\
\tilde{u} & =u\quad\text{on }\mathbb{R}^{n}\times\{0\}.\label{eq:Dirichlet-BC}
\end{align}
Refer to \cite[equation~(1.8) in Theorem 1.1]{ST10ExtensionProblem}, the fractional elliptic operator $(-P)^{s}$ satisfies 
\begin{equation}
(-P)^{s}u(x')=c_{s}\lim_{x_{n+1}\rightarrow0}x_{n+1}^{1-2s}\partial_{n+1}\tilde{u}(x) \label{eq:Neumann-BC}
\end{equation}
with
\[
c_{s} = \frac{4^{s}\Gamma(s)}{2s\Gamma(-s)} < 0 \quad (\text{In particular }c_{1/2}=-1),
\]
see also \cite{Sti10Fractional}. 
The following lemma is a special case of \cite[Proposition~2.1]{GR19unique}: 
\begin{lem}
\label{lem:well0} Let $0<s<1$, and assuming that $a_{jk}=a_{kj}\in \mathcal{C}^{0,1}(\mathbb{R}^{n})$ satisfies the elliptic condition \eqref{eq:unif-ellip}. Then there exists an extension operator
\[
\mathsf{E}_{s}:{\rm dom}\,((-P)^{s}) \rightarrow 
H_{\rm loc}^{1}(\mathbb{R}_{+}^{n+1},x_{n+1}^{1-2s}) \cap \mathcal{C}_{\rm loc}^{2,1}(\mathbb{R}_{+}^{n+1})
\]
such that $\tilde{u} = \mathsf{E}_{s}(u)$ is a solution of \eqref{eq:Sch-ext} and the boundary conditions \eqref{eq:Dirichlet-BC} and \eqref{eq:Neumann-BC} are attained as $L^{2}(\mathbb{R}^{n})$-limits. 
\end{lem}
The proof of Lemma~\ref{lem:well0} is same as in \cite{ST10ExtensionProblem,Sti10Fractional}. The following estimate also holds true: 
\begin{equation}
	\|\tilde{u}(\bullet,x_{n+1})\|_{L^{2}(\mathbb{R}^{n})}\le\|u\|_{L^{2}(\mathbb{R}^{n})} \quad \text{for all}\;\; x_{n+1}>0.\label{eq:apriori}
\end{equation}
with $\tilde{u} = \mathsf{E}_{s}(u)$, see \cite[page~2097]{ST10ExtensionProblem} or \cite[pages~48--49]{Sti10Fractional}. From \cite[Proposition~2.6]{Yu17unique}, indeed 
\begin{equation}
\mathsf{E}_{s} : H^{s}(\mathbb{R}^{n}) \rightarrow H_{\rm loc}^{1}(\mathbb{R}_{+}^{n+1},x_{n+1}^{1-2s}) \label{eq:extension-domain}
\end{equation}
is a bounded linear operator. Using \cite[Remark~7.4]{LM72FractionalSobolevSpacesVol1}, we know that  
\begin{equation*}
	\mathcal{C}_{c}^{\infty}(\overline{\mathbb{R}_{+}^{n+1}}) \text{ is dense in } H_{\rm loc}^{1}(\mathbb{R}_{+}^{n+1},x_{n+1}^{1-2s}),
\end{equation*}
thus, given any $v \in H^{s}(\mathbb{R}^{n})$, we have $\tilde{v} = \mathsf{E}_{s}(v) \in H_{\rm loc}^{1}(\mathbb{R}_{+}^{n+1},x_{n+1}^{1-2s})$ and 
\begin{align*}
& \bigg| \int_{\mathbb{R}^{n} \times \{0\}} ((-P)^{s}u) v \, \mathsf{d}x' \bigg| \equiv \bigg| \int_{\mathbb{R}^{n} \times \{0\}} \bigg( \lim_{x_{n+1} \rightarrow 0}x_{n+1}^{1-2s} \partial_{n+1}\tilde{u}\bigg) v \, \mathsf{d}x' \bigg|\\ 
= & \bigg| \int_{\mathbb{R}_{+}^{n+1}} x_{n+1}^{1-2s} \partial_{n+1}^{1-2s} \partial_{n+1} \tilde{u} \partial_{n+1} \tilde{v} \, \mathsf{d}x + \int_{\mathbb{R}_{+}^{n+1}} A(x') \nabla' \tilde{u} \cdot \nabla' \tilde{v} \, \mathsf{d}x \bigg| \\
\le & \lambda^{-1} \| \nabla \tilde{u} \|_{L^{2}(\mathbb{R}_{+}^{n+1},x_{n+1}^{1-2s})} \| \nabla \tilde{v} \|_{L^{2}(\mathbb{R}_{+}^{n+1},x_{n+1}^{1-2s})} \\
\equiv & \lambda^{-1} \| \mathsf{E}_{s}(u) \|_{\dot{H}^{1}(\mathbb{R}_{+}^{n+1},x_{n+1}^{1-2s})} \| \mathsf{E}_{s}(v) \|_{\dot{H}^{1}(\mathbb{R}_{+}^{n+1},x_{n+1}^{1-2s})} \\ 
\le & C \| u \|_{H^{s}(\mathbb{R}^{n})} \| v \|_{H^{s}(\mathbb{R}^{n})} \quad \text{using \eqref{eq:extension-domain}}.
\end{align*}
Therefore, by arbitrariness of $v \in H^{s}(\mathbb{R}^{n})$, we conclude the following lemma: 
\begin{lem}
\label{lem:well1} Let $0<s<1$ and $a_{jk}$ given as in Lemma~{\rm \ref{lem:well0}}. Then $(-P)^{s} : H^{s}(\mathbb{R}^{n}) \rightarrow H^{-s}(\mathbb{R}^{n})$ is a bounded linear operator. 
\end{lem}

Note that 
\[
Pu=\sum_{j,k=1}^{n}a_{jk}\partial_{j}\partial_{k}u+\sum_{j,k=1}^{n}(\partial_{j}a_{jk})\partial_{k}u.
\]
Since $a_{jk}$ is uniformly Lipschitz, then 
\begin{equation}
\|-Pu\|_{L^{2}(\mathbb{R}^{n})}\le C\|u\|_{H^{2}(\mathbb{R}^{n})}.\label{eq:dom1}
\end{equation}
We here also remark that ${\rm dom}\,(-P) = H^{2}(\mathbb{R}^{n})$ is the maximal extension such that $-P$ is self-adjoint and densely defined in $L^{2}(\mathbb{R}^{n})$, see \cite[equation~(2.8)]{GLX17FractionalCalderon}. Given any $\phi \in \mathcal{C}_{c}^{\infty}(\mathbb{R}^{n})$, we see that 
\[
\langle Pu,\phi \rangle = (u,P\phi)_{L^{2}(\mathbb{R}^{n})} \le \|u\|_{L^{2}(\mathbb{R}^{n})} \|P\phi\|_{L^{2}(\mathbb{R}^{n})} \le C\|u\|_{L^{2}(\mathbb{R}^{n})} \|\phi\|_{H^{2}(\mathbb{R}^{n})}
\]
where $\langle\bullet,\bullet\rangle$ is the $H^{-2}(\mathbb{R}^{n})\oplus H^{2}(\mathbb{R}^{n})$ duality pair. Since 
\begin{equation*}
\mathcal{C}_{c}^{\infty}(\mathbb{R}^{n}) \text{ is dense in } H^{\gamma}(\mathbb{R}^{n}) \text{ for each }\gamma \in \mathbb{R} \quad (\text{see e.g. \cite[Remark~7.4]{LM72FractionalSobolevSpacesVol1}}),
\end{equation*}
then we know that 
\begin{equation}
\|Pu\|_{H^{-2}(\mathbb{R}^{n})} \le C\|u\|_{L^{2}(\mathbb{R}^{n})}.\label{eq:dom2}
\end{equation}
We shall prove the followings: 
\begin{lem}
\label{lem:well2} Let $0<s<1$ and $a_{jk}$ given as in Lemma~{\rm \ref{lem:well0}}. We have the inequality
\begin{equation}
\|(-P)^{s}u\|_{L^{2}(\mathbb{R}^{n})}\le C\|u\|_{H^{2s}(\mathbb{R}^{n})}.\label{eq:well2-1}
\end{equation}
Moreover, we have 
\begin{equation}
\|(-P)^{s}u\|_{H^{-2s}(\mathbb{R}^{n})}\le C\|u\|_{L^{2}(\mathbb{R}^{n})}.\label{eq:well2-2}
\end{equation}
\end{lem}

\begin{rem}
Using the duality argument as in \eqref{eq:dom2}, we know that \eqref{eq:well2-1} and \eqref{eq:well2-2} are equivalent. 
\end{rem}

In order to prove Lemma~\ref{lem:well2}, we introduce the Balakrishnan operator as in \cite[Definition~3.1.1 and Definition~5.1.1]{MS01FractionalOperators}. 

\begin{defn}
Let $\alpha\in\mathbb{C}_{+}=\{z\in\mathbb{C}:\Re z>0\}$. 
\begin{enumerate}
\item If $0<\Re\alpha<1$, then ${\rm dom}\,((-P)_{B}^{\alpha})={\rm dom}\,(-P)$
and 
\[
(-P)_{B}^{\alpha}\phi=\frac{\sin\alpha\pi}{\pi}\int_{0}^{\infty}\lambda^{\alpha-1}(\lambda-P)^{-1}(-P)\phi\,d\lambda.
\]
\item If $\Re\alpha=1$, then ${\rm dom}\,((-P)_{B}^{\alpha})={\rm dom}\,((-P)^{2})$
and 
\[
(-P)_{B}^{\alpha}\phi=\frac{\sin\alpha\pi}{\pi}\int_{0}^{\infty}\lambda^{\alpha-1}\bigg[(\lambda-P)^{-1}-\frac{\lambda}{\lambda^{2}+1}\bigg](-P)\phi\,d\lambda+\sin\frac{\alpha\pi}{2}(-P)\phi.
\]
\item If $n<\Re\alpha<n+1$ for $n\in\mathbb{N}$, then ${\rm dom}\,((-P)_{B}^{\alpha})={\rm dom}\,((-P)^{n+1})$
and 
\[
(-P)_{B}^{\alpha}\phi=(-P)_{B}^{\alpha-n}(-P)^{n}\phi.
\]
\item If $\Re\alpha=n+1$ for $n\in\mathbb{N}$, then ${\rm dom}\,((-P)_{B}^{\alpha})={\rm dom}\,((-P)^{n+2})$
and 
\[
(-P)_{B}^{\alpha}\phi=(-P)_{B}^{\alpha-n}(-P)^{n}\phi.
\]
\end{enumerate}
\end{defn}

The following proposition, which can be found at \cite[Theorem~6.1.6]{MS01FractionalOperators}, shows that $(-P)_{B}^{s}$ and $(-P)^{s}$ are equivalent. 
\begin{prop}
Let $0<s<1$. If $u\in{\rm dom}\,((-P)_{B}^{s})$, then the strong limit 
\[
\lim_{\epsilon\rightarrow0_{+}}\int_{\epsilon}^{\infty}(1-e^{tP})u\,\frac{dt}{t^{1+s}}\quad\text{exists}
\]
and 
\[
(-P)_{B}^{s}u=c_{s}'\lim_{\epsilon\rightarrow0_{+}}\int_{\epsilon}^{\infty}(1-e^{tP})u\,\frac{dt}{t^{1+s}}\quad\text{for some positive constant }c_{s}',
\]
where $\{e^{tP}\}_{t\ge0}$ is the heat-diffusion semigroup generated
by $-P$. 
\end{prop}

Here and after, we shall not distinguish
between $(-P)^{s}$ and $(-P)_{B}^{s}$, as well as ${\rm dom}\,((-P)^{s})$ and ${\rm dom}\,((-P)_{B}^{s})$. Using \cite[Theorem~5.1.2]{MS01FractionalOperators}, we have the following fact: 
\[
\text{If } u \in {\rm dom}\,((-P)^{\alpha + \beta}) \text{, then }(-P)^{\beta}u \in {\rm dom}\,((-P)^{\alpha}),
\]	
and the following identity holds:  
\begin{equation}
(-P)^{\alpha}(-P)^{\beta}u=(-P)^{\alpha+\beta}u \quad \text{for all}\;\;u\in {\rm dom}\,((-P)^{\alpha + \beta}) \label{eq:add}
\end{equation}
for all $\alpha,\beta\in\mathbb{C}$ with $\Re \alpha>0$
and $\Re \beta>0$. Since $(-P)^{s}$ is self-adjoint in $L^{2}(\mathbb{R}^{n})$, then  
\[
\|(-P)^{s}u\|_{L^{2}(\mathbb{R}^{n})}^{2} = ((-P)^{2s}u,u)_{L^{2}(\mathbb{R}^{n})}.
\]
Now we are ready to prove Lemma~\ref{lem:well2}. 
\begin{proof}[Proof of Lemma~{\rm \ref{lem:well2}}]
We first consider the case when $0< s \le 1/2$. Since $(-P)^{s}$ is self-adjoint, by observing that $(-P)^{2s}=(-P)^{s}(-P)^{s}$ (using \eqref{eq:add}), Lemma~\ref{lem:well1} immediate implies
\begin{equation}
\|(-P)^{s}u\|_{L^{2}(\mathbb{R}^{n})}^{2} = ((-P)^{2s}u,u)_{L^{2}(\mathbb{R}^{n})} \le \|(-P)^{2s}u\|_{H^{-2s}(\mathbb{R}^{n})} \|u\|_{H^{2s}(\mathbb{R}^{n})} \le C\|u\|_{H^{2s}(\mathbb{R}^{n})}^{2}.\label{eq:continuity-2s}
\end{equation}
When $1/2<s<1$, by observing that $(-P)^{2s}=(-P)^{2s-1}(-P)=(-P)(-P)^{2s-1}$ (using \eqref{eq:add}) and $0<2s-1<1$, using Lemma~\ref{lem:well1} we can easily show that 
\begin{align*}
\|(-P)^{2s}u\|_{H^{1-2s}(\mathbb{R}^{n})} & \le C\|u\|_{H^{1+2s}(\mathbb{R}^{n})} \\
\|(-P)^{2s}u\|_{H^{-1-2s}(\mathbb{R}^{n})} & \le C\|u\|_{H^{-1+2s}(\mathbb{R}^{n})}.
\end{align*}
By interpolating the above two inequalities, we conclude that \eqref{eq:continuity-2s} holds for all $0<s<1$, and we complete the proof of Lemma~\ref{lem:well2}. 
\end{proof}

\section{\label{sec:decay}Boundary Decay Implies Bulk Decay}

Firstly, we translate the decay behavior on $\mathbb{R}^{n}$
to decay behavior which is also holds on $\mathbb{R}_{+}^{n+1}$. 
\begin{prop}
\label{prop:bulk-decay} Let $s\in(0,1)$ and $u\in H^{s}(\mathbb{R}^{n})$
be a solution to \eqref{eq:Sch}, with \eqref{eq:unif-ellip} and \eqref{eq:decay}.
For $s\neq\frac{1}{2}$, we further assume \eqref{eq:2der-bdd}. Assume
that $|q(x)|\le1$ and there exists $\alpha>1$ such that 
\[
\int_{\mathbb{R}^{n}}e^{|x|^{\alpha}}|u|^{2}\,dx\le C<\infty.
\]
Then there exist constants $C_{1},C_{2}>0$  so that
the Caffarelli-Silvestre type extension $\tilde{u}(x)$ satisfies
\[
|\tilde{u}(x)|\le C_{1}e^{-C_{2}|x|^{\alpha}}\quad\text{for all }x\in\mathbb{R}_{+}^{n+1}.
\]
\end{prop}
The ideas of proving Proposition~\ref{prop:bulk-decay} is similar to \cite[Proposition~2.2]{RW19Landis}. The proof of \cite[Proposition~2.2]{RW19Landis} utilized \cite[Propositions~5.10--5.12]{RS20Calderon}. The extension of such propositions involving many details, especially the Carleman estimate in \cite[Propositions~5.7]{RS20Calderon}. For sake of readability, here we present the details of the proofs.

In order to obtain the interior decay, similar to \cite[Proposition~2.3]{RW19Landis}, we need the following three-ball inequalities. 
\begin{lem}
\label{lem:interior}Let $s\in(0,1)$ and $\tilde{u}\in H^{1}(B_{4}^{+},x_{n+1}^{1-2s})$
be a solution to 
\[
\bigg[\partial_{n+1}x_{n+1}^{1-2s}\partial_{n+1}+x_{n+1}^{1-2s}P\bigg]\tilde{u}=0\quad\text{in }\mathbb{R}_{+}^{n+1}
\]
with \eqref{eq:unif-ellip}. Assume that $r\in(0,1)$ and $\overline{x}_{0}=(\overline{x}_{0}',5r)\in B_{2}^{+}$.
Then, there exists $\alpha=\alpha(n,s)\in(0,1)$ such that 
\[
\|\tilde{u}\|_{L^{\infty}(B_{2r}^{+}(\overline{x}_{0}))}\le C\|\tilde{u}\|_{L^{\infty}(B_{r}^{+}(\overline{x}_{0}))}^{\alpha}\|\tilde{u}\|_{L^{\infty}(B_{4r}^{+}(\overline{x}_{0}))}^{1-\alpha}.
\]
\end{lem}

\begin{proof}
As $(\overline{x}_{0})_{n+1}=5r$, this follows from a standard interior
$L^{2}$ three ball inequalities together with $L^{\infty}$-$L^{2}$
estimates for uniformly elliptic equations. 
\end{proof}
Also, we need the following boundary-bulk propagation of smallness
estimation: 
\begin{lem}
\label{lem:small}Let $s\in(0,1)$ and let $\tilde{u}\in H^{1}(\mathbb{R}_{+}^{n+1},x_{n+1}^{1-2s})$
be a solution to \eqref{eq:Sch-ext} with \eqref{eq:unif-ellip} and
$q\in L^{\infty}(\mathbb{R}^{n})$. We assume that 
\[
\max_{1\le j,k\le n}\|a_{jk}-\delta_{jk}\|_{\infty}+\max_{1\le j,k\le n}\|\nabla'a_{jk}\|_{\infty}\le\epsilon
\]
for some sufficiently small $\epsilon>0$. For $s\neq\frac{1}{2}$,
we further assume 
\[
\max_{1\le j,k\le n}\|(\nabla')^{2}a_{jk}\|_{\infty}\le C
\]
for some positive constant $C$. Assume that $x_{0}\in\mathbb{R}^{n}\times\{0\}$.
Then 
\begin{enumerate}
	\renewcommand{\labelenumi}{\theenumi}
	\renewcommand{\theenumi}{(\alph{enumi})}
\item \label{itm:part-a:lem:small} There exist $\alpha=\alpha(n,s)\in(0,1)$ and $c=c(n,s)\in(0,1)$
such that 
\begin{align*}
 & \|x_{n+1}^{\frac{1-2s}{2}}\tilde{u}\|_{L^{2}(B_{cr}^{+}(x_{0}))}\\
\le & C\bigg[\|x_{n+1}^{\frac{1-2s}{2}}\tilde{u}\|_{L^{2}(B_{16r}^{+}(x_{0}))}+r^{1-s}\|u\|_{L^{2}(B_{16r}'(x_{0}))}\bigg]^{\alpha}\times\\
 & \quad\times\bigg[r^{s+1}\|\lim_{x_{n+1}\rightarrow0}x_{n+1}^{1-2s}\partial_{n+1}\tilde{u}\|_{L^{2}(B_{16r}'(x_{0}))}+r^{1-s}\|u\|_{L^{2}(B_{16r}'(x_{0}))}\bigg]^{1-\alpha}\\
 & +C\bigg[\|x_{n+1}^{\frac{1-2s}{2}}\tilde{u}\|_{L^{2}(B_{16r}^{+}(x_{0}))}+r^{1-s}\|u\|_{L^{2}(B_{16r}'(x_{0}))}\bigg]^{\frac{2s}{1+s}}\times\\
 & \quad\times\bigg[r^{s+1}\|\lim_{x_{n+1}\rightarrow0}x_{n+1}^{1-2s}\partial_{n+1}\tilde{u}\|_{L^{2}(B_{16r}'(x_{0}))}+r^{1-s}\|u\|_{L^{2}(B_{16r}'(x_{0}))}\bigg]^{\frac{1-s}{1+s}}.
\end{align*}
\item \label{itm:part-b:lem:small} There exist $\alpha=\alpha(n,s)\in(0,1)$ and $c=c(n,s)\in(0,1)$
such that 
\begin{align*}
 & \|x_{n+1}^{\frac{1-2s}{2}}\tilde{u}\|_{L^{\infty}(B_{\frac{cr}{2}}^{+}(x_{0}))}\\
\le & Cr^{-\frac{n}{2}}\bigg[r^{s-1}\|x_{n+1}^{\frac{1-2s}{2}}\tilde{u}\|_{L^{2}(B_{16r}^{+}(x_{0}))}+\|u\|_{L^{2}(B_{16r}'(x_{0}))}\bigg]^{\alpha}\times\\
 & \quad\times\bigg[r^{2s}\|\lim_{x_{n+1}\rightarrow0}x_{n+1}^{1-2s}\partial_{n+1}\tilde{u}\|_{L^{2}(B_{16r}'(x_{0}))}+\|u\|_{L^{2}(B_{16r}'(x_{0}))}\bigg]^{1-\alpha}\\
 & +Cr^{-\frac{n}{2}}\bigg[r^{s-1}\|x_{n+1}^{\frac{1-2s}{2}}\tilde{u}\|_{L^{2}(B_{16r}^{+}(x_{0}))}+\|u\|_{L^{2}(B_{16r}'(x_{0}))}\bigg]^{\frac{2s}{1+s}}\times\\
 & \quad\times\bigg[r^{2s}\|\lim_{x_{n+1}\rightarrow0}x_{n+1}^{1-2s}\partial_{n+1}\tilde{u}\|_{L^{2}(B_{16r}'(x_{0}))}+\|u\|_{L^{2}(B_{16r}'(x_{0}))}\bigg]^{\frac{1-s}{1+s}}\\
 & +Cr^{-\frac{n}{2}}r^{s}\|qu\|_{L^{2}(B_{16r}'(x_{0}))}^{\frac{1}{2}}\|u\|_{L^{2}(B_{16r}'(x_{0}))}^{\frac{1}{2}}.
\end{align*}
\end{enumerate}
\end{lem}

Using Lemma~\ref{lem:interior} and Lemma~\ref{lem:small}, and imitating the chain-ball argument in \cite{RW19Landis}, we can obtain Proposition~\ref{prop:bulk-decay}. 

\subsection{Proof of the part~\ref{itm:part-a:lem:small} of Lemma~\ref{lem:small} for the case $s\in[1/2,1)$}

We first prove the following extension of the Carleman estimate in \cite[Proposition 5.7]{RS20Calderon}. 
\begin{lem}
\label{lem:0Carl}Let $s\in[\frac{1}{2},1)$ and let $w\in H^{1}(\mathbb{R}_{+}^{n+1},x_{n+1}^{1-2s})$
with ${\rm supp}(w)\subset B_{1/2}^{+}$ be a solution to 
\begin{align*}
\bigg[\partial_{n+1}x_{n+1}^{1-2s}\partial_{n+1}+x_{n+1}^{1-2s}\sum_{j,k=1}^{n}a_{jk}\partial_{j}\partial_{k}\bigg]w & =f\quad\text{in }\mathbb{R}_{+}^{n+1},\\
w & =0\quad\text{on }\mathbb{R}^{n}\times\{0\}.
\end{align*}
Suppose that 
\[
\phi(x)=\phi(x',x_{n+1}):=-\frac{|x'|^{2}}{4}+2\bigg(-\frac{1}{2-2s}x_{n+1}^{2-2s}+\frac{1}{2}x_{n+1}^{2}\bigg).
\]
We assume that 
\[
\max_{1\le j,k\le n}\|a_{jk}-\delta_{jk}\|_{\infty}+\max_{1\le j,k\le n}\|\nabla'a_{jk}\|_{\infty}\le\epsilon
\]
for some sufficiently small $\epsilon>0$. For $s\neq\frac{1}{2}$,
we further assume 
\[
\max_{1\le j,k\le n}\|(\nabla')^{2}a_{jk}\|_{\infty}\le C
\]
for some positive constant $C$. Assume additionally that 
\begin{align*}
 & \|x_{n+1}^{\frac{2s-1}{2}}f\|_{L^{2}(\mathbb{R}_{+}^{n+1})}+\lim_{x_{x+1}\rightarrow0}\|\Delta'w\|_{L^{2}(\mathbb{R}^{n}\times\{0\})}\\
 & +\lim_{x_{n+1}\rightarrow0}\|\nabla'w\|_{L^{2}(\mathbb{R}^{n}\times\{0\})}+\lim_{x_{n+1}\rightarrow0}\|x_{n+1}^{1-2s}\partial_{n+1}w\|_{L^{2}(\mathbb{R}^{n}\times\{0\})}<\infty.
\end{align*}
Then there exist $\tau_{0}>1$ and a constant $C$ such that 
\begin{align*}
 & \tau^{3}\|e^{\tau\phi}x_{n+1}^{\frac{1-2s}{2}}w\|_{L^{2}(\mathbb{R}_{+}^{n+1})}^{2}+\tau\|e^{\tau\phi}x_{n+1}^{\frac{1-2s}{2}}\nabla w\|_{L^{2}(\mathbb{R}_{+}^{n+1})}^{2}\\
\le & C\bigg(\|e^{\tau\phi}x_{n+1}^{\frac{2s-1}{2}}f\|_{L^{2}(\mathbb{R}_{+}^{n+1})}^{2}+\tau^{-1}\lim_{x_{n+1}\rightarrow0}\|e^{\tau\phi}\Delta'w\|_{L^{2}(\mathbb{R}^{n}\times\{0\})}^{2}\\
 & +\tau\lim_{x_{n+1}\rightarrow0}\|e^{\tau\phi}x'\cdot\nabla'w\|_{L^{2}(\mathbb{R}^{n}\times\{0\})}^{2}+\tau\lim_{x_{n+1}\rightarrow0}\|e^{\tau\phi}x_{n+1}^{1-2s}\partial_{n+1}w\|_{L^{2}(\mathbb{R}^{n}\times\{0\})}^{2}\bigg)
\end{align*}
for all $\tau\ge\tau_{0}$. 
\end{lem}

\begin{proof}
Now we prove the Carleman estimate for $s\in(\frac{1}{2},1)$, as the case $s=\frac{1}{2}$ is naturally included in our estimates.

\textbf{Step 1: Conjugation.} Let $\tilde{u}=x_{n+1}^{\frac{1-2s}{2}}w$, we have 
\[
x_{n+1}^{\frac{2s-1}{2}}f = \Delta\tilde{u} + \mathring{c}_{s} x_{n+1}^{-2}\tilde{u}+\sum_{j,k=1}^{n}(a_{jk}-\delta_{jk})\partial_{j}\partial_{k}\tilde{u},
\]
where $\mathring{c}_{s} = \frac{1-4s^{2}}{4}$. Let $u=e^{\tau\phi}\tilde{u}$, we have 
\begin{align*}
e^{\tau\phi}x_{n+1}^{\frac{2s-1}{2}}f= & \bigg[\Delta+\tau^{2}|\nabla\phi|^{2} + \mathring{c}_{s} x_{n+1}^{-2}-\tau\Delta\phi-2\tau\nabla\phi\cdot\nabla\bigg]u\\
 & +\sum_{j,k=1}^{n}(a_{jk}-\delta_{jk})\partial_{j}\partial_{k}u -\tau\sum_{j,k=1}^{n}(a_{jk}-\delta_{jk})\bigg[(\partial_{k}\phi)\partial_{j}+(\partial_{j}\phi)\partial_{k}\bigg]u\\
 & +\sum_{j,k=1}^{n}(a_{jk}-\delta_{jk})\bigg[\tau^{2}(\partial_{k}\phi)(\partial_{j}\phi)-\tau(\partial_{j}\partial_{k}\phi)\bigg]u.
\end{align*}
We write $L^{+}=S+A+(I)+(II)+(III)$, where 
\begin{align*}
S & =\Delta+\tau^{2}|\nabla\phi|^{2} + \mathring{c}_{s} x_{n+1}^{-2}, \quad A = -2\tau\nabla\phi\cdot\nabla-\tau\Delta\phi,\\
(I) & =\sum_{j,k=1}^{n}(a_{jk}-\delta_{jk})\partial_{j}\partial_{k}\\
(II) & =-\tau\sum_{j,k=1}^{n}(a_{jk}-\delta_{jk})\bigg[(\partial_{k}\phi)\partial_{j}+(\partial_{j}\phi)\partial_{k}\bigg]\\
(III) & =\sum_{j,k=1}^{n}(a_{jk}-\delta_{jk})\bigg[\tau^{2}(\partial_{k}\phi)(\partial_{j}\phi)-\tau(\partial_{j}\partial_{k}\phi)\bigg].
\end{align*}
We now define $L^{-}:=S-A+(I)-(II)+(III)$, 
\[
\mathscr{D}:=\|L^{+}u\|^{2}-\|L^{-}u\|^{2} \quad \text{and} \quad \mathscr{S} :=  \|L^{+}u\|^{2}+\|L^{-}u\|^{2},
\]
where
\begin{align*}
	\|\bullet\| =\|\bullet\|_{L^{2}(\mathbb{R}_{+}^{n+1})}, & \qquad \|\bullet\|_{0} =\|\bullet\|_{L^{2}(\mathbb{R}^{n}\times\{0\})}\\
	\langle\bullet,\bullet\rangle =\langle\bullet,\bullet\rangle_{L^{2}(\mathbb{R}_{+}^{n+1})}, & \qquad \langle\bullet,\bullet\rangle_{0} =\langle\bullet,\bullet\rangle_{L^{2}(\mathbb{R}^{n}\times\{0\})}
\end{align*}
and we omit the notations ``$\lim_{x_{n+1}\rightarrow0}$'' in $\|\bullet\|_{0}$
and $\langle\bullet,\bullet\rangle_{0}$. 

\textbf{Step 2: Estimating the bulk contributions.}

\textbf{Step 2.1: Estimating the difference $\mathscr{D}$.} Observe that $\mathscr{D}=4\langle Su,Au\rangle+R$, where 
\[
R=4\langle Su,(II)u\rangle+4\langle Au,(I)u\rangle+4\langle Au,(III)u\rangle+4\langle(I)u,(II)u\rangle+4\langle(II)u,(III)u\rangle.
\]

\textbf{Step 2.1.1: Computation the principal term.} Note that 
	\begin{equation}
		2\langle Su,Au\rangle=\langle[S,A]u,u\rangle+2\tau\langle Su,(\partial_{n+1}\phi)u\rangle_{0}-\langle Au,\partial_{n+1}u\rangle_{0}+\langle\partial_{n+1}(Au),u\rangle_{0}.\label{eq:0Carl1}
	\end{equation}
	Observe that $[S,A]=[S,A]_{1}+[S,A]_{2}$, where 
	\begin{align*}
		[S,A]_{1}= & [\Delta'+\tau^{2}|\nabla'\phi|^{2},-2\tau\nabla'\phi\cdot\nabla'-\tau\Delta'\phi],\\{}
		[S,A]_{2}= & \bigg[\partial_{n+1}^{2}+\tau^{2}(\partial_{n+1}\phi)^{2}+\frac{1-4s^{2}}{4}x_{n+1}^{-2},-2\tau\partial_{n+1}\phi\partial_{n+1}-\tau\partial_{n+1}^{2}\phi\bigg].
	\end{align*}
The following identity can be found in \cite[equation~(5.20) of Proposition~5.7]{RS20Calderon}: 
\begin{equation}
	\langle[S,A]_{1}u,u\rangle=-\frac{1}{2}\tau^{3}\||x'|u\|^{2}-2\tau\|\nabla'u\|^{2}.\label{eq:0Carl2}
\end{equation}
For our purpose, we need to refine the estimate \cite[equation~(5.22) of Proposition~5.7]{RS20Calderon}. The following identity can be found in \cite[equation~(5.19) of Proposition~5.7]{RS20Calderon}: 
\begin{align*}
	\langle[S,A]_{2}u,u\rangle= & 4\tau^{3}\langle u,(\partial_{n+1}\phi)^{2}(\partial_{n+1}^{2}\phi)u\rangle+4\tau\langle\partial_{n+1}u,(\partial_{n+1}^{2}\phi)\partial_{n+1}u\rangle\\
	& -\tau\langle u,(\partial_{n+1}^{4}\phi)u\rangle -4\mathring{c}_{s} \tau\langle u,x_{n+1}^{-3}(\partial_{n+1}\phi)u\rangle\\
	& +4\tau\langle(\partial_{n+1}^{2}\phi)\partial_{n+1}u,u\rangle_{0}.
\end{align*}
From \cite[equation~(5.21) of Proposition~5.7]{RS20Calderon}, we have 
\[
(\partial_{n+1}\phi)^{2}(\partial_{n+1}^{2}\phi)=8(x_{n+1}^{1-2s}-x_{n+1})^{2}((2s-1)x_{n+1}^{-2s}+1)
\]
and
\begin{align*}
	& -\tau\langle u,(\partial_{n+1}^{4}\phi)u\rangle+(2s+1)(2s-1)\tau\langle u,x_{n+1}^{-3}(\partial_{n+1}\phi)u\rangle\\
	= & 2\tau(1-2s)(1+2s)^{2}\|x_{n+1}^{-1-s}u\|^{2} -8 \tau \mathring{c}_{s} \|x_{n+1}^{-1}u\|^{2}.
\end{align*}
Hence, we have 
\begin{align}
	\langle[S,A]_{2}u,u\rangle= & 32\tau^{3}\|(x_{n+1}^{1-2s}-x_{n+1})u\|^{2}+32(2s-1)\tau^{3}\|(x_{n+1}^{1-2s}-x_{n+1})x_{n+1}^{-s}u\|^{2}\nonumber \\
	& +8\tau(2s-1)\|x_{n+1}^{-s}\partial_{n+1}u\|^{2} +2\tau(1-2s)(1+2s)^{2}\|x_{n+1}^{-1-s}u\|^{2}\nonumber \\
	& +8\tau\|\partial_{n+1}u\|^{2} -8 \tau \mathring{c}_{s} \|x_{n+1}^{-1}u\|^{2}\nonumber \\
	& +4\tau\langle(\partial_{n+1}^{2}\phi)\partial_{n+1}u,u\rangle_{0}.\label{eq:0Carl3}
\end{align}
Combining \eqref{eq:0Carl1}, \eqref{eq:0Carl2} and \eqref{eq:0Carl3},
we reach 
\begin{align}
4\langle Su,Au\rangle= & 64\tau^{3}\|(x_{n+1}^{1-2s}-x_{n+1})u\|^{2}+64(2s-1)\tau^{3}\|(x_{n+1}^{1-2s}-x_{n+1})x_{n+1}^{-s}u\|^{2}\nonumber \\
 & +16\tau(2s-1)\|x_{n+1}^{-s}\partial_{n+1}u\|^{2}+16\tau\|\partial_{n+1}u\|^{2} -8 \tau \mathring{c}_{s} \|x_{n+1}^{-1}u\|^{2}\nonumber \\
 & +4\tau(1-2s)(1+2s)^{2}\|x_{n+1}^{-1-s}u\|^{2}-\tau^{3}\||x'|u\|^{2}-4\tau\|\nabla'u\|^{2}\nonumber \\
 & +8\tau\langle(\partial_{n+1}^{2}\phi)\partial_{n+1}u,u\rangle_{0}\nonumber \\
 & +4\tau\langle Su,(\partial_{n+1}\phi)u\rangle_{0}-2\langle Au,\partial_{n+1}u\rangle_{0}+2\langle\partial_{n+1}(Au),u\rangle_{0}.\label{eq:0Carl4}
\end{align}

\textbf{Step 2.1.2: Estimating the remainder.} Using integration by parts, we can estimate $R$ from below: 
\begin{align*}
R\ge & -C\epsilon\bigg[\tau\|(x_{n+1}^{1-2s}-x_{n+1})\nabla'u\|^{2}+\tau\|\partial_{n+1}u\|^{2}+\tau^{3}\|(x_{n+1}^{1-2s}-x_{n+1})x_{n+1}^{-s}u\|^{2}+\tau\|x_{n+1}^{-1}u\|^{2}\\
 & +\tau\|x_{n+1}^{\frac{1-2s}{2}}\partial_{n+1}u\|_{0}^{2}+\tau\|x_{n+1}^{\frac{2s-1}{2}}|x'|\nabla'u\|_{0}^{2}+\tau^{3}\|(x_{n+1}^{1-2s}-x_{n+1})^{\frac{1}{2}}u\|_{0}^{2}\bigg].
\end{align*}
Here we would like to highlight some features when estimating the
second term of $R$, that is, $\langle Au,(I)u\rangle$. Note that
\begin{align}
 & \langle-2\tau\partial_{n+1}\phi\partial_{n+1}u,(a_{jk}-\delta_{jk})\partial_{j}\partial_{k}u\rangle\nonumber \\
= & \tau\langle\partial_{n+1}\phi\partial_{n+1}u,(\partial_{j}a_{jk})\partial_{k}u\rangle - \boxed{\tau\langle\partial_{n+1}^{2}\phi\partial_{j}u,(a_{jk}-\delta_{jk})\partial_{k}u\rangle} \nonumber \\
 & +\tau\langle\partial_{n+1}\phi\partial_{j}u,(\partial_{k}a_{jk})\partial_{n+1}u\rangle-\tau\langle\partial_{n+1}\phi\partial_{j}u,(a_{jk}-\delta_{jk})\partial_{k}u\rangle_{0}\label{eq:harm1}
\end{align}
and 
\begin{align}
 & \langle-\tau(\partial_{n+1}^{2}\phi)u,(a_{jk}-\delta_{jk})\partial_{j}\partial_{k}u\rangle\nonumber \\
= & \tau\langle(\partial_{n+1}^{2}\phi)u,(\partial_{j}a_{jk})\partial_{k}u\rangle+\tau\langle\partial_{n+1}^{2}\phi\partial_{j}u,(a_{jk}-\delta_{jk})\partial_{k}u\rangle\label{eq:harm2-1}\\
= & -\frac{\tau}{2}\langle(\partial_{n+1}^{2}\phi)u,(\partial_{j}\partial_{k}a_{jk})u\rangle + \boxed{\tau\langle\partial_{n+1}^{2}\phi\partial_{j}u,(a_{jk}-\delta_{jk})\partial_{k}u\rangle}.\label{eq:harm2-2}
\end{align}
So, summing up \eqref{eq:harm1} and \eqref{eq:harm2-2}, we note
that the problematic term $\tau\langle\partial_{n+1}^{2}\phi\partial_{j}u,(a_{jk}-\delta_{jk})\partial_{k}u\rangle$
is canceled. It problematic because $\partial_{n+1}^{2}\phi$ has singularity $x_{n+1}^{-2s}$ for $s\in(1/2,1)$. However, when
$s=\frac{1}{2}$, $\partial_{n+1}^{2}\phi$ has no singularity. In
this case, we consider \eqref{eq:harm2-1} rather than \eqref{eq:harm2-2}.
This is the reason why we can loosen the second derivative assumption
for the case $s=\frac{1}{2}$. 

\textbf{Step 2.1.3: Combining the commutator and the remainder.} Using the Hardy inequality in Lemma \ref{lem:Hardy}, we reach 
\[
\|x_{n+1}^{-s-1}u\|^{2}\le\frac{4}{(2s+1)^{2}}\|x_{n+1}^{-s}\partial_{n+1}u\|^{2}+\frac{2}{2s+1}\|x_{n+1}^{-\frac{1}{2}-s}u\|_{0}^{2},
\]
thus 
\begin{align*}
 & 16\tau(2s-1)\|x_{n+1}^{-s}\partial_{n+1}u\|^{2}-4\tau(2s-1)(2s+1)^{2}\|x_{n+1}^{-1-s}u\|^{2}\\
\ge & -8\tau(2s-1)(2s+1)\|x_{n+1}^{-\frac{1}{2}-s}u\|_{0}^{2}.
\end{align*}
Therefore, choosing sufficiently small $\epsilon>0$, we reach
\begin{align}
\mathscr{D}\ge & 64\tau^{3}\|(x_{n+1}^{1-2s}-x_{n+1})u\|^{2}+\frac{639}{10}(2s-1)\tau^{3}\|(x_{n+1}^{1-2s}-x_{n+1})x_{n+1}^{-s}u\|^{2}\nonumber \\
 & +\frac{159}{10}\tau\|\partial_{n+1}u\|^{2}+\frac{39}{10}\tau(2s-1)(2s+1)\|x_{n+1}^{-1}u\|^{2}-4\tau\|\nabla'u\|^{2}\nonumber \\
 & -C\epsilon\tau\|(x_{n+1}^{1-2s}-x_{n+1})\nabla'u\|^{2}+8\tau\langle(\partial_{n+1}^{2}\phi)\partial_{n+1}u,u\rangle_{0}+4\tau\langle Su,\partial_{n+1}\phi\rangle_{0}\nonumber \\
 & -2\langle Au,\partial_{n+1}u\rangle_{0}+2\langle\partial_{n+1}(Au),u\rangle_{0}-\tau\|x_{n+1}^{\frac{1-2s}{2}}\partial_{n+1}u\|_{0}^{2}-\tau\|x_{n+1}^{\frac{2s-1}{2}}|x'|\nabla'u\|_{0}^{2}\nonumber \\
 & -\tau^{3}\|(x_{n+1}^{1-2s}-x_{n+1})^{\frac{1}{2}}u\|_{0}^{2}-8\tau(2s-1)(2s+1)\|x_{n+1}^{-\frac{1}{2}-s}u\|_{0}^{2}.\label{eq:0Carl5}
\end{align}

\textbf{Step 2.2: Estimating the sum $\mathscr{S}$.} Observe that  
\[
\mathscr{S} \ge 2\|Su\|^{2}+2\|Au\|^{2}-C\epsilon\bigg[\sum_{j,k=1}^{n}\|\partial_{j}\partial_{k}u\|^{2}+\tau^{2}\|\nabla'u\|^{2}+\tau^{4}\|u\|^{2}\bigg].
\]
Since $\mathring{c}_{s} < 0$, then  
\begin{align*}
2\|Su\|^{2}= & 2\| \Delta'u + (\partial_{n+1}^{2}u+\tau^{2}|\nabla\phi|^{2}u + \mathring{c}_{s} x_{n+1}^{-2}u ) \|^{2}\\
= & 2\|\Delta'u\|^{2}+4\langle\Delta'u,\partial_{n+1}^{2}u\rangle+4\tau^{2}\langle\Delta'u,|\nabla\phi|^{2}u\rangle \\
 & + 4\mathring{c}_{s} \langle\Delta'u,x_{n+1}^{-2}u\rangle +2 \| \partial_{n+1}^{2}u+\tau^{2}|\nabla\phi|^{2}u + \mathring{c}_{s} x_{n+1}^{-2}u \|^{2}\\
= & 2\sum_{j,k=1}^{n}\|\partial_{j}\partial_{k}u\|^{2}+4\langle\Delta'u,\partial_{n+1}^{2}u\rangle+4\tau^{2}\langle\Delta'u,|\nabla\phi|^{2}u\rangle\\
 & -4\mathring{c}_{s} \langle\nabla'u,x_{n+1}^{-2}\nabla'u\rangle +2 \| \partial_{n+1}^{2}u+\tau^{2}|\nabla\phi|^{2}u + \mathring{c}_{s} x_{n+1}^{-2}u \|^{2}\\
\ge & 2\sum_{j,k=1}^{n}\|\partial_{j}\partial_{k}u\|^{2}+4\langle\Delta'u,\partial_{n+1}^{2}u\rangle+4\tau^{2}\langle\Delta'u,|\nabla\phi|^{2}u\rangle.
\end{align*}
Since 
\[
4\langle\Delta'u,\partial_{n+1}^{2}u\rangle=4\langle\nabla'\partial_{n+1}u,\nabla'\partial_{n+1}u\rangle-4\langle\Delta'u,\partial_{n+1}u\rangle_{0}
\]
and for $\epsilon_{0}>0$, we have 
\begin{align*}
 & 4\tau^{2}\langle\Delta'u,|\nabla\phi|^{2}u\rangle\\
= & \tau^{2}\langle\Delta'u,|x'|^{2}u\rangle+16\tau^{2}\langle\Delta'u,(x_{n+1}^{1-2s}-x_{n+1})^{2}u\rangle\\
\ge & -\tau^{2}(1+\epsilon_{0})\|\nabla'u\|^{2}-\tau^{2}C\epsilon_{0}^{-1}\|u\|^{2}-16\tau^{2}\|(x_{n+1}^{1-2s}-x_{n+1})\nabla'u\|^{2}.
\end{align*}
Thus, 
\begin{align}
\mathscr{S}\ge & 2\|Su\|^{2}+2\|Au\|^{2}-C\epsilon\bigg[\sum_{j,k=1}^{n}\|\partial_{j}\partial_{k}u\|^{2}+\tau^{2}\|\nabla'u\|^{2}+\tau^{4}\|u\|^{2}\bigg]\nonumber \\
\ge & 2\|\nabla(\nabla'u)\|^{2}-\tau^{2}(1+\epsilon_{0})\|\nabla'u\|^{2}-\tau^{2}C\epsilon_{0}^{-1}\|u\|^{2}-16\tau^{2}\|(x_{n+1}^{1-2s}-x_{n+1})\nabla'u\|^{2}\nonumber \\
 & -C\epsilon\bigg[\sum_{j,k=1}^{n}\|\partial_{j}\partial_{k}u\|^{2}+\tau^{2}\|\nabla'u\|^{2}+\tau^{4}\|u\|^{2}\bigg]-4\langle\Delta'u,\partial_{n+1}u\rangle_{0}.\label{eq:0Carl6}
\end{align}

\textbf{Step 2.3: Combining the difference $\mathscr{D}$ and the sum $\mathscr{S}$.} After combining \eqref{eq:0Carl5} and \eqref{eq:0Carl6}, we choose small $\epsilon>0$, and consequently choose small $\epsilon_{0}>0$ and large $\tau$, hence 
\begin{align}
 & \bigg(\tau+s+\frac{1}{2}\bigg)\|L^{+}u\|^{2}\nonumber \\
\ge & \frac{9}{10}(2s-1)\|\nabla(\nabla'u)\|^{2}+\|Su\|^{2}+64\tau^{4}\|(x_{n+1}^{1-2s}-x_{n+1})u\|^{2}\nonumber \\
 & +\frac{639}{10}(2s-1)\tau^{4}\|(x_{n+1}^{1-2s}-x_{n+1})x_{n+1}^{-s}u\|^{2}+\frac{159}{10}\tau^{2}\|\partial_{n+1}u\|^{2}-4\tau^{2}\|\nabla'u\|^{2}\nonumber \\
 & -\frac{171}{20}(2s-1)\tau^{2}\|(x_{n+1}^{1-2s}-x_{n+1})\nabla'u\|^{2}+\frac{39}{10}\tau^{2}(2s-1)(2s+1)\|x_{n+1}^{-1}u\|^{2}\nonumber \\
 & +8\tau^{2}\langle(\partial_{n+1}^{2}\phi)\partial_{n+1}u,u\rangle_{0}+4\tau^{2}\langle Su,\partial_{n+1}\phi\rangle_{0}-2\tau\langle Au,\partial_{n+1}u\rangle_{0}+2\tau\langle\partial_{n+1}(Au),u\rangle_{0}\nonumber \\
 & -\tau^{2}\|x_{n+1}^{\frac{1-2s}{2}}\partial_{n+1}u\|_{0}^{2}-\tau^{2}\|x_{n+1}^{\frac{2s-1}{2}}|x'|\nabla'u\|_{0}^{2}-\tau^{4}\|(x_{n+1}^{1-2s}-x_{n+1})^{\frac{1}{2}}u\|_{0}^{2}\nonumber \\
 & -8\tau^{2}(2s-1)(2s+1)\|x_{n+1}^{-\frac{1}{2}-s}u\|_{0}^{2}-2(2s-1)\langle\Delta'u,\partial_{n+1}u\rangle_{0}.\label{eq:0Carl7}
\end{align}
\textbf{Step 2.4: Obtaining gradient estimates.} Since ${\rm supp}(u)\subset B_{1/2}^{+}$ and $s>\frac{1}{2}$, thus
\[
0\le(x_{n+1}^{1-2s}-x_{n+1})x_{n+1}^{s}=x_{n+1}^{1-s}-x_{n+1}^{1+s} \le x_{n+1}^{1-s} \le 1,
\]
and hence 
\begin{align*}
 & \frac{172}{20}(2s-1)\tau^{2}\|(x_{n+1}^{1-2s}-x_{n+1})\nabla'u\|^{2}\\
= & -\frac{172}{20}(2s-1)\tau^{2}\langle(x_{n+1}^{1-2s}-x_{n+1})x^{s}\Delta'u,(x_{n+1}^{1-2s}-x_{n+1})x^{-s}u\rangle\\
\le & \frac{86}{20}(2s-1)\delta\|(x_{n+1}^{1-2s}-x_{n+1})x^{s}\Delta'u\|^{2}+\frac{86}{20}(2s-1)\tau^{4}\delta^{-1}\|(x_{n+1}^{1-2s}-x_{n+1})x^{-s}u\|^{2}\\
\le & \frac{86}{20}(2s-1)\delta\|\Delta'u\|^{2}+\frac{86}{20}(2s-1)\tau^{4}\delta^{-1}\|(x_{n+1}^{1-2s}-x_{n+1})x^{-s}u\|^{2}.
\end{align*}
Choose $\delta=\frac{8}{43}$, we reach 
\begin{align}
 & \frac{172}{20}(2s-1)\tau^{2}\|(x_{n+1}^{1-2s}-x_{n+1})\nabla'u\|^{2}\nonumber \\
\le & \frac{8}{10}(2s-1)\|\Delta'u\|^{2}+23.1125(2s-1)\tau^{4}\|(x_{n+1}^{1-2s}-x_{n+1})x^{-s}u\|^{2}.\label{eq:0Carl8}
\end{align}
Moreover, we have 
\begin{align}
 & \frac{41}{10}\tau^{2}\langle Su,u\rangle \nonumber \\
= & \frac{41}{10}\tau^{2}\|\nabla u\|^{2}-\frac{41}{10}\tau^{4}\||\nabla\phi|u\|^{2} + \frac{41}{5}(2s+1)(2s-1)\tau\|x_{n+1}^{-1}u\|^{2} +\frac{41}{10}\tau^{2}\langle\partial_{n+1}u,u\rangle_{0} \nonumber \\
\ge & \frac{41}{10}\tau^{2}\|\nabla u\|^{2}-\frac{41}{10}\tau^{4}\bigg(\frac{1}{16}\|u\|^{2}+4\|(x_{n+1}^{1-2s}-x_{n+1})u\|^{2}\bigg)  \nonumber \\
 & + \frac{41}{5}(2s+1)(2s-1)\tau\|x_{n+1}^{-1}u\|^{2} + \frac{41}{10}\tau^{2}\langle\partial_{n+1}u,u\rangle_{0} \nonumber \\
= & \frac{41}{10}\tau^{2} \| \nabla u \|^{2} - \frac{41}{160} \tau^{4} \|u\|^{2} - \frac{164}{10} \| (x_{n+1}^{1-2s} - x_{n+1}) u \|^{2} + \frac{41}{5}(2s+1)(2s-1)\tau\|x_{n+1}^{-1}u\|^{2} \nonumber \\
 & + \frac{41}{5}(2s+1)(2s-1)\tau\|x_{n+1}^{-1}u\|^{2} + \frac{41}{10}\tau^{2}\langle\partial_{n+1}u,u\rangle_{0} \label{eq:rev1-1} 
\end{align}
Define $\psi_{s}(x_{n+1}) := x_{n+1}^{1-2s} - x_{n+1}$. Since ${\rm supp}\,(u) \subset B_{1/2}^{+}$, so $0 \le x_{n+1} \le 1/2$, for $s \in (1/2,1)$, the derivative can be easily estimated 
\[
\psi_{s}'(x_{n+1}) = (1-2s)x_{n+1}^{-2s} - 1 < 0 \quad \text{for } 0 \le x_{n+1} \le 1/2.
\]
Since $\psi_{s}(x_{n+1})$ is decreasing on $[0,1/2]$, for $s \in (1/2,1)$, 
\[
\inf_{0 \le x_{n+1} \le 1/2} (x_{n+1}^{1-2s} - x_{n+1}) = \inf_{0 \le x_{n+1} \le 1/2} \psi_{s}(x_{n+1}) = \psi_{s}\bigg( \frac{1}{2} \bigg) = \frac{1}{2} (4^{s} - 1) \ge \frac{1}{2}.
\]
Combining this with \eqref{eq:rev1-1}, we reach the estimate 
\begin{align*}
 & \frac{41}{10}\tau^{2}\|\nabla'u\|^{2}+\frac{41}{5}(2s+1)(2s-1)\tau^{2}\|x_{n+1}^{-1}u\|^{2}+\frac{41}{10}\tau^{2}\langle\partial_{n+1}u,u\rangle_{0}\\
\le & \frac{41}{10}\tau^{2}\|\nabla u\|^{2}+\frac{41}{5}(2s+1)(2s-1)\tau^{2}\|x_{n+1}^{-1}u\|^{2}+\frac{41}{10}\tau^{2}\langle\partial_{n+1}u,u\rangle_{0}\\
\le & \frac{41}{10}\tau^{2}\langle Su,u\rangle+\frac{41}{160}\tau^{4}\|u\|^{2}+\frac{164}{10}\tau^{4}\|(x_{n+1}^{1-2s}-x_{n+1})u\|^{2}\\
\le & \frac{41}{20}\delta\|Su\|^{2}+\frac{41}{20}\delta^{-1}\tau^{4}\|u\|^{2}+\frac{41}{160}\tau^{4}\|u\|^{2}+\frac{164}{10}\tau^{4}\|(x_{n+1}^{1-2s}-x_{n+1})u\|^{2}\\
\le & \frac{41}{20}\delta\|Su\|^{2}+\frac{82}{10}\delta^{-1}\tau^{4}\|(x_{n+1}^{1-2s}-x_{n+1})u\|^{2}+\frac{41}{40}\tau^{4}\|(x_{n+1}^{1-2s}-x_{n+1})u\|^{2}\\
 & +\frac{164}{10}\tau^{4}\|(x_{n+1}^{1-2s}-x_{n+1})u\|^{2}.
\end{align*}
Choosing $\delta=\frac{20}{41}$, hence 
\begin{align}
 & \frac{41}{10}\tau^{2}\|\nabla'u\|^{2}+\frac{41}{5}(2s+1)(2s-1)\tau^{2}\|x_{n+1}^{-1}u\|^{2}+\frac{41}{10}\tau^{2}\langle\partial_{n+1}u,u\rangle_{0}\nonumber \\
\le & \|Su\|^{2}+34.235\tau^{4}\|(x_{n+1}^{1-2s}-x_{n+1})u\|^{2}.\label{eq:0Carl9}
\end{align}

\textbf{Step 2.5: Plugging gradient estimates into \eqref{eq:0Carl7}.} Combining \eqref{eq:0Carl7}, \eqref{eq:0Carl8} and \eqref{eq:0Carl9},
we reach 
\begin{align}
 & \bigg(\tau+s+\frac{1}{2}\bigg)\|L^{+}u\|^{2}\nonumber \\
\ge & \frac{1}{10}(2s-1)\|\nabla(\nabla'u)\|^{2}+29.765\tau^{4}\|(x_{n+1}^{1-2s}-x_{n+1})u\|^{2}\nonumber \\
 & +40.7875(2s-1)\tau^{4}\|(x_{n+1}^{1-2s}-x_{n+1})x_{n+1}^{-s}u\|^{2}+\frac{159}{10}\tau^{2}\|\partial_{n+1}u\|^{2}+\frac{1}{10}\tau^{2}\|\nabla'u\|^{2}\nonumber \\
 & +\frac{1}{20}(2s-1)\tau^{2}\|(x_{n+1}^{1-2s}-x_{n+1})\nabla'u\|^{2}+12.1\tau^{2}(2s-1)(2s+1)\|x_{n+1}^{-1}u\|^{2}\nonumber \\
 & +8\tau^{2}\langle(\partial_{n+1}^{2}\phi)\partial_{n+1}u,u\rangle_{0}+4\tau^{2}\langle Su,(\partial_{n+1}\phi)u\rangle_{0}-2\tau\langle Au,\partial_{n+1}u\rangle_{0}+2\tau\langle\partial_{n+1}(Au),u\rangle_{0}\nonumber \\
 & -\tau^{2}\|x_{n+1}^{\frac{1-2s}{2}}\partial_{n+1}u\|_{0}^{2}-\tau^{2}\|x_{n+1}^{\frac{2s-1}{2}}|x'|\nabla'u\|_{0}^{2}-\tau^{4}\|(x_{n+1}^{1-2s}-x_{n+1})^{\frac{1}{2}}u\|_{0}^{2}\nonumber \\
 & -8\tau^{2}(2s-1)(2s+1)\|x_{n+1}^{-\frac{1}{2}-s}u\|_{0}^{2}-4(2s-1)\langle\Delta'u,\partial_{n+1}u\rangle_{0}+\frac{41}{10}\tau^{2}\langle\partial_{n+1}u,u\rangle_{0}.\label{eq:0Carl10}
\end{align}
Hence, we reach 
\begin{align}
 & 2\tau\|L^{+}u\|^{2}\nonumber \\
\ge & 25\tau^{4}\|(x_{n+1}^{1-2s}-x_{n+1})u\|^{2}+\frac{1}{10}\tau^{2}\|\nabla u\|^{2}+12\tau^{2}(2s-1)(2s+1)\|x_{n+1}^{-1}u\|^{2}\nonumber \\
 & +8\tau^{2}\langle(\partial_{n+1}^{2}\phi)\partial_{n+1}u,u\rangle_{0}+4\tau^{2}\langle Su,(\partial_{n+1}\phi)u\rangle_{0}-2\tau\langle Au,\partial_{n+1}u\rangle_{0}+2\tau\langle\partial_{n+1}(Au),u\rangle_{0}\nonumber \\
 & -\tau^{2}\|x_{n+1}^{\frac{1-2s}{2}}\partial_{n+1}u\|_{0}^{2}-\tau^{2}\|x_{n+1}^{\frac{2s-1}{2}}|x'|\nabla'u\|_{0}^{2}-\tau^{4}\|(x_{n+1}^{1-2s}-x_{n+1})^{\frac{1}{2}}u\|_{0}^{2}\nonumber \\
 & -8\tau^{2}(2s-1)(2s+1)\|x_{n+1}^{-\frac{1}{2}-s}u\|_{0}^{2}-4(2s-1)\langle\Delta'u,\partial_{n+1}u\rangle_{0}+\frac{41}{10}\tau^{2}\langle\partial_{n+1}u,u\rangle_{0}.\label{eq:0Carl11}
\end{align}
Since $u=e^{\tau\phi}x_{n+1}^{\frac{1-2s}{2}}w$, we estimate that
\begin{align*}
\|\nabla u\|^{2}\ge & \frac{1}{2}\|e^{\tau\phi}x_{n+1}^{\frac{1-2s}{2}}\nabla w\|^{2}-2\tau^{2}\|e^{\tau\phi}|\nabla\phi|x_{n+1}^{\frac{1-2s}{2}}w\|^{2}-2\bigg(\frac{2s-1}{2}\bigg)^{2}\|e^{\tau\phi}x_{n+1}^{-\frac{1+2s}{2}}w\|^{2}\\
\ge & \frac{1}{2}\|e^{\tau\phi}x_{n+1}^{\frac{1-2s}{2}}\nabla w\|^{2}-16\tau^{2}\|(x_{n+1}^{1-2s}-x_{n+1})u\|^{2}-(2s-1)^{2}\|x_{n+1}^{-1}u\|^{2}.
\end{align*}

\textbf{Step 3: Estimating the boundary contributions.} We want to show that 
\begin{equation}
\|e^{\tau\phi}x_{n+1}^{-2s}w\|_{0}\le C_{s}\|e^{\tau\phi}x_{n+1}^{1-2s}\partial_{n+1}w\|_{0}<\infty.\label{eq:0Carl12}
\end{equation}
Indeed, since $w(x',0)\equiv0$, thus 
\[
x_{n+1}^{-2s}w(x',x_{n+1})=x_{n+1}^{1-2s}\int_{0}^{1}\partial_{n+1}w(x',tx_{n+1})\,dt=\int_{0}^{1}(tx_{n+1})^{1-2s}\partial_{n+1}w(x',tx_{n+1})t^{2s-1}\,dt.
\]
Multiplying above equation by $e^{\tau\phi}$, taking the $L^{2}$-norm with respect
to $x'$ and using the fact that $\partial_{n+1}\phi<0$ on ${\rm supp}(w)$
gives 
\[
\|e^{\tau\phi}x_{n+1}^{-2s}w(\bullet,x_{n+1})\|_{0}\le\sup_{t\in(0,1)}\|e^{\tau\phi(\bullet,tx_{n+1})}(tx_{n+1})^{1-2s}\partial_{n+1}w(\bullet,tx_{n+1})\|_{0}\int_{0}^{1}t^{2s-1}\,dt.
\]
Taking $x_{n+1}\rightarrow0$ proves \eqref{eq:0Carl12}. 

We observe that 
\begin{align*}
 & 4\tau^{2}\langle Su,(\partial_{n+1}\phi)u\rangle_{0}-2\tau\langle Au,\partial_{n+1}u\rangle_{0}+2\tau\langle\partial_{n+1}(Au),u\rangle_{0}\\
= & 8\tau^{2}\langle\partial_{n+1}u,\nabla'\phi\cdot\nabla'u\rangle_{0}+4\tau^{2}\langle(\partial_{n+1}u)^{2},\partial_{n+1}\phi\rangle_{0}-4\tau^{2}\langle(\partial_{n+1}\phi),|\nabla'u|^{2}\rangle_{0}\\
 & +4\tau^{2}\langle(\Delta'\phi-\partial_{n+1}^{2}\phi)u,\partial_{n+1}u\rangle_{0}-2\tau^{2}\langle(\partial_{n+1}^{3}\phi)u,u\rangle_{0}+4\tau^{4}\langle(\partial_{n+1}\phi)|\nabla\phi|^{2}u,u\rangle_{0}\\
 & -\tau^{2}(2s+1)(2s-1)\langle x_{n+1}^{-2}u,(\partial_{n+1}\phi)u\rangle_{0}\\
\ge & 8\tau^{2}\langle\partial_{n+1}u,\nabla'\phi\cdot\nabla'u\rangle_{0}+4\tau^{2}\langle(\partial_{n+1}u)^{2},\partial_{n+1}\phi\rangle_{0}+4\tau^{2}\langle(\Delta'\phi-\partial_{n+1}^{2}\phi)u,\partial_{n+1}u\rangle_{0}\\
 & +4\tau^{4}\langle(\partial_{n+1}\phi)|\nabla\phi|^{2}u,u\rangle_{0}.
\end{align*}
Note that \eqref{eq:0Carl12} imply 
\begin{align*}
\partial_{n+1}u & =e^{\tau\phi}\bigg(x_{n+1}^{\frac{1-2s}{2}}\partial_{n+1}w-\frac{2s-1}{2}x_{n+1}^{-\frac{1+2s}{2}}w\bigg)+x_{n+1}^{\frac{3-2s}{2}}R\\
\nabla'u & =e^{\tau\phi}x_{n+1}^{\frac{1-2s}{2}}\nabla'w+x_{n+1}^{s+\frac{1}{2}}R',
\end{align*}
where $\|R\|_{0}\le C\tau$ and $\|R'\|_{0}\le C\tau$. 

Hence, 
\begin{align*}
 & |\langle\partial_{n+1}u,\nabla'\phi\cdot\nabla'u\rangle_{0}|\\
= & \bigg|\bigg\langle e^{\tau\phi}\bigg(x_{n+1}^{\frac{1-2s}{2}}\partial_{n+1}w-\frac{2s-1}{2}x_{n+1}^{-\frac{1+2s}{2}}w\bigg),e^{\tau\phi}x_{n+1}^{\frac{1-2s}{2}}\nabla'\phi\cdot\nabla'w\bigg\rangle_{0}\bigg|\\
= & \bigg|\bigg\langle e^{\tau\phi}\bigg(x_{n+1}^{1-2s}\partial_{n+1}w-\frac{2s-1}{2}x_{n+1}^{-2}w\bigg),\frac{1}{2}e^{\tau\phi}x'\cdot\nabla'w\bigg\rangle_{0}\bigg|\\
\le & \frac{1}{2}|\langle e^{\tau\phi}x_{n+1}^{1-2s}\partial_{n+1}w,e^{\tau\phi}x'\cdot\nabla'w\rangle_{0}|+\frac{2s-1}{4}|\langle e^{\tau\phi}x_{n+1}^{-2}w,e^{\tau\phi}x'\cdot\nabla'w\rangle_{0}|.
\end{align*}
Using \eqref{eq:0Carl12}, we reach 
\[
|\langle\partial_{n+1}u,\nabla'\phi\cdot\nabla'u\rangle_{0}|\le\|e^{\tau\phi}x_{n+1}^{1-2s}\partial_{n+1}w\|_{0}\|e^{\tau\phi}x'\cdot\nabla'w\|_{0}.
\]
Similarly, using \eqref{eq:0Carl12}, we have 
\begin{align*}
|\langle(\partial_{n+1}u)^{2},\partial_{n+1}\phi\rangle_{0}|+|\langle(\Delta'\phi-\partial_{n+1}^{2}\phi)u,\partial_{n+1}u\rangle_{0}| & \le C\|e^{\tau\phi}x_{n+1}^{1-2s}\partial_{n+1}w\|_{0}^{2}\\
|\langle(\partial_{n+1}\phi)|\nabla\phi|^{2}u,u\rangle_{0}| & \le C\|e^{\tau\phi}x_{n+1}^{2-4s}w\|_{0}^{2}\rightarrow0.
\end{align*}
Also, 
\begin{align*}
|\langle(\partial_{n+1}^{2}\phi)\partial_{n+1}u,u\rangle_{0}| & \le C\|e^{\tau\phi}x_{n+1}^{1-2s}\partial_{n+1}w\|_{0}^{2}\\
\|x_{n+1}^{\frac{1-2s}{2}}\partial_{n+1}u\|_{0}^{2} & =\bigg\| e^{\tau\phi}x_{n+1}^{1-2s}\partial_{n+1}w-\frac{2s-1}{2}e^{\tau\phi}x_{n+1}^{-2s}w\bigg\|_{0}^{2}\le C\|e^{\tau\phi}x_{n+1}^{1-2s}\partial_{n+1}w\|_{0}^{2}\\
\|x_{n+1}^{\frac{2s-1}{2}}|x'|\nabla'u\|_{0}^{2} & =\|e^{\tau\phi}|x'|\nabla'w\|_{0}^{2}\\
\|(x_{n+1}^{1-2s}-x_{n+1})^{\frac{1}{2}}u\|_{0}^{2} & \rightarrow0\\
\|x_{n+1}^{-\frac{1}{2}-s}u\|_{0}^{2} & =\|e^{\tau\phi}x_{n+1}^{-2s}w\|_{0}^{2}\le C\|e^{\tau\phi}x_{n+1}^{1-2s}\partial_{n+1}w\|_{0}^{2}\\
|\langle\partial_{n+1}u,u\rangle_{0}| & \rightarrow0.
\end{align*}
Finally, we also have 
\begin{align*}
|\langle\Delta'u,\partial_{n+1}u\rangle_{0}|\le & \|x_{n+1}^{\frac{2s-1}{2}}\Delta'u\|_{0}^{2}+\|x_{n+1}^{\frac{1-2s}{2}}\partial_{n+1}u\|_{0}^{2}\\
= & \bigg\|-\frac{n\tau}{2}e^{\tau\phi}w+\frac{\tau^{2}}{4}|x'|^{2}e^{\tau\phi}w-\tau e^{\tau\phi}x'\cdot\nabla'w+e^{\tau\phi}\Delta'w\bigg\|_{0}^{2}+\|x_{n+1}^{\frac{1-2s}{2}}\partial_{n+1}u\|_{0}^{2}\\
\le & C\|e^{\tau\phi}\Delta'w\|_{0}^{2}+C\tau^{2}\|e^{\tau\phi}x'\cdot\nabla'w\|_{0}^{2}+C\|e^{\tau\phi}x_{n+1}^{1-2s}\partial_{n+1}w\|_{0}^{2}.
\end{align*}

\textbf{Step 4: Conclusion.} Put them together, we reach 
\begin{align*}
 & \tau^{3}\|u\|^{2}+\tau\|e^{\tau\phi}x_{n+1}^{\frac{1-2s}{2}}\nabla w\|^{2}\\
\le & C\bigg(\|L^{+}u\|^{2}+\tau^{-1}\|e^{\tau\phi}\Delta'w\|_{0}^{2}+\tau\|e^{\tau\phi}x'\cdot\nabla'w\|_{0}^{2}+\tau\|e^{\tau\phi}x_{n+1}^{1-2s}\partial_{n+1}w\|_{0}^{2}\bigg),
\end{align*}
which is our desired result. 
\end{proof}
As in \cite{RS20Calderon}, we introduce the following sets for $s\in[\frac{1}{2},1)$:
\begin{align*}
C_{s,r}^{+} & :=\bigg\{(x',x_{n+1})\in\mathbb{R}_{+}^{n+1}:x_{n+1}\le\bigg[(1-s)\bigg(r-\frac{|x'|^{2}}{4}\bigg)\bigg]^{\frac{1}{2-2s}}\bigg\}\\
C_{s,r}' & := \bigg\{(x',0)\in \mathbb{R}^{n} \times \{0\} : 0\le\bigg[(1-s)\bigg(r-\frac{|x'|^{2}}{4}\bigg)\bigg]^{\frac{1}{2-2s}}\bigg\}.
\end{align*}
With this notation, we infer the following analogous to \cite[Proposition 5.10]{RS20Calderon}: 
\begin{lem}
Let $s\in[\frac{1}{2},1)$. Suppose that $\tilde{w}\in H^{1}(\mathbb{R}_{+}^{n+1},x_{n+1}^{1-2s})$
is a solution to 
\begin{align*}
\bigg[\partial_{n+1}x_{n+1}^{1-2s}\partial_{n+1}+x_{n+1}^{1-2s}\sum_{j,k=1}^{n}\partial_{j}a_{jk}\partial_{k}\bigg]\tilde{w} & =0\quad\text{in }\mathbb{R}_{+}^{n+1},\\
\tilde{w} & =w\quad\text{on }\mathbb{R}^{n}\times\{0\},
\end{align*}
with $w=0$ on $B_{1}'$. We assume that 
\[
\max_{1\le j,k\le n}\|a_{jk}-\delta_{jk}\|_{\infty}+\max_{1\le j,k\le n}\|\nabla'a_{jk}\|_{\infty}\le\epsilon
\]
for some sufficiently small $\epsilon>0$. For $s\neq\frac{1}{2}$,
we further assume 
\[
\max_{1\le j,k\le n}\|(\nabla')^{2}a_{jk}\|_{\infty}\le C
\]
for some positive constant $C$. Then there exists $\alpha=\alpha(n,s)\in(0,1)$
such that 
\[
\|x_{n+1}^{\frac{1-2s}{2}}\tilde{w}\|_{L^{2}(C_{s,1/8}^{+})}\le C\|x_{n+1}^{\frac{1-2s}{2}}\tilde{w}\|_{L^{2}(C_{s,1/2}^{+})}^{\alpha}\cdot\lim_{x_{n+1}\rightarrow0}\|x_{n+1}^{1-2s}\partial_{n+1}\tilde{w}\|_{L^{2}(C_{s,1/2}')}^{1-\alpha}.
\]
\end{lem}

\begin{proof}
We may assume that $\|x_{n+1}^{\frac{1-2s}{2}}\tilde{w}\|_{L^{2}(C_{s,1/2}^{+})}>0$
and 
\[
\|x_{n+1}^{\frac{1-2s}{2}}\tilde{w}\|_{L^{2}(C_{s,1/2}^{+})}\ge c_{0}\lim_{x_{n+1}\rightarrow0}\|x_{n+1}^{1-2s}\partial_{n+1}\tilde{w}\|_{L^{2}(C_{s,1/2}')}^{1-\alpha}
\]
for some sufficiently large constant $c_{0}>0$. Otherwise the result
is trivial. 

Let $\eta$ is a smooth cut-off function satisfies 
\[
\eta(x)=\begin{cases}
1 & \text{in }C_{s,3/16}^{+},\\
0 & \text{in }\mathbb{R}_{+}^{n+1}\setminus C_{s,1/4}^{+},
\end{cases}
\]
and $|\partial_{n+1}\eta|\le Cx_{n+1}$ in $\mathbb{R}_{+}^{n+1}$
with $\partial_{n+1}\eta=0$ on $\mathbb{R}^{n}\times\{0\}$. Define
$\overline{w}=\eta\tilde{w}$. Note that $\overline{w}$ satisfies
${\rm supp}(\overline{w})\subset B_{1/2}^{+}$ and it solves 
\begin{align*}
\bigg[\partial_{n+1}x_{n+1}^{1-2s}\partial_{n+1}+x_{n+1}^{1-2s}\sum_{j,k=1}^{n}a_{jk}\partial_{j}\partial_{k}\bigg]\overline{w} & =f\quad\text{in }\mathbb{R}_{+}^{n+1},\\
\overline{w} & =0\quad\text{on }\mathbb{R}^{n}\times\{0\},
\end{align*}
where 
\begin{align*}
f= & \partial_{n+1}(x_{n+1}^{1-2s}\partial_{n+1}\eta)\tilde{w}+x_{n+1}^{1-2s}\sum_{j,k=1}^{n}\partial_{j}(a_{jk}\partial_{k}\eta)\tilde{w}\\
 & +2x_{n+1}^{1-2s}\partial_{n+1}\eta\partial_{n+1}\tilde{w}+x_{n+1}^{1-2s}\sum_{j,k=1}^{n}a_{jk}\partial_{k}\eta\partial_{j}\tilde{w}+x_{n+1}^{1-2s}\sum_{j,k=1}^{n}a_{jk}\partial_{j}\eta\partial_{k}\tilde{w}\\
 & -x_{n+1}^{1-2s}\sum_{j,k=1}^{n}(\partial_{j}a_{jk})\partial_{k}\overline{w}.
\end{align*}
Since $\eta$ and $\nabla\eta$ are bounded, together with $|\partial_{n+1}\eta|\le Cx_{n+1}$,
we know that 
\[
\|x_{n+1}^{\frac{2s-1}{2}}f\|_{L^{2}(\mathbb{R}_{+}^{n+1})}\le C(\|x_{n+1}^{\frac{1-2s}{2}}\tilde{w}\|_{L^{2}(C_{s,1/4}^{+})}+\|x_{n+1}^{\frac{1-2s}{2}}\nabla\tilde{w}\|_{L^{2}(C_{s,1/4}^{+})})<\infty.
\]
Moreover, since $w|_{B_{1}'}=0$ and ${\rm supp}(\eta)\subset B_{1}'$
on $\mathbb{R}^{n}\times\{0\}$, then 
\[
\lim_{x_{n+1}\rightarrow0}\nabla'\overline{w}=0,\quad\lim_{x_{n+1}\rightarrow0}\Delta'\overline{w}=0\quad\text{and also}\quad\lim_{x_{n+1}\rightarrow0}x_{n+1}^{1-2s}\partial_{n+1}\overline{w}=\eta\lim_{x_{n+1}\rightarrow0}x_{n+1}^{1-2s}\partial_{n+1}\tilde{w}.
\]
So, by the Carleman estimate in Lemma \ref{lem:0Carl}, there exists $\tau_{0}>1$ such that 
\begin{align*}
 & \tau^{3}\|e^{\tau\phi}x_{n+1}^{\frac{1-2s}{2}}\overline{w}\|_{L^{2}(\mathbb{R}_{+}^{n+1})}^{2}+\tau\|e^{\tau\phi}x_{n+1}^{\frac{1-2s}{2}}\nabla\overline{w}\|_{L^{2}(\mathbb{R}_{+}^{n+1})}^{2}\\
\le & C(\|e^{\tau\phi}x_{n+1}^{\frac{2s-1}{2}}f\|_{L^{2}(\mathbb{R}_{+}^{n+1})}^{2}+\tau\lim_{x_{n+1}\rightarrow0}\|e^{\tau\phi}x_{n+1}^{1-2s}\partial_{n+1}\overline{w}\|_{L^{2}(\mathbb{R}^{n}\times\{0\})}^{2})
\end{align*}
for all $\tau\ge\tau_{0}$. Then, for large $\tau_{0}$, the last term of $f$ was absorbed by the gradient term in the left-hand-side, so we have 
\[
\tau^{3}\|e^{\tau\phi}x_{n+1}^{\frac{1-2s}{2}}\overline{w}\|_{L^{2}(\mathbb{R}_{+}^{n+1})}^{2}\le C(\|e^{\tau\phi}x_{n+1}^{\frac{2s-1}{2}}g\|_{L^{2}(\mathbb{R}_{+}^{n+1})}^{2}+\tau\lim_{x_{n+1}\rightarrow0}\|e^{\tau\phi}x_{n+1}^{1-2s}\partial_{n+1}\overline{w}\|_{L^{2}(\mathbb{R}^{n}\times\{0\})}^{2}),
\]
where $g=f+x_{n+1}^{1-2s}\sum_{j,k=1}^{n}(\partial_{j}a_{jk})\partial_{k}\overline{w}$. 

Let 
\[
\phi_{-}:=\inf_{x\in C_{s,1/8}^{+}}\phi(x)\quad\text{and}\quad\phi_{+}:=\sup_{x\in C_{s,1/4}^{+}\setminus C_{s,3/16}^{+}}\phi(x).
\]
Hence, 
\begin{align*}
 & \tau^{3}e^{2\tau\phi_{-}}\|x_{n+1}^{\frac{1-2s}{2}}\tilde{w}\|_{L^{2}(C_{s,1/8}^{+})}^{2}\\
\le & C\bigg[e^{2\tau\phi_{+}}\|x_{n+1}^{\frac{2s-1}{2}}g\|_{L^{2}(C_{s,1/4}^{+}\setminus C_{s,3/16}^{+})}^{2}+\tau\lim_{x_{n+1}\rightarrow0}\|x_{n+1}^{1-2s}\partial_{n+1}\tilde{w}\|_{L^{2}(C_{s,1/4}')}^{2}\bigg].
\end{align*}
Dividing above equation by $\tau$, since $\tau\ge1$ and applying Caccioppoli's inequality
(Lemma \ref{lem:Caccio}), we obtain 
\[
\|x_{n+1}^{\frac{1-2s}{2}}\tilde{w}\|_{L^{2}(C_{s,1/8}^{+})}\le C\bigg[e^{\tau(\phi_{+}-\phi_{-})}\|x_{n+1}^{\frac{1-2s}{2}}\tilde{w}\|_{L^{2}(C_{s,1/2}^{+})}+e^{-\tau\phi_{-}}\lim_{x_{n+1}\rightarrow0}\|x_{n+1}^{1-2s}\partial_{n+1}\tilde{w}\|_{L^{2}(C_{s,1/2}')}\bigg].
\]

Observe that 
\[
-\frac{|x'|^{2}}{4}\ge\frac{1}{1-s}x_{n+1}^{2-2s}-\frac{1}{8}\quad\text{in }C_{s,1/8}^{+},
\]
and also since $s\ge\frac{1}{2}$, 
\[
x_{n+1}^{2}\le\bigg[(1-s)\bigg(\frac{1}{4}-\frac{|x'|^{2}}{4}\bigg)\bigg]^{\frac{1}{1-s}}\le\frac{1}{8^{\frac{1}{1-s}}}\le\frac{1}{64}
\]
and 
\[
-\frac{|x'|^{2}}{4}\le\frac{1}{1-s}x_{n+1}^{2-2s}-\frac{3}{16}\quad\text{in }C_{s,1/4}^{+}\setminus C_{s,3/16}^{+},
\]
so $\phi_{-}\ge-\frac{1}{8}$ and $\phi_{+}\le-\frac{11}{64}$, that is, $\phi_{+}-\phi_{-}\le-\frac{19}{64}<0$. So, we can choose $\tau$ (which is large) to satisfy 
\[
e^{\tau(\phi_{+}-\phi_{-})}=\frac{\lim_{x_{n+1}\rightarrow0}\|x_{n+1}^{1-2s}\partial_{n+1}\tilde{w}\|_{L^{2}(C_{s,1/2}')}^{1-\alpha}}{\|x_{n+1}^{\frac{1-2s}{2}}\tilde{w}\|_{L^{2}(C_{s,1/2}^{+})}^{1-\alpha}}\le\frac{1}{c_{0}}
\]
for large $c_{0}$, where $\alpha\in(0,1)$ will be chosen later.
Note that 
\[
e^{-\tau\phi_{-}}=\frac{\|x_{n+1}^{\frac{1-2s}{2}}\tilde{w}\|_{L^{2}(C_{s,1/2}^{+})}^{\frac{\phi_{-}}{\phi_{+}-\phi_{-}}(1-\alpha)}}{\lim_{x_{n+1}\rightarrow0}\|x_{n+1}^{1-2s}\partial_{n+1}\tilde{w}\|_{L^{2}(C_{s,1/2}')}^{\frac{\phi_{-}}{\phi_{+}-\phi_{-}}(1-\alpha)}}.
\]
Finally, choosing $\alpha\in(0,1)$ satisfies $\alpha=\frac{\phi_{-}}{\phi_{+}-\phi_{-}}(1-\alpha)$
will implies our desired result. 
\end{proof}
For our purpose, we only need the following simplified version of
the Lemma above: 
\begin{cor}
\label{cor:interIII}Let $s\in[\frac{1}{2},1)$. Suppose that $\tilde{w}\in H^{1}(\mathbb{R}_{+}^{n+1},x_{n+1}^{1-2s})$
is a solution to 
\begin{align*}
\bigg[\partial_{n+1}x_{n+1}^{1-2s}\partial_{n+1}+x_{n+1}^{1-2s}\sum_{j,k=1}^{n}\partial_{j}a_{jk}\partial_{k}\bigg]\tilde{w} & =0\quad\text{in }\mathbb{R}_{+}^{n+1},\\
\tilde{w} & =w\quad\text{on }\mathbb{R}^{n}\times\{0\},
\end{align*}
with $w=0$ on $B_{1}'$. We assume that 
\[
\max_{1\le j,k\le n}\|a_{jk}-\delta_{jk}\|_{\infty}+\max_{1\le j,k\le n}\|\nabla'a_{jk}\|_{\infty}\le\epsilon
\]
for some sufficiently small $\epsilon>0$. For $s\neq\frac{1}{2}$,
we further assume 
\[
\max_{1\le j,k\le n}\|(\nabla')^{2}a_{jk}\|_{\infty}\le C
\]
for some positive constant $C$. Then there exist $\alpha=\alpha(n,s)\in(0,1)$,
$c=c(n,s)\in(0,1)$ and a constant $C$ such that 
\[
\|x_{n+1}^{\frac{1-2s}{2}}\tilde{w}\|_{L^{2}(B_{c}^{+})}\le C\|x_{n+1}^{\frac{1-2s}{2}}\tilde{w}\|_{L^{2}(B_{2}^{+})}^{\alpha}\cdot\lim_{x_{n+1}\rightarrow0}\|x_{n+1}^{1-2s}\partial_{n+1}\tilde{w}\|_{L^{2}(B_{2}')}^{1-\alpha}.
\]
\end{cor}

Now we are ready to prove the part~\ref{itm:part-a:lem:small} of Lemma~\ref{lem:small} for
the case when $s\in[\frac{1}{2},1)$. 
\begin{proof}
[Proof of the part~{\rm \ref{itm:part-a:lem:small}} of Lemma~{\rm \ref{lem:small}} for $s \in [ \frac{1}{2} ,1 )$]
In order to invoke the estimation from Corollary~\ref{cor:interIII},
we split our solution $u$ into two parts $\tilde{u}=u_{1}+u_{2}$,
where $u_{1}:= \mathsf{E}_{s}(\zeta u)$ satisfies 
\begin{align*}
\bigg[\partial_{n+1}x_{n+1}^{1-2s}\partial_{n+1}+x_{n+1}^{1-2s}\sum_{j,k=1}^{n}\partial_{j}a_{jk}\partial_{k}\bigg]u_{1} & =0\quad\text{in }\mathbb{R}_{+}^{n+1},\\
u_{1} & =\zeta u\quad\text{on }\mathbb{R}^{n}\times\{0\},
\end{align*}
where $\zeta\in\mathcal{C}_{0}^{\infty}(B_{16}')$ is a smooth cut-off
function with $\zeta=1$ on $B_{8}'$. Since $u_{1}:= \mathsf{E}_{s}(\zeta u)$, from \eqref{eq:apriori} 
we have 
\[
\int_{\mathbb{R}^{n}}|u_{1}(x',x_{n+1})|^{2}\,dx'\le\|u_{1}\|_{L^{2}(\mathbb{R}^{n}\times\{0\})}^{2}\le\|u\|_{L^{2}(B_{16}')}^{2}.
\]
So, 
\begin{align}
\|x_{n+1}^{\frac{1-2s}{2}}u_{1}\|_{L^{2}(B_{10}^{+})}^{2} & \le\int_{0}^{10}\int_{\mathbb{R}^{n}}x_{n+1}^{1-2s}|u_{1}(x',x_{n+1})|^{2}\,dx'\,dx_{n+1}\nonumber \\
 & \le\bigg(\int_{0}^{10}x_{n+1}^{1-2s}\,dx_{n+1}\bigg)\|u\|_{L^{2}(B_{16}')}^{2}=C\|u\|_{L^{2}(B_{16}')}^{2}.\label{eq:small1}
\end{align}

Note that $u_{2}$ satisfies 
\begin{align*}
\bigg[\partial_{n+1}x_{n+1}^{1-2s}\partial_{n+1}+x_{n+1}^{1-2s}\sum_{j,k=1}^{n}\partial_{j}a_{jk}\partial_{k}\bigg]u_{2} & =0\quad\text{in }\mathbb{R}_{+}^{n+1},\\
u_{2} & =u-\zeta u\quad\text{on }\mathbb{R}^{n}\times\{0\}.
\end{align*}
Since $u_{2}=0$ on $B_{8}'$, by Corollary~\ref{cor:interIII}, there
exist $\alpha=\alpha(n,s)\in(0,1)$, $c=c(n,s)\in(0,1)$ and a constant
$C$ such that 
\begin{equation}
\|x_{n+1}^{\frac{1-2s}{2}}u_{2}\|_{L^{2}(B_{c}^{+})}\le C\|x_{n+1}^{\frac{1-2s}{2}}u_{2}\|_{L^{2}(B_{2}^{+})}^{\alpha}\cdot\lim_{x_{n+1}\rightarrow0}\|x_{n+1}^{1-2s}\partial_{n+1}u_{2}\|_{L^{2}(B_{2}')}^{1-\alpha}.\label{eq:small2}
\end{equation}
Let $\eta$ be a smooth, radial cut-off function with $\eta=1$ in
$B_{2}^{+}$ and $\eta=0$ outside $B_{4}^{+}$. Plug $w=\eta x_{n+1}^{1-2s}\partial_{n+1}u_{2}$
into the trace characterization lemma (Lemma \ref{lem:inter2(b)}),
we reach 
\begin{align}
 & \lim_{x_{n+1}\rightarrow0}\|x_{n+1}^{1-2s}\partial_{n+1}u_{2}\|_{L^{2}(B_{2}')}\nonumber \\
\le & C\bigg[\mu^{1-s}(\|x_{n+1}^{\frac{1-2s}{2}}\partial_{n+1}u_{2}\|_{L^{2}(B_{4}^{+})}+\|x_{n+1}^{\frac{2s-1}{2}}\nabla(\eta x_{n+1}^{1-2s}\partial_{n+1}u_{2})\|_{L^{2}(\mathbb{R}_{+}^{n+1})})\nonumber \\
 & +\mu^{-2s}\lim_{x_{n+1}\rightarrow0}\|\eta x_{n+1}^{1-2s}\partial_{n+1}u_{2}\|_{H^{-2s}(\mathbb{R}^{n}\times\{0\})}\bigg].\label{eq:small2-0}
\end{align}
We first control the boundary term of \eqref{eq:small2-0}. Since $\eta$ is a bounded multiplier on $H^{2s}(\mathbb{R}^{n})$, using duality, we have 
\begin{align*}
\|\eta v\|_{H^{-2s}(\mathbb{R}^{n}\times\{0\})} & =\sup_{\|\varphi\|_{H^{2s}(\mathbb{R}^{n}\times\{0\})}=1}|\langle v,\eta\varphi\rangle_{L^{2}(\mathbb{R}^{n}\times\{0\})}|\\
 & \le\|v\|_{H^{-2s}(B_{8}')}\sup_{\|\varphi\|_{H^{2s}(\mathbb{R}^{n}\times\{0\})}=1}\|\eta\varphi\|_{H^{2s}(\mathbb{R}^{n}\times\{0\})}\\
 & \le C\|v\|_{H^{-2s}(B_{8}')}\sup_{\|\varphi\|_{H^{2s}(\mathbb{R}^{n}\times\{0\})}=1}\|\varphi\|_{H^{2s}(\mathbb{R}^{n}\times\{0\})}=C\|v\|_{H^{-2s}(B_{8}')}.
\end{align*}
Plug $v=x_{n+1}^{1-2s}\partial_{n+1}u_{2}$, we have 
\begin{equation}
\lim_{x_{n+1}\rightarrow0}\|\eta x_{n+1}^{1-2s}\partial_{n+1}u_{2}\|_{H^{-2s}(\mathbb{R}^{n}\times\{0\})}\le C\cdot\lim_{x_{n+1}\rightarrow0}\|x_{n+1}^{1-2s}\partial_{n+1}u_{2}\|_{H^{-2s}(B_{8}')}.\label{eq:small2-1}
\end{equation}
Applying the Caccioppoli's inequality in Lemma~\ref{lem:Caccio}, with zero Dirichlet condition and zero inhomogeneous terms, we have 
\begin{equation}
\|x_{n+1}^{\frac{1-2s}{2}}\nabla u_{2}\|_{L^{2}(B_{4}^{+})}\le C\|x_{n+1}^{\frac{1-2s}{2}}u_{2}\|_{L^{2}(B_{8}^{+})}.\label{eq:small2-2a}
\end{equation}
Also, we have 
\begin{align*}
 & \|x_{n+1}^{\frac{2s-1}{2}}\nabla(\eta x_{n+1}^{1-2s}\partial_{n+1}u_{2})\|_{L^{2}(\mathbb{R}_{+}^{n+1})}\\
\le & \|x_{n+1}^{\frac{1-2s}{2}}(\nabla\eta)\partial_{n+1}u_{2}\|_{L^{2}(\mathbb{R}_{+}^{n+1})}+\|x_{n+1}^{\frac{1-2s}{2}}\eta\nabla'\partial_{n+1}u_{2}\|_{L^{2}(\mathbb{R}_{+}^{n+1})}\\
 & +\|x_{n+1}^{\frac{2s-1}{2}}\eta\partial_{n+1}x_{n+1}^{1-2s}\partial_{n+1}u_{2}\|_{L^{2}(\mathbb{R}_{+}^{n+1})}\\
\le & C\|x_{n+1}^{\frac{1-2s}{2}}\partial_{n+1}u_{2}\|_{L^{2}(B_{4}^{+})}+\|x_{n+1}^{\frac{1-2s}{2}}\partial_{n+1}(\nabla'u_{2})\|_{L^{2}(B_{4}^{+})}+\bigg\| x_{n+1}^{\frac{1-2s}{2}}\sum_{j,k=1}^{n}\partial_{j}a_{jk}\partial_{k}u_{2}\bigg\|_{L^{2}(B_{4}^{+})}\\
\le & C\bigg[\|x_{n+1}^{\frac{1-2s}{2}}\nabla u_{2}\|_{L^{2}(B_{4}^{+})}+\|x_{n+1}^{\frac{1-2s}{2}}\nabla(\nabla'u_{2})\|_{L^{2}(B_{4}^{+})}\bigg],
\end{align*}
where the last inequality follows by the boundedness assumptions
of $a_{jk}$. Observe that 
\begin{align*}
0 & =\nabla'\bigg[\partial_{n+1}(x_{n+1}^{1-2s}\partial_{n+1}u_{2})+x_{n+1}^{1-2s}\sum_{j,k=1}^{n}\partial_{j}(a_{jk}\partial_{k}u_{2})\bigg]\\
 & =\bigg[\partial_{n+1}x_{n+1}^{1-2s}\partial_{n+1}+x_{n+1}^{1-2s}\sum_{j,k=1}^{n}\partial_{j}a_{jk}\partial_{k}\bigg](\nabla'u_{2})+x_{n+1}^{1-2s}\sum_{j=1}^{n}\partial_{j}\bigg(\sum_{k=1}^{n}\nabla'a_{jk}\partial_{k}u_{2}\bigg).
\end{align*}
Applying Caccioppoli's inequality in Lemma~\ref{lem:Caccio} on $\nabla'u_{2}$
with zero Dirichlet condition and $f_{j}=\sum_{k=1}^{n}\nabla'a_{jk}\partial_{k}u_{2}$,
since $\|\nabla'a_{jk}\|_{\infty}\le\epsilon$, we have 
\begin{equation}
\|x_{n+1}^{\frac{1-2s}{2}}\partial_{n+1}\nabla(\nabla'u_{2})\|_{L^{2}(B_{4}^{+})}\le C'\|x_{n+1}^{\frac{1-2s}{2}}\nabla'u_{2}\|_{L^{2}(B_{6}^{+})}\le C\|x_{n+1}^{\frac{1-2s}{2}}u_{2}\|_{L^{2}(B_{8}^{+})},\label{eq:small2-2b}
\end{equation}
where the second inequality follows by \eqref{eq:small2-2a}. Hence,
we reach 
\begin{equation}
\|x_{n+1}^{\frac{2s-1}{2}}\nabla(\eta x_{n+1}^{1-2s}\partial_{n+1}u_{2})\|_{L^{2}(\mathbb{R}_{+}^{n+1})}\le C\|x_{n+1}^{\frac{1-2s}{2}}u_{2}\|_{L^{2}(B_{8}^{+})}.\label{eq:small2-3}
\end{equation}
Plugging \eqref{eq:small2-1}, \eqref{eq:small2-2a} and \eqref{eq:small2-3}
into \eqref{eq:small2-0}, and optimizing the result estimate in $\mu>0$
gives 
\[
\lim_{x_{n+1}\rightarrow0}\|x_{n+1}^{1-2s}\partial_{n+1}u_{2}\|_{L^{2}(B_{2}')}\le C\|x_{n+1}^{\frac{1-2s}{2}}u_{2}\|_{L^{2}(B_{8}^{+})}^{\frac{2s}{1+s}}\cdot\lim_{x_{n+1}\rightarrow0}\|x_{n+1}^{1-2s}\partial_{n+1}u_{2}\|_{H^{-2s}(B_{8}')}^{\frac{1-s}{1+s}}.
\]
Plugging this into \eqref{eq:small2} leads to 
\begin{align}
 & \|x_{n+1}^{\frac{1-2s}{2}}u_{2}\|_{L^{2}(B_{c}^{+})}\nonumber \\
\le & C\|x_{n+1}^{\frac{1-2s}{2}}u_{2}\|_{L^{2}(B_{8}^{+})}^{\tilde{\alpha}}\cdot\lim_{x_{n+1}\rightarrow0}\|x_{n+1}^{1-2s}\partial_{n+1}u_{2}\|_{H^{-2s}(B_{8}')}^{1-\tilde{\alpha}}\nonumber \\
\le & C(\|x_{n+1}^{\frac{1-2s}{2}}\tilde{u}\|_{L^{2}(B_{8}^{+})}+\|x_{n+1}^{\frac{1-2s}{2}}u_{1}\|_{L^{2}(B_{8}^{+})})^{\tilde{\alpha}}\times\nonumber \\
 & \quad\times(\lim_{x_{n+1}\rightarrow0}\|x_{n+1}^{1-2s}\partial_{n+1}\tilde{u}\|_{H^{-2s}(B_{8}')}+\lim_{x_{n+1}\rightarrow0}\|x_{n+1}^{1-2s}\partial_{n+1}u_{1}\|_{H^{-2s}(B_{8}')})^{1-\tilde{\alpha}},\label{eq:small3}
\end{align}
where $\tilde{\alpha}=\frac{1-s}{1+s}\alpha+\frac{2s}{1+s}$. Then
we have 
\begin{align}
\lim_{x_{n+1}\rightarrow0}\|x_{n+1}^{1-2s}\partial_{n+1}u_{1}\|_{H^{-2s}(\mathbb{R}^{n}\times\{0\})} & =\|(-P)^{s}u_{1}\|_{H^{-2s}(\mathbb{R}^{n}\times\{0\})}\nonumber \\
 & \le C\|u_{1}\|_{L^{2}(\mathbb{R}^{n}\times\{0\})}\le C\|\tilde{u}\|_{L^{2}(B_{16}')},\label{eq:small4}
\end{align}
where the second inequality follows by Lemma~\ref{lem:well2}. 

By combining \eqref{eq:small1}, \eqref{eq:small3} and \eqref{eq:small4},
we reach 
\begin{align}
 & \|x_{n+1}^{\frac{1-2s}{2}}\tilde{u}\|_{L^{2}(B_{c}^{+})}\nonumber \\
\le & C\bigg(\|x_{n+1}^{\frac{1-2s}{2}}\tilde{u}\|_{L^{2}(B_{16}^{+})}+\|\tilde{u}\|_{L^{2}(B_{16}')}\bigg)^{\tilde{\alpha}}\bigg(\lim_{x_{n+1}\rightarrow0}\|x_{n+1}^{1-2s}\partial_{n+1}\tilde{u}\|_{L^{2}(B_{16}')}+\|u\|_{L^{2}(B_{16}')}\bigg)^{1-\tilde{\alpha}},\label{eq:small5}
\end{align}
which is our desired claim of (a). 
\end{proof}
Indeed, by combining \eqref{eq:small5} with the Caccioppoli's inequality
(Lemma~\ref{lem:Caccio}), we reach 
\begin{align}
 & \|x_{n+1}^{\frac{1-2s}{2}}\tilde{u}\|_{L^{2}(B_{\tilde{c}}^{+})}+\|x_{n+1}^{\frac{1-2s}{2}}\nabla\tilde{u}\|_{L^{2}(B_{\tilde{c}}^{+})}\nonumber \\
\le & C\bigg(\|x_{n+1}^{\frac{1-2s}{2}}\tilde{u}\|_{L^{2}(B_{16}^{+})}+\|\tilde{u}\|_{L^{2}(B_{16}')}\bigg)^{\tilde{\alpha}}\bigg(\lim_{x_{n+1}\rightarrow0}\|x_{n+1}^{1-2s}\partial_{n+1}\tilde{u}\|_{L^{2}(B_{16}')}+\|u\|_{L^{2}(B_{16}')}\bigg)^{1-\tilde{\alpha}}\nonumber \\
 & +\lim_{x_{n+1}\rightarrow0}\|x_{n+1}^{1-2s}\partial_{n+1}\tilde{u}\|_{L^{2}(B_{16}')}^{\frac{1}{2}}\|u\|_{L^{2}(B_{16}')}^{\frac{1}{2}}\label{eq:small6}
\end{align}
with $\tilde{c}=c/2$. Slightly modify the proof of \eqref{eq:small3},
we can obtain the following analogue of Proposition 5.11 of \cite{RS20Calderon}: 
\begin{lem}
\label{lem:inter-IV}Let $s\in[\frac{1}{2},1)$ and $\tilde{w}$ is
the Caffarelli-Silvestre type extension of some $f\in H^{\gamma}(\mathbb{R}^{n})$
as in \eqref{eq:Sch-ext}, where $\gamma\in\mathbb{R}$ with $f|_{C_{s,1}'}=0$.
We assume that 
\[
\max_{1\le j,k\le n}\|a_{jk}-\delta_{jk}\|_{\infty}+\max_{1\le j,k\le n}\|\nabla'a_{jk}\|_{\infty}\le\epsilon
\]
for some sufficiently small $\epsilon>0$. For $s\neq\frac{1}{2}$,
we further assume 
\[
\max_{1\le j,k\le n}\|(\nabla')^{2}a_{jk}\|_{\infty}\le C
\]
for some positive constant $C$. Then there exist $C=C(n,s)$ and
$\alpha=\alpha(n,s)\in(0,1)$ such that 
\[
\|x_{n+1}^{\frac{1-2s}{2}}\tilde{w}\|_{L^{2}(C_{s,1/8}^{+})}\le C\|x_{n+1}^{\frac{1-2s}{2}}\tilde{w}\|_{L^{2}(C_{s,1}^{+})}^{\alpha}\cdot\lim_{x_{n+1}\rightarrow0}\|x_{n+1}^{1-2s}\partial_{n+1}\tilde{w}\|_{H^{-s}(C_{s,1/2}')}^{1-\alpha}.
\]
\end{lem}

\begin{proof}
Let $\eta$ be a smooth cut-off function supported in $C_{s,1/2}^{+}$ with $\eta=1$ in $C_{s,1/4}^{+}$. Using this cut-off function, and following the ideas in the proof of \eqref{eq:small3}, by using Lemma~\ref{lem:inter2(a)} rather than Lemma~\ref{lem:inter2(b)}, we can obtain the above inequality. 
\end{proof}

\subsection{Proof of the part~\ref{itm:part-a:lem:small} of Lemma \ref{lem:small} for the case $s\in(0,1/2)$}

Let $\tilde{w}$ solves \eqref{eq:Sch-ext}. If we define $\overline{s}:=1-s\in(1/2,1)$,
\begin{equation}
v(x)=x_{n+1}^{1-2s}\partial_{n+1}\tilde{w}(x)\quad\text{and}\quad f=\lim_{x_{n+1}\rightarrow0}x_{n+1}^{1-2s}\partial_{n+1}\tilde{w}=c^{-1}_{s}(-P)^{s}u,\label{eq:0small1}
\end{equation}
then 
\begin{align*}
\bigg[\partial_{n+1}x_{n+1}^{1-2\overline{s}}\partial_{n+1}+x_{n+1}^{1-2\overline{s}}\sum_{j,k=1}^{n}\partial_{j}a_{jk}\partial_{k}\bigg]v & =0\quad\text{in }\mathbb{R}_{+}^{n+1},\\
v & =f\quad\text{on }\mathbb{R}^{n}\times\{0\}.
\end{align*}
Using this observation, and follows the ideas in \cite[Proposition 5.12]{RS20Calderon}, we can obtain an analogue of Lemma~\ref{lem:inter-IV}: 
\begin{lem}
Let $s\in(0,1/2)$ and let $x_{0}\in\mathbb{R}^{n}\times\{0\}$. Suppose
\begin{align*}
\bigg[\partial_{n+1}x_{n+1}^{1-2s}\partial_{n+1}+x_{n+1}^{1-2s}\sum_{j,k=1}^{n}\partial_{j}a_{jk}\partial_{k}\bigg]\tilde{w} & =0\quad\text{in }\mathbb{R}_{+}^{n+1},\\
\tilde{w} & =w\quad\text{on }\mathbb{R}^{n}\times\{0\},
\end{align*}
with $w=0$ on $C_{\overline{s},2}'$. We assume that 
\[
\max_{1\le j,k\le n}\|a_{jk}-\delta_{jk}\|_{\infty}+\max_{1\le j,k\le n}\|\nabla'a_{jk}\|_{\infty}\le\epsilon
\]
for some sufficiently small $\epsilon>0$. We further assume 
\[
\max_{1\le j,k\le n}\|(\nabla')^{2}a_{jk}\|_{\infty}\le C
\]
for some positive constant $C$. Then there exist $C=C(n,s)$ and
$\alpha=\alpha(n,s)\in(0,1)$ such that 
\begin{align*}
 & \|x_{n+1}^{\frac{1-2s}{2}}\tilde{w}\|_{L^{2}(C_{\overline{s},1/8}^{+})}\\
\le & C\max\{\|x_{n+1}^{\frac{1-2s}{2}}\tilde{w}\|_{L^{2}(C_{\overline{s},2}^{+})},\lim_{x_{n+1}\rightarrow0}\|x_{n+1}^{1-2s}\partial_{n+1}\tilde{w}\|_{H^{-s}(C_{\overline{s},2}')}\}^{\alpha}\lim_{x_{n+1}\rightarrow0}\|x_{n+1}^{1-2s}\partial_{n+1}\tilde{w}\|_{H^{-s}(C_{\overline{s},2}')}^{1-\alpha}.
\end{align*}
\end{lem}

\begin{proof}
Let $v$ and $f$ as in \eqref{eq:0small1}. Let $\tilde{v}$ be
the Caffarelli-Silvestre type extension of $\eta f$ as in \eqref{eq:Sch-ext},
where $\eta$ is a cut-off function satisfies 
\[
\eta=\begin{cases}
1 & \text{in }C_{\overline{s},1}^{+},\\
0 & \text{outside }C_{\overline{s},2}^{+},
\end{cases}
\]
with $|\partial_{n+1}\eta|\le Cx_{n+1}$. As consequences, the function
$\overline{v}:=v-\tilde{v}$ is the Caffarelli-Silvestre extension
of $(1-\eta)f$ and solves 
\begin{align*}
\bigg[\partial_{n+1}x_{n+1}^{1-2s}\partial_{n+1}+x_{n+1}^{1-2s}\sum_{j,k=1}^{n}\partial_{j}a_{jk}\partial_{k}\bigg]\overline{v} & =0\quad\text{in }\mathbb{R}_{+}^{n+1},\\
\overline{v} & =0\quad\text{on }C_{\overline{s},1}'.
\end{align*}
Hence, by Lemma~\ref{lem:inter-IV} and since $\overline{s}=1-s$,
we have 
\begin{align*}
\|x_{n+1}^{\frac{1-2\overline{s}}{2}}\overline{v}\|_{L^{2}(C_{\overline{s},1/8}^{+})} & \le C\|x_{n+1}^{\frac{1-2\overline{s}}{2}}\overline{v}\|_{L^{2}(C_{\overline{s},1}^{+})}^{\alpha}\cdot\lim_{x_{n+1}\rightarrow0}\|x_{n+1}^{1-2\overline{s}}\partial_{n+1}\overline{v}\|_{H^{-\overline{s}}(C_{\overline{s},1/2}')}^{1-\alpha}\\
 & =C\|x_{n+1}^{\frac{1-2\overline{s}}{2}}\overline{v}\|_{L^{2}(C_{\overline{s},1}^{+})}^{\alpha}\cdot\lim_{x_{n+1}\rightarrow0}\|x_{n+1}^{1-2\overline{s}}\partial_{n+1}\overline{v}\|_{H^{-1+s}(C_{\overline{s},1/2}')}^{1-\alpha}.
\end{align*}
Since $\tilde{w}=0$ on $C_{\overline{s},2}'$, thus 
\begin{align*}
\lim_{x_{n+1}\rightarrow0}x_{n+1}^{1-2\overline{s}}\partial_{n+1}v\bigg|_{C_{\overline{s},1/2}'} & =\lim_{x_{n+1}\rightarrow0}x_{n+1}^{1-2\overline{s}}(\partial_{n+1}x_{n+1}^{1-2s}\partial_{n+1}\tilde{w})\bigg|_{C_{\overline{s},1/2}'}\\
 & =-\lim_{x_{n+1}\rightarrow0}\sum_{j,k=1}^{n}\partial_{j}a_{jk}\partial_{k}\tilde{w}\bigg|_{C_{\overline{s},1/2}'}=0.
\end{align*}
Hence, 
\[
\lim_{x_{n+1}\rightarrow0}x_{n+1}^{1-2\overline{s}}\partial_{n+1}\overline{v}\bigg|_{C_{\overline{s},1/2}'}=\lim_{x_{n+1}\rightarrow0}x_{n+1}^{1-2\overline{s}}\partial_{n+1}\tilde{v}\bigg|_{C_{\overline{s},1/2}'},
\]
and thus 
\[
\|x_{n+1}^{\frac{1-2\overline{s}}{2}}\overline{v}\|_{L^{2}(C_{\overline{s},1/8}^{+})}\le C\|x_{n+1}^{\frac{1-2\overline{s}}{2}}\overline{v}\|_{L^{2}(C_{\overline{s},1}^{+})}^{\alpha}\cdot\lim_{x_{n+1}\rightarrow0}\|x_{n+1}^{1-2\overline{s}}\partial_{n+1}\tilde{v}\|_{H^{-1+s}(C_{\overline{s},1/2}')}^{1-\alpha}.
\]
Using  $\lim_{x_{n+1}\rightarrow0}x_{n+1}^{1-2\overline{s}}\partial_{n+1}\tilde{v}=-c_{\overline{s}}(-P)^{\overline{s}}(\eta f)=-c_{\overline{s}}(-P)^{1-s}(\eta f)$, we have 
\begin{equation}
\lim_{x_{n+1}\rightarrow0}\|x_{n+1}^{1-2\overline{s}}\partial_{n+1}\tilde{v}\|_{H^{-1+s}(C_{\overline{s},1/2}')}\le C\|(-P)^{1-s}(\eta f)\|_{H^{-1+s}(\mathbb{R}^{n})}\le C\|\eta f\|_{H^{1-s}(\mathbb{R}^{n})},\label{eq:0small1-1}
\end{equation}
where the last inequality follows by Lemma~\ref{lem:well1}. Thus,
\begin{equation}
\|x_{n+1}^{\frac{1-2\overline{s}}{2}}\overline{v}\|_{L^{2}(C_{\overline{s},1/8}^{+})}\le C\|x_{n+1}^{\frac{1-2\overline{s}}{2}}\overline{v}\|_{L^{2}(C_{\overline{s},1}^{+})}^{\alpha}\|\eta f\|_{H^{1-s}(\mathbb{R}^{n})}^{1-\alpha}.\label{eq:0small2}
\end{equation}
Firstly, we estimate the right hand side of \eqref{eq:0small2} by 
\begin{align*}
\|x_{n+1}^{\frac{1-2\overline{s}}{2}}\overline{v}\|_{L^{2}(C_{\overline{s},1}^{+})}\le & \|x_{n+1}^{\frac{1-2\overline{s}}{2}}v\|_{L^{2}(C_{\overline{s},1}^{+})}+\|x_{n+1}^{\frac{1-2\overline{s}}{2}}\tilde{v}\|_{L^{2}(C_{\overline{s},1}^{+})}\\
\le & \|x_{n+1}^{\frac{1-2\overline{s}}{2}}v\|_{L^{2}(C_{\overline{s},1}^{+})}+C\|\eta f\|_{H^{\overline{s}}(\mathbb{R}^{n}\times\{0\})}\\
= & \|x_{n+1}^{\frac{1-2s}{2}}\partial_{n+1}\tilde{w}\|_{L^{2}(C_{\overline{s},1}^{+})}+C\|\eta f\|_{H^{1-s}(\mathbb{R}^{n}\times\{0\})}\\
\le & C\bigg[\|x_{n+1}^{\frac{1-2s}{2}}\tilde{w}\|_{L^{2}(C_{\overline{s},2}^{+})}+\|\eta f\|_{H^{1-s}(\mathbb{R}^{n}\times\{0\})}\bigg],
\end{align*}
where the second inequality follows by \eqref{eq:apriori} and the
last one is followed by the Caccioppoli's inequality in Lemma~\ref{lem:Caccio}.
Similarly, we can estimate the left hand side of \eqref{eq:0small2}
by 
\begin{align*}
\|x_{n+1}^{\frac{1-2\overline{s}}{2}}\overline{v}\|_{L^{2}(C_{\overline{s},1/8}^{+})}\ge & \|x_{n+1}^{\frac{1-2\overline{s}}{2}}v\|_{L^{2}(C_{\overline{s},1/8}^{+})}-\|x_{n+1}^{\frac{1-2\overline{s}}{2}}\tilde{v}\|_{L^{2}(C_{\overline{s},1/8}^{+})}\\
\ge & \|x_{n+1}^{\frac{1-2s}{2}}\partial_{n+1}\tilde{w}\|_{L^{2}(C_{\overline{s},1/8}^{+})}-C\|\eta f\|_{H^{1-s}(\mathbb{R}^{n}\times\{0\})}\\
\ge & c\|x_{n+1}^{\frac{1-2s}{2}}\tilde{w}\|_{L^{2}(C_{\overline{s},1/8}^{+})}-C\|\eta f\|_{H^{1-s}(\mathbb{R}^{n}\times\{0\})},
\end{align*}
where the last inequality is followed by Poincar\'{e} inequality. Thus, \eqref{eq:0small2} becomes 
\begin{equation}
\|x_{n+1}^{\frac{1-2s}{2}}\tilde{w}\|_{L^{2}(C_{\overline{s},1/8}^{+})}\le C\bigg[\|x_{n+1}^{\frac{1-2s}{2}}\tilde{w}\|_{L^{2}(C_{\overline{s},2}^{+})}+\|\eta f\|_{H^{1-s}(\mathbb{R}^{n}\times\{0\})}\bigg]^{\alpha}\|\eta f\|_{H^{1-s}(\mathbb{R}^{n})}^{1-\alpha}.\label{eq:0small3}
\end{equation}

Next, we estimate the boundary contribution $\|\eta f\|_{H^{1-s}(\mathbb{R}^{n}\times\{0\})}$.
Using the interpolation inequality in Lemma \ref{lem:inter2(a)}, we have 
\begin{align*}
\|\eta f\|_{H^{\beta}(\mathbb{R}^{n}\times\{0\})}= & \|\langle D'\rangle^{\beta}\eta f\|_{L^{2}(\mathbb{R}^{n}\times\{0\})}\\
\le & C\mu^{1-s}\bigg(\|x_{n+1}^{\frac{2s-1}{2}}\langle D'\rangle^{\beta}(\eta v)\|_{L^{2}(\mathbb{R}_{+}^{n+1})}+\|x_{n+1}^{\frac{2s-1}{2}}\nabla(\langle D'\rangle^{\beta}(\eta v))\|_{L^{2}(\mathbb{R}_{+}^{n+1})}\bigg)\\
 & +C\mu^{-s}\|\langle D'\rangle^{\beta}(\eta f)\|_{H^{-s}(\mathbb{R}^{n}\times\{0\})}.
\end{align*}
Using $\|\langle D'\rangle^{\beta}u\|_{L^{2}}\le\|u\|_{L^{2}}+\|\nabla' u\|_{L^{2}}$
for $\beta\le1$, we have 
\begin{align*}
\|x_{n+1}^{\frac{2s-1}{2}}\langle D'\rangle^{\beta}(\eta v)\|_{L^{2}(\mathbb{R}_{+}^{n+1})}\le & \|x_{n+1}^{\frac{2s-1}{2}}\eta v\|_{L^{2}(\mathbb{R}_{+}^{n+1})}+\|x_{n+1}^{\frac{2s-1}{2}}\nabla'(\eta v)\|_{L^{2}(\mathbb{R}_{+}^{n+1})}\\
\|x_{n+1}^{\frac{2s-1}{2}}\nabla(\langle D'\rangle^{\beta}(\eta v))\|_{L^{2}(\mathbb{R}_{+}^{n+1})}\le & \|x_{n+1}^{\frac{2s-1}{2}}\nabla(\eta v)\|_{L^{2}(\mathbb{R}_{+}^{n+1})}+\|x_{n+1}^{\frac{2s-1}{2}}\nabla\nabla'(\eta v)\|_{L^{2}(\mathbb{R}_{+}^{n+1})}.
\end{align*}
Using \eqref{eq:small2-2a} and \eqref{eq:small2-2b}, we know that
\[
\|x_{n+1}^{\frac{2s-1}{2}}\langle D'\rangle^{\beta}(\eta v)\|_{L^{2}(\mathbb{R}_{+}^{n+1})}+\|x_{n+1}^{\frac{2s-1}{2}}\nabla(\langle D'\rangle^{\beta}(\eta v))\|_{L^{2}(\mathbb{R}_{+}^{n+1})}\le C\|x_{n+1}^{\frac{2s-1}{2}}\tilde{w}\|_{L^{2}(C_{\overline{s},2}^{+})},
\]
hence 
\begin{equation}
\|\eta f\|_{H^{\beta}(\mathbb{R}^{n}\times\{0\})}\le C\bigg[\mu^{1-s}\|x_{n+1}^{\frac{2s-1}{2}}\tilde{w}\|_{L^{2}(C_{\overline{s},2}^{+})}+\mu^{-s}\|\eta f\|_{H^{\beta-s}(\mathbb{R}^{n}\times\{0\})}\bigg].\label{eq:0small4}
\end{equation}

Choosing $\mu>0$ in \eqref{eq:0small4} such that the right contributions
become equal, i.e. 
\[
\mu=\frac{\|\eta f\|_{H^{\beta-s}(\mathbb{R}^{n}\times\{0\})}}{\|x_{n+1}^{\frac{2s-1}{2}}\tilde{w}\|_{L^{2}(C_{\overline{s},2}^{+})}}.
\]
Here, using unique continuation, we notice  $\|x_{n+1}^{\frac{2s-1}{2}}\tilde{w}\|_{L^{2}(C_{\overline{s},2}^{+})}\neq0$,
unless $\tilde{w}$ vanishes globally. Using this choice of $\mu>0$,
we reach the multiplicative estimate 
\begin{equation}
\|\eta f\|_{H^{\beta}(\mathbb{R}^{n}\times\{0\})}\le C\|x_{n+1}^{\frac{2s-1}{2}}\tilde{w}\|_{L^{2}(C_{\overline{s},2}^{+})}^{s}\|\eta f\|_{H^{\beta-s}(\mathbb{R}^{n}\times\{0\})}^{1-s}.\label{eq:0small5}
\end{equation}
Starting from $\beta=1-s$, if we iterate \eqref{eq:0small5} for
$k$ times, we reach 
\[
\|\eta f\|_{H^{1-s}(\mathbb{R}^{n}\times\{0\})}\le C\|x_{n+1}^{\frac{2s-1}{2}}\tilde{w}\|_{L^{2}(C_{\overline{s},2}^{+})}^{\gamma}\|\eta f\|_{H^{1-s-ks}(\mathbb{R}^{n}\times\{0\})}^{1-\gamma}.
\]
Choose $k\in\mathbb{N}$ be the smallest integer such that $1-ks<0$,
we reach 
\begin{align}
\|\eta f\|_{H^{1-s}(\mathbb{R}^{n}\times\{0\})} & \le C\|x_{n+1}^{\frac{2s-1}{2}}\tilde{w}\|_{L^{2}(C_{\overline{s},2}^{+})}^{\gamma}\|\eta f\|_{H^{-s}(\mathbb{R}^{n}\times\{0\})}^{1-\gamma}\nonumber \\
 & \le C\|x_{n+1}^{\frac{2s-1}{2}}\tilde{w}\|_{L^{2}(C_{\overline{s},2}^{+})}^{\gamma}\|f\|_{H^{-s}(C_{\overline{s},2}')}^{1-\gamma}.\label{eq:0small6}
\end{align}
Inserting \eqref{eq:0small6} into \eqref{eq:0small3} gives our desired
result. 
\end{proof}
For our purpose, we only need the following version of inequality: 
\begin{cor}
\label{cor:inter-V}Let $s\in(0,1/2)$ and let $x_{0}\in\mathbb{R}^{n}\times\{0\}$.
Suppose 
\begin{align*}
\bigg[\partial_{n+1}x_{n+1}^{1-2s}\partial_{n+1}+x_{n+1}^{1-2s}\sum_{j,k=1}^{n}\partial_{j}a_{jk}\partial_{k}\bigg]\tilde{w} & =0\quad\text{in }\mathbb{R}_{+}^{n+1},\\
\tilde{w} & =w\quad\text{on }\mathbb{R}^{n}\times\{0\},
\end{align*}
with $w=0$ on $C_{\overline{s},2}'$. We assume that 
\[
\max_{1\le j,k\le n}\|a_{jk}-\delta_{jk}\|_{\infty}+\max_{1\le j,k\le n}\|\nabla'a_{jk}\|_{\infty}\le\epsilon
\]
for some sufficiently small $\epsilon>0$. We further assume 
\[
\max_{1\le j,k\le n}\|(\nabla')^{2}a_{jk}\|_{\infty}\le C
\]
for some positive constant $C$. Then there exist $C=C(n,s)$, $c=c(n,s)$
and $\alpha=\alpha(n,s)\in(0,1)$ such that 
\begin{align*}
 & \|x_{n+1}^{\frac{1-2s}{2}}\tilde{w}\|_{L^{2}(B_{c}^{+})}\\
\le & C\max\{\|x_{n+1}^{\frac{1-2s}{2}}\tilde{w}\|_{L^{2}(B_{2}^{+})},\lim_{x_{n+1}\rightarrow0}\|x_{n+1}^{1-2s}\partial_{n+1}\tilde{w}\|_{H^{-s}(B_{2}')}\}^{\alpha}\lim_{x_{n+1}\rightarrow0}\|x_{n+1}^{1-2s}\partial_{n+1}\tilde{w}\|_{H^{-s}(B_{2}')}^{1-\alpha}\\
\le & C\bigg[\|x_{n+1}^{\frac{1-2s}{2}}\tilde{w}\|_{L^{2}(B_{2}^{+})}^{\alpha}\cdot\lim_{x_{n+1}\rightarrow0}\|x_{n+1}^{1-2s}\partial_{n+1}\tilde{w}\|_{H^{-s}(B_{2}')}^{1-\alpha}+\lim_{x_{n+1}\rightarrow0}\|x_{n+1}^{1-2s}\partial_{n+1}\tilde{w}\|_{H^{-s}(B_{2}')}\bigg].
\end{align*}
\end{cor}

Now, we are ready to proof the part~\ref{itm:part-a:lem:small} of Lemma~\ref{lem:small} for the case $s\in(0,1/2)$. 
\begin{proof}
[Proof of the part~{\rm \ref{itm:part-a:lem:small}} of Lemma~{\rm \ref{lem:small}} for $s \in (0, \frac{1}{2})$]
The case $s\in(0,1/2)$ is similar as the case $s\in(1/2,1)$. As
above, the estimation for $u_{1}$ is a direct result of \eqref{eq:apriori}.
For $u_{2}$, we use Corollary~\ref{cor:inter-V} and the interpolation
inequality in Lemma~\ref{lem:inter2(b)}. With this estimation,
the analogues of \eqref{eq:small5} and \eqref{eq:small6} are followed by combining the estimates in splitting argument as above. Note
that \eqref{eq:small6} becomes 
\begin{align}
 & \|x_{n+1}^{\frac{1-2s}{2}}\tilde{u}\|_{L^{2}(B_{\tilde{c}}^{+})}+\|x_{n+1}^{\frac{1-2s}{2}}\nabla\tilde{u}\|_{L^{2}(B_{\tilde{c}}^{+})}\nonumber \\
\le & C\bigg(\|x_{n+1}^{\frac{1-2s}{2}}\tilde{u}\|_{L^{2}(B_{16}^{+})}+\|\tilde{u}\|_{L^{2}(B_{16}')}\bigg)^{\alpha} \bigg(\lim_{x_{n+1}\rightarrow0}\|x_{n+1}^{1-2s}\partial_{n+1}\tilde{u}\|_{L^{2}(B_{16}')}+\|u\|_{L^{2}(B_{16}')}\bigg)^{1-\alpha}\nonumber \\
 & +C\bigg(\|x_{n+1}^{\frac{1-2s}{2}}\tilde{u}\|_{L^{2}(B_{16}^{+})}+\|\tilde{u}\|_{L^{2}(B_{16}')}\bigg)^{\frac{2s}{1+s}} \bigg(\lim_{x_{n+1}\rightarrow0}\|x_{n+1}^{1-2s}\partial_{n+1}\tilde{u}\|_{L^{2}(B_{16}')}+\|u\|_{L^{2}(B_{16}')}\bigg)^{\frac{1-s}{1+s}}\nonumber \\
 & +\lim_{x_{n+1}\rightarrow0}\|x_{n+1}^{1-2s}\partial_{n+1}\tilde{u}\|_{L^{2}(B_{16}')}^{\frac{1}{2}}\|u\|_{L^{2}(B_{16}')}^{\frac{1}{2}},\label{eq:0small7}
\end{align}
which is our desired result. 
\end{proof}
Finally, combining \eqref{eq:0small7} and Lemma~\ref{lem:Linfty-L2}, we can immediately obtain the part~\ref{itm:part-b:lem:small} of Lemma~\ref{lem:small}. 

\section{\label{sec:Carleman}Carleman Estimate}

\subsection{A Carleman estimate with differentiability assumption}

Modifying the arguments in \cite{Reg97StrongUniquenessSecondOrderElliptic}, we can proof the following
Carleman estimate. 
\begin{thm}
\label{thm:Carl-diff} Let $s\in(0,1)$ and let $\tilde{u}\in H^{1}(\mathbb{R}_{+}^{n+1},x_{n+1}^{1-2s})$
with ${\rm supp}(\tilde{u})\subset\mathbb{R}_{+}^{n+1}\setminus B_{1}^{+}$
be a solution to 
\begin{align*}
\bigg[\partial_{n+1}x_{n+1}^{1-2s}\partial_{n+1}+x_{n+1}^{1-2s}\sum_{j,k=1}^{n}a_{jk}\partial_{j}\partial_{k}\bigg]\tilde{u} & =f\quad\text{in }\mathbb{R}_{+}^{n+1},\\
\lim_{x_{n+1}\rightarrow0}x_{n+1}^{1-2s}\partial_{n+1}\tilde{u} & =V\tilde{u}\quad\text{on }\mathbb{R}^{n}\times\{0\},
\end{align*}
where $x=(x',x_{n+1})\in\mathbb{R}^{n}\times\mathbb{R}_{+}$, $f\in L^{2}(\mathbb{R}_{+}^{n+1},x_{n+1}^{2s-1})$
with compact support in $\mathbb{R}_{+}^{n+1}$, and $V\in\mathcal{C}^{1}(\mathbb{R}^{n})$.
Assume that 
\[
\max_{1\le j,k\le n}\sup_{|x'|\ge1}|a_{jk}(x')-\delta_{jk}(x')|+\max_{1\le j,k\le n}\sup_{|x'|\ge1}|x'||\nabla'a_{jk}(x')|\le\epsilon
\]
for some sufficiently small $\epsilon>0$. Let further $\phi(x)=|x|^{\alpha}$
for $\alpha\ge1$. Then there exist constants $C=C(n,s,\alpha)$
and $\tau_{0}=\tau_{0}(n,s,\alpha)$ such that 
\begin{align*}
 & \tau^{3}\|e^{\tau\phi}|x|^{\frac{3\alpha}{2}-1}x_{n+1}^{\frac{1-2s}{2}}\tilde{u}\|_{L^{2}(\mathbb{R}_{+}^{n+1})}^{2}+\tau\|e^{\tau\phi}|x|^{\frac{\alpha}{2}}x_{n+1}^{\frac{1-2s}{2}}\nabla\tilde{u}\|_{L^{2}(\mathbb{R}_{+}^{n+1})}^{2}\\
 & +\tau^{-1}\|e^{\tau\phi}|x|^{-\frac{\alpha}{2}+1}x_{n+1}^{\frac{1-2s}{2}}\nabla(\nabla'\tilde{u})\|_{L^{2}(\mathbb{R}_{+}^{n+1})}^{2}\\
\le & C\bigg[\|e^{\tau\phi}x_{n+1}^{\frac{2s-1}{2}}|x|f\|_{L^{2}(\mathbb{R}_{+}^{n+1})}^{2}+\tau\|e^{\tau\phi}|x|^{\frac{\alpha}{2}}(|V|^{\frac{1}{2}}+|x'|^{\frac{1}{2}}|\nabla'V|^{\frac{1}{2}})\tilde{u}\|_{L^{2}(\mathbb{R}^{n}\times\{0\})}^{2}\\
 & \qquad+\tau^{-1}\|e^{\tau\phi}|x|^{-\frac{\alpha}{2}+1}(|V|^{\frac{1}{2}}+|x'|^{\frac{1}{2}}|\nabla'V|^{\frac{1}{2}})\nabla'\tilde{u}\|_{L^{2}(\mathbb{R}^{n}\times\{0\})}^{2}\bigg].
\end{align*}
for all $\tau\ge\tau_{0}$. Here, $\nabla'=(\partial_{1},\cdots,\partial_{n})$
and $\nabla=(\partial_{1},\cdots,\partial_{n},\partial_{n+1})$. 
\end{thm}

\begin{proof}
[Proof of Theorem {\rm \ref{thm:Carl-diff}}] \textbf{Step 1: Changing the coordinates.} Write $x=e^{t}\omega$ with
$t\in\mathbb{R}$ and $\omega\in\mathcal{S}_{+}^{n}$, we have 
\[
\partial_{j}=e^{-t}(\omega_{j}\partial_{t}+\Omega_{j})\quad\text{for all }j=1,\cdots,n+1.
\]
Since 
\begin{equation}
\Omega_{k}\omega_{j}=\delta_{jk}-\omega_{k}\omega_{j},\label{eq:Carl0}
\end{equation}
so 
\[
\partial_{j}\partial_{k}=e^{-2t}(\omega_{j}\omega_{k}\partial_{t}^{2}+\omega_{j}\Omega_{k}\partial_{t}+\omega_{k}\Omega_{j}\partial_{t}+(\delta_{jk}-2\omega_{j}\omega_{k})\partial_{t}+\Omega_{j}\Omega_{k}-\omega_{j}\Omega_{k}).
\]
Since $\partial_{j}$ and $\partial_{k}$ commute, then 
\[
\Omega_{j}\Omega_{k}-\omega_{j}\Omega_{k}=\Omega_{k}\Omega_{j}-\omega_{k}\Omega_{j},
\]
that is, $\Omega_{j}$ and $\Omega_{k}$ commute up to some lower
order terms. Write $\partial_{j}\partial_{k}=\frac{1}{2}(\partial_{j}\partial_{k}+\partial_{k}\partial_{j})$,
we reach 
\begin{align*}
\partial_{j}\partial_{k}= & e^{-2t}\bigg(\omega_{j}\omega_{k}\partial_{t}^{2}+\omega_{j}\Omega_{k}\partial_{t}+\omega_{k}\Omega_{j}\partial_{t}+(\delta_{jk}-2\omega_{j}\omega_{k})\partial_{t}\\
 & \qquad+\frac{1}{2}\Omega_{j}\Omega_{k}+\frac{1}{2}\Omega_{k}\Omega_{j}-\frac{1}{2}\omega_{j}\Omega_{k}-\frac{1}{2}\omega_{k}\Omega_{j}\bigg).
\end{align*}
Also, the vector fields have the following properties 
\begin{align*}
\sum_{j=1}^{n+1}\omega_{j}\Omega_{j} & =0\quad\text{and}\quad\sum_{j=1}^{n+1}\Omega_{j}\omega_{j}=n\quad\text{in }\mathcal{S}_{+}^{n},\\
\sum_{j=1}^{n}\omega_{j}\Omega_{j} & =0\quad\text{and}\quad\sum_{j=1}^{n}\Omega_{j}\omega_{j}=n\quad\text{on }\partial\mathcal{S}_{+}^{n}.
\end{align*}
Using this coordinate, 
\begin{align*}
f= & e^{-(1+2s)t}\bigg[\omega_{n+1}^{1-2s}\partial_{t}^{2}+\omega_{n+1}^{1-2s}(n-2s)\partial_{t}+\sum_{j=1}^{n+1}\Omega_{j}\omega_{n+1}^{1-2s}\Omega_{j}\bigg]\tilde{u}\\
 & +e^{-(1+2s)t}\omega_{n+1}^{1-2s}\sum_{j,k=1}^{n}(a_{jk}-\delta_{jk})\bigg[\omega_{j}\omega_{k}\partial_{t}^{2}+\omega_{j}\Omega_{k}\partial_{t}+\omega_{k}\Omega_{j}\partial_{t}+\frac{1}{2}\Omega_{j}\Omega_{k}+\frac{1}{2}\Omega_{k}\Omega_{j}\bigg]\tilde{u}\\
 & +e^{-(1+2s)t}\omega_{n+1}^{1-2s}\sum_{j,k=1}^{n}(a_{jk}-\delta_{jk})\bigg[(\delta_{jk}-2\omega_{j}\omega_{k})\partial_{t}-\frac{1}{2}\omega_{j}\Omega_{k}-\frac{1}{2}\omega_{k}\Omega_{j}\bigg]\tilde{u}\quad\text{in }\mathcal{S}_{+}^{n}\times\mathbb{R}.
\end{align*}
Next, let $\overline{u}=e^{\frac{n-2s}{2}t}\tilde{u}$ and $\tilde{f}=e^{\frac{n-2s}{2}t}e^{(1+2s)t}f=e^{\frac{n+2+2s}{2}t}f$,
\begin{align}
\tilde{f}= & \bigg[\omega_{n+1}^{1-2s}\partial_{t}^{2}+\sum_{j=1}^{n+1}\Omega_{j}\omega_{n+1}^{1-2s}\Omega_{j}-\omega_{n+1}^{1-2s}\frac{(n-2s)^{2}}{4}\bigg]\overline{u}\nonumber \\
 & +\omega_{n+1}^{1-2s}\sum_{j,k=1}^{n}(a_{jk}-\delta_{jk})\bigg[\omega_{j}\omega_{k}\partial_{t}^{2}+\omega_{j}\Omega_{k}\partial_{t}+\omega_{k}\Omega_{j}\partial_{t}+\frac{1}{2}\Omega_{j}\Omega_{k}+\frac{1}{2}\Omega_{k}\Omega_{j}\bigg]\overline{u}\nonumber \\
 & +\omega_{n+1}^{1-2s}\sum_{j,k=1}^{n}(a_{jk}-\delta_{jk})\bigg[(\delta_{jk}-(n+2-2s)\omega_{j}\omega_{k})\partial_{t}\nonumber \\
 & \qquad-\frac{n+1-2s}{2}\omega_{j}\Omega_{k}-\frac{n+1-2s}{2}\omega_{k}\Omega_{j}\bigg]\overline{u}\nonumber \\
 & +\omega_{n+1}^{1-2s}\sum_{j,k=1}^{n}(a_{jk}-\delta_{jk})\bigg[\frac{(n-2s)^{2}}{4}\omega_{j}\omega_{k}-\frac{n-2s}{2}(\delta_{jk}-2\omega_{j}\omega_{k})\bigg]\overline{u}\quad\text{in }\mathcal{S}_{+}^{n}\times\mathbb{R}.\label{eq:Carl1}
\end{align}
Also, 
\[
\lim_{\omega_{n+1}\rightarrow0}\omega_{n+1}^{1-2s}\Omega_{n+1}\overline{u}=\tilde{V}\overline{u},
\]
where $\tilde{V}=e^{2st}V$. 

\textbf{Step 2: Conjugation.} Next, setting $\overline{v}=\omega_{n+1}^{\frac{1-2s}{2}}e^{\tau\varphi}\overline{u}$,
where $\varphi(t)=\phi(e^{t}\omega)=e^{\alpha t}$, we reach 
\begin{equation}
\omega_{n+1}^{\frac{2s-1}{2}}e^{\tau\varphi}\tilde{f}=L^{+}\overline{v}=(S-A+(I)+(II)+(III))\overline{v}\quad\text{in }\mathcal{S}_{+}^{n}\times\mathbb{R},\label{eq:Carl2}
\end{equation}
where 
\begin{align*}
S & =\partial_{t}^{2}+\tilde{\Delta}_{\omega}+\tau^{2}|\varphi'|^{2}-\tau\varphi''-\frac{(n-2s)^{2}}{4},\quad\tilde{\Delta}_{\omega}=\sum_{j=1}^{n+1}\omega_{n+1}^{\frac{2s-1}{2}}\Omega_{j}\omega_{n+1}^{1-2s}\Omega_{j}\omega_{n+1}^{\frac{2s-1}{2}}\\
A & =2\tau\varphi'\partial_{t}\\
(I) & =\sum_{j,k=1}^{n}(a_{jk}-\delta_{jk})\bigg[\omega_{j}\omega_{k}\partial_{t}^{2}+\omega_{j}\Omega_{k}\partial_{t}+\omega_{k}\Omega_{j}\partial_{t}+\frac{1}{2}\Omega_{j}\Omega_{k}+\frac{1}{2}\Omega_{k}\Omega_{j}\bigg]\\
(II) & =\sum_{j,k=1}^{n}(a_{jk}-\delta_{jk})\bigg[(-2\tau\varphi'\omega_{j}\omega_{k}+(\delta_{jk}-(n+1)\omega_{j}\omega_{k}))\partial_{t}-\bigg(\tau\varphi'+\frac{n}{2}\bigg)(\omega_{j}\Omega_{k}+\omega_{k}\Omega_{j})\bigg]\\
(III) & =\sum_{j,k=1}^{n}(a_{jk}-\delta_{jk})\bigg[\omega_{j}\omega_{k}(\tau^{2}|\varphi'|^{2}-\tau\varphi''+(n+1)\tau\varphi'+C_{1})+C_{2}\bigg],
\end{align*}
for some constants $C_{1}$ and $C_{2}$. Also, 
\begin{equation}
\lim_{\omega_{n+1}\rightarrow0}\omega_{n+1}^{1-2s}\Omega_{n+1}\omega_{n+1}^{\frac{2s-1}{2}}\overline{v}=\tilde{V}\omega_{n+1}^{\frac{2s-1}{2}}\overline{v}\quad\text{on }\partial\mathcal{S}_{+}^{n}\times\mathbb{R}.\label{eq:Carl3}
\end{equation}
We denote the norm and the scalar product in the bulk and the boundary
space by 
\begin{align*}
\|\bullet\| :=\|\bullet\|_{L^{2}(\mathcal{S}_{+}^{n}\times\mathbb{R})}, & \qquad \|\bullet\|_{0} :=\|\bullet\|_{L^{2}(\partial\mathcal{S}_{+}^{n}\times\mathbb{R})}\\
\langle\bullet,\bullet\rangle :=\langle\bullet,\bullet\rangle_{L^{2}(\mathcal{S}_{+}^{n}\times\mathbb{R})}, & \qquad 
\langle\bullet,\bullet\rangle_{0} :=\langle\bullet,\bullet\rangle_{L^{2}(\partial\mathcal{S}_{+}^{n}\times\mathbb{R})}
\end{align*}
and we omit the notation ``$\lim_{\omega_{n+1}\rightarrow0}$''
in $\|\bullet\|_{0}$ and $\langle\bullet,\bullet\rangle_{0}$. 

\textbf{Step 3: Showing the ellipticity of $\tilde{\Delta}_{\omega}$.} We need to prove the ellipticity of $\tilde{\Delta}_{\omega}$: 
\begin{lem}
\label{lem:ellip} Suppose \eqref{eq:Carl3} holds, then 
\begin{align*}
\|\tilde{\Delta}_{\omega}\overline{v}\|^{2}\ge & c_{0}\sum_{(j,k)\neq(n+1,n+1)}\|\omega_{n+1}^{\frac{1-2s}{2}}\Omega_{j}\Omega_{k}\omega_{n+1}^{\frac{2s-1}{2}}\overline{v}\|^{2}\\
 & -C\bigg(\sum_{j=1}^{n+1}\|\omega_{n+1}^{\frac{1-2s}{2}}\Omega_{j}\omega_{n+1}^{\frac{2s-1}{2}}\overline{v}\|^{2}+\|\overline{v}\|^{2}+\|(|\tilde{V}|^{\frac{1}{2}}+|\nabla_{\omega}'\tilde{V}|^{\frac{1}{2}})\nabla_{\omega}'\omega_{n+1}^{\frac{2s-1}{2}}\overline{v}\|_{0}^{2}\\
 & \qquad+\||\nabla_{\omega}'\tilde{V}|^{\frac{1}{2}}\omega_{n+1}^{\frac{2s-1}{2}}\overline{v}\|_{0}^{2}\bigg).
\end{align*}
\end{lem}

\begin{proof}
Note that 
\begin{align*}
\|\tilde{\Delta}_{\omega}\overline{v}\|^{2}= & \bigg\|\sum_{j=1}^{n}\omega_{n+1}^{\frac{2s-1}{2}}\Omega_{j}\omega_{n+1}^{1-2s}\Omega_{j}\omega_{n+1}^{\frac{2s-1}{2}}\overline{v}+\omega_{n+1}^{\frac{2s-1}{2}}\Omega_{n+1}\omega_{n+1}^{1-2s}\Omega_{n+1}\omega_{n+1}^{\frac{2s-1}{2}}\overline{v}\bigg\|^{2}\\
\ge & \bigg\|\sum_{j=1}^{n}\omega_{n+1}^{\frac{2s-1}{2}}\Omega_{j}\omega_{n+1}^{1-2s}\Omega_{j}\omega_{n+1}^{\frac{2s-1}{2}}\overline{v}\bigg\|^{2}\\
 & +2\sum_{j=1}^{n}\langle\omega_{n+1}^{\frac{2s-1}{2}}\Omega_{j}\omega_{n+1}^{1-2s}\Omega_{j}\omega_{n+1}^{\frac{2s-1}{2}}\overline{v},\omega_{n+1}^{\frac{2s-1}{2}}\Omega_{n+1}\omega_{n+1}^{1-2s}\Omega_{n+1}\omega_{n+1}^{\frac{2s-1}{2}}\overline{v}\rangle.
\end{align*}
The integration by parts is given by 
\begin{align*}
 & \int_{\mathbb{R}_{+}^{n+1}}(\Omega_{n+1}v)u\,dx+\int_{\mathbb{R}_{+}^{n+1}}v(\Omega_{n+1}u)\,dx=\int_{\mathbb{R}_{+}^{n+1}}\Omega_{n+1}(uv)\,dx\\
= & \int_{\mathbb{R}_{+}^{n+1}}|x|\partial_{n+1}(uv)\,dx-\int_{\mathcal{S}_{+}^{n}}\int_{0}^{\infty}r\omega_{n+1}\partial_{t}(uv)r^{n}\,dr\,d\omega\\
= & -\int_{\mathbb{R}^{n}\times\{0\}}|x'|uv\,dx'-\int_{\mathbb{R}_{+}^{n+1}}\omega_{n+1}uv\,dx+(n+1)\int_{\mathcal{S}_{+}^{n}}\int_{0}^{\infty}\omega_{n+1}(uv)r^{n}\,dr\,d\omega\\
= & -\int_{\mathbb{R}^{n}\times\{0\}}|x'|uv\,dx'+n\int_{\mathbb{R}_{+}^{n+1}}\omega_{n+1}uv\,dx.
\end{align*}
Similar integration by parts formula holds for $\Omega_{j}$ for $j=1,\cdots,n$. 

Indeed, by \eqref{eq:Carl0}, we know that for $j=1,\cdots,n$, $\Omega_{j}$
and $\omega_{n+1}$ are commute up to some lower order term. So, to
estimate the first term, it is suffice to estimate $\|\sum_{j=1}^{n}\Omega_{j}^{2}\overline{v}\|^{2}$.
Finally, the lower order terms can be easily estimated using integration by parts, 
\end{proof}
Defining $L^{-}:=S+A+(I)-(II)+(III)$, 
\[
\mathscr{D} := \|L^{+}\overline{v}\|^{2}-\|L^{-}\overline{v}\|^{2} \quad \text{and} \quad \mathscr{S} :=  \||\varphi'|^{-\frac{1}{2}}L^{+}\overline{v}\|^{2}+\||\varphi'|^{-\frac{1}{2}}L^{-}\overline{v}\|^{2}.
\] 

\textbf{Step 4: Estimating the difference $\mathscr{D}$.} Observe that $\mathscr{D}=-4\langle S\overline{v},A\overline{v}\rangle+R$, where 
\[
R=4\langle S\overline{v},(II)\overline{v}\rangle-4\langle A\overline{v},(I)\overline{v}\rangle-4\langle A\overline{v},(III)\overline{v}\rangle+4\langle(I)\overline{v},(II)\overline{v}\rangle+4\langle(II)\overline{v},(III)\overline{v}\rangle.
\]
By using \eqref{eq:Carl0} and integration by parts, we can compute 
\begin{align*}
-4\langle S\overline{v},A\overline{v}\rangle\ge & 4\tau\||\varphi''|^{\frac{1}{2}}\partial_{t}\overline{v}\|^{2}-4\tau\sum_{j=1}^{n+1}\||\varphi''|^{\frac{1}{2}}\omega_{n+1}^{\frac{1-2s}{2}}\Omega_{j}\omega_{n+1}^{\frac{2s-1}{2}}\overline{v}\|^{2}+\frac{119}{10}\tau^{3}\|\varphi'|\varphi''|^{\frac{1}{2}}\overline{v}\|^{2}\\
 & -2\tau\|(|\tilde{V}|^{\frac{1}{2}}+|\partial_{t}\tilde{V}|^{\frac{1}{2}})|\varphi''|^{\frac{1}{2}}\omega_{n+1}^{\frac{2s-1}{2}}\overline{v}\|_{0}^{2}.
\end{align*}
Since 
\[
\max_{1\le j,k\le n}|a_{jk}-\delta_{jk}|+\max_{1\le j,k\le n}|\partial_{t}a_{jk}|+\max_{1\le j,k\le n}|\nabla_{\omega}'a_{jk}|\le\epsilon,
\]
by using integration by parts, again we reach 
\begin{align*}
R\ge & -\tau\epsilon C\||\varphi'|^{\frac{1}{2}}\partial_{t}\overline{v}\|^{2}-\tau\epsilon C\sum_{j=1}^{n+1}\||\varphi''|^{\frac{1}{2}}\omega_{n+1}^{\frac{1-2s}{2}}\Omega_{j}\omega_{n+1}^{\frac{2s-1}{2}}\overline{v}\|^{2}-\tau^{3}\epsilon C\|\varphi'|\varphi''|^{\frac{1}{2}}\overline{v}\|^{2}\\
 & -\tau\epsilon C\|(|\tilde{V}|^{\frac{1}{2}}+|\partial_{t}\tilde{V}|^{\frac{1}{2}}+|\nabla_{\omega}'\tilde{V}|^{\frac{1}{2}})|\varphi''|^{\frac{1}{2}}\omega_{n+1}^{\frac{2s-1}{2}}\overline{v}\|_{0}^{2}.
\end{align*}
Hence, for small $\epsilon>0$ and large $\tau_{0}$, we reach 
\begin{align}
\mathscr{D}\ge & \frac{39}{10}\tau\||\varphi''|^{\frac{1}{2}}\partial_{t}\overline{v}\|^{2}-\frac{41}{10}\tau\sum_{j=1}^{n+1}\||\varphi''|^{\frac{1}{2}}\omega_{n+1}^{\frac{1-2s}{2}}\Omega_{j}\omega_{n+1}^{\frac{2s-1}{2}}\overline{v}\|^{2}+\frac{118}{10}\tau^{3}\|\varphi'|\varphi''|^{\frac{1}{2}}\overline{v}\|^{2}\nonumber \\
 & -C\tau\|(|\tilde{V}|^{\frac{1}{2}}+|\partial_{t}\tilde{V}|^{\frac{1}{2}}+|\nabla_{\omega}'\tilde{V}|^{\frac{1}{2}})|\varphi''|^{\frac{1}{2}}\omega_{n+1}^{\frac{2s-1}{2}}\overline{v}\|_{0}^{2}.\label{eq:Carl-difference}
\end{align}

\textbf{Step 5: Estimating the sum $\mathscr{S}$.} Note that  
\begin{align*}
\mathscr{S}\ge & 2\||\varphi'|^{-\frac{1}{2}}S\overline{v}\|^{2}+2\||\varphi'|^{-\frac{1}{2}}A\overline{v}\|^{2}\\
 & -C\epsilon\||\varphi'|^{-\frac{1}{2}}\partial_{t}^{2}\overline{v}\|^{2}-C\epsilon\sum_{j=1}^{n}\||\varphi'|^{-\frac{1}{2}}\partial_{t}\omega_{n+1}^{\frac{1-2s}{2}}\Omega_{j}\omega_{n+1}^{\frac{2s-1}{2}}\overline{v}\|^{2}-C\epsilon\sum_{j,k=1}^{n}\||\varphi'|^{-\frac{1}{2}}\Omega_{j}\Omega_{k}\overline{v}\|^{2}\\
 & -C\epsilon\tau^{2}\||\varphi'|^{\frac{1}{2}}\partial_{t}\overline{v}\|^{2}-C\epsilon\tau^{2}\sum_{j=1}^{n}\||\varphi'|^{\frac{1}{2}}\Omega_{j}\overline{v}\|^{2}-C\epsilon\tau^{4}\||\varphi'|^{\frac{3}{2}}\overline{v}\|^{2}.
\end{align*}
Observe that 
\[
2\||\varphi'|^{-\frac{1}{2}}S\overline{v}\|^{2}\ge\frac{19}{10}\||\varphi'|^{-\frac{1}{2}}\partial_{t}^{2}\overline{v}+|\varphi'|^{-\frac{1}{2}}\tilde{\Delta}_{\omega}\overline{v}+\tau^{2}|\varphi'|^{\frac{3}{2}}\overline{v}\|^{2}-C\tau^{2}\||\varphi''|^{\frac{1}{2}}\overline{v}\|^{2}.
\]
For $\delta\in(0,1)$, write 
\begin{align*}
 & \||\varphi'|^{-\frac{1}{2}}\partial_{t}^{2}\overline{v}+|\varphi'|^{-\frac{1}{2}}\tilde{\Delta}_{\omega}\overline{v}+\tau^{2}|\varphi'|^{\frac{3}{2}}\overline{v}\|^{2}\\
= & \||\varphi'|^{-\frac{1}{2}}\partial_{t}^{2}\overline{v}\|^{2}+(1-\delta)\||\varphi'|^{-\frac{1}{2}}\tilde{\Delta}_{\omega}\overline{v}\|^{2}+\delta\||\varphi'|^{-\frac{1}{2}}\tilde{\Delta}_{\omega}\overline{v}\|^{2}+\tau^{4}\||\varphi'|^{\frac{3}{2}}\overline{v}\|^{2}\\
 & +\langle|\varphi'|^{-1}\partial_{t}^{2}\overline{v},\tilde{\Delta}_{\omega}\overline{v}\rangle+\tau^{2}\langle\varphi'\partial_{t}^{2}\overline{v},\overline{v}\rangle+\tau^{2}\langle\varphi'\tilde{\Delta}_{\omega}\overline{v},\overline{v}\rangle.
\end{align*}
Hence, by using integration by parts, and apply Lemma \ref{lem:ellip}
on the term $\delta\||\varphi'|^{-\frac{1}{2}}\tilde{\Delta}_{\omega}\overline{v}\|^{2}$,
choose $\delta>0$ small, and then choose $\epsilon>0$ small, we
reach 
\begin{align}
\mathscr{S}\ge & \frac{19}{10}\||\varphi'|^{-\frac{1}{2}}\partial_{t}^{2}\overline{v}\|^{2}+\frac{19}{10}\sum_{j=1}^{n+1}\||\varphi'|^{-\frac{1}{2}}\omega_{n+1}^{\frac{1-2s}{2}}\Omega_{j}\omega_{n+1}^{\frac{2s-1}{2}}\partial_{t}\overline{v}\|^{2}\nonumber \\
 & +c_{1}\sum_{(j,k)\neq(n+1,n+1)}\||\varphi'|^{-\frac{1}{2}}\omega_{n+1}^{\frac{1-2s}{2}}\Omega_{j}\Omega_{k}\omega_{n+1}^{\frac{2s-1}{2}}\overline{v}\|^{2}+\frac{18}{10}\||\varphi'|^{-\frac{1}{2}}\tilde{\Delta}_{\omega}\overline{v}\|^{2}\nonumber \\
 & +\frac{9}{10}\tau^{4}\||\varphi'|^{\frac{3}{2}}\overline{v}\|^{2}+\frac{39}{10}\tau^{2}\||\varphi'|^{\frac{1}{2}}\partial_{t}\overline{v}\|^{2}-\frac{11}{10}\tau^{2}\sum_{j=1}^{n+1}\||\varphi'|^{\frac{1}{2}}\omega_{n+1}^{\frac{1-2s}{2}}\Omega_{j}\omega_{n+1}^{\frac{2s-1}{2}}\overline{v}\|^{2}\nonumber \\
 & -C\|(|\tilde{V}|^{\frac{1}{2}}+|\partial_{t}\tilde{V}|^{\frac{1}{2}}+|\nabla_{\omega}'\tilde{V}|^{\frac{1}{2}})|\varphi'|^{-\frac{1}{2}}\partial_{t}\omega_{n+1}^{\frac{2s-1}{2}}\overline{v}\|_{0}^{2}\nonumber \\
 & -C\|(|\tilde{V}|^{\frac{1}{2}}+|\partial_{t}\tilde{V}|^{\frac{1}{2}}+|\nabla_{\omega}'\tilde{V}|^{\frac{1}{2}})|\varphi'|^{-\frac{1}{2}}\nabla_{\omega}'\omega_{n+1}^{\frac{2s-1}{2}}\overline{v}\|_{0}^{2}\nonumber \\
 & -C\|(|\tilde{V}|^{\frac{1}{2}}+|\partial_{t}\tilde{V}|^{\frac{1}{2}}+|\nabla_{\omega}'\tilde{V}|^{\frac{1}{2}})|\varphi'|^{\frac{1}{2}}\omega_{n+1}^{\frac{2s-1}{2}}\overline{v}\|_{0}^{2}.\label{eq:Carl-sum}
\end{align}

\textbf{Step 6: Combining the difference $\mathscr{D}$ and the sum $\mathscr{S}$.} Multiplying \eqref{eq:Carl-difference} by $\tau$, and summing with \eqref{eq:Carl-sum},
we reach 
\begin{align}
 & (\tau+1)\|L^{+}\overline{v}\|^{2}\ge\tau\mathscr{D}+\mathscr{S}\nonumber \\
\ge & c_{1}\bigg(\||\varphi'|^{-\frac{1}{2}}\partial_{t}^{2}\overline{v}\|^{2}+\sum_{j=1}^{n+1}\||\varphi'|^{-\frac{1}{2}}\omega_{n+1}^{\frac{1-2s}{2}}\Omega_{j}\omega_{n+1}^{\frac{2s-1}{2}}\partial_{t}\overline{v}\|^{2}\nonumber \\
 & \qquad+\sum_{(j,k)\neq(n+1,n+1)}\||\varphi'|^{-\frac{1}{2}}\omega_{n+1}^{\frac{1-2s}{2}}\Omega_{j}\Omega_{k}\omega_{n+1}^{\frac{2s-1}{2}}\overline{v}\|^{2}\bigg)\nonumber \\
 & +\frac{39}{5}\tau^{2}\||\varphi'|^{\frac{1}{2}}\partial_{t}\overline{v}\|^{2}+\frac{208}{10}\tau^{4}\|\varphi'|\varphi''|^{\frac{1}{2}}\overline{v}\|^{2}-\frac{11}{10}\tau^{2}\sum_{j=1}^{n+1}\||\varphi'|^{\frac{1}{2}}\omega_{n+1}^{\frac{1-2s}{2}}\Omega_{j}\omega_{n+1}^{\frac{2s-1}{2}}\overline{v}\|^{2}\nonumber \\
 & +\frac{18}{10}\||\varphi'|^{-\frac{1}{2}}\tilde{\Delta}_{\omega}\overline{v}\|^{2}\nonumber \\
 & -C\|(|\tilde{V}|^{\frac{1}{2}}+|\partial_{t}\tilde{V}|^{\frac{1}{2}}+|\nabla_{\omega}'\tilde{V}|^{\frac{1}{2}})|\varphi'|^{-\frac{1}{2}}\partial_{t}\omega_{n+1}^{\frac{2s-1}{2}}\overline{v}\|_{0}^{2}\nonumber \\
 & -C\|(|\tilde{V}|^{\frac{1}{2}}+|\partial_{t}\tilde{V}|^{\frac{1}{2}}+|\nabla_{\omega}'\tilde{V}|^{\frac{1}{2}})|\varphi'|^{-\frac{1}{2}}\nabla_{\omega}'\omega_{n+1}^{\frac{2s-1}{2}}\overline{v}\|_{0}^{2}\nonumber \\
 & -C\tau^{2}\|(|\tilde{V}|^{\frac{1}{2}}+|\partial_{t}\tilde{V}|^{\frac{1}{2}}+|\nabla_{\omega}'\tilde{V}|^{\frac{1}{2}})|\varphi''|^{\frac{1}{2}}\omega_{n+1}^{\frac{2s-1}{2}}\overline{v}\|_{0}^{2}.\label{eq:Carl4}
\end{align}

\textbf{Step 7: Obtaining gradient estimates.} Note that 
\begin{align}
 & \frac{12}{10}\tau^{2}\sum_{j=1}^{n+1}\||\varphi'|^{\frac{1}{2}}\omega_{n+1}^{\frac{1-2s}{2}}\Omega_{j}\omega_{n+1}^{\frac{2s-1}{2}}\overline{v}\|^{2}+\frac{12}{10}\tau^{2}\langle\tilde{V}\omega_{n+1}^{\frac{2s-1}{2}}\overline{v},\varphi'\omega_{n+1}^{\frac{2s-1}{2}}\overline{v}\rangle_{0}\nonumber \\
= & -\frac{12}{10}\tau^{2}\langle\varphi'\overline{v},\tilde{\Delta}_{\omega}\overline{v}\rangle\le\frac{16}{10}\||\varphi'|^{-\frac{1}{2}}\tilde{\Delta}_{\omega}\overline{v}\|^{2}+\frac{144}{100}\tau^{4}\||\varphi'|^{\frac{3}{2}}\overline{v}\|^{2}.\label{eq:Carl5}
\end{align}

\textbf{Step 8: Conclusion.} Summing up \eqref{eq:Carl4} and \eqref{eq:Carl5}, we reach 
\begin{align}
 & \||\varphi'|^{-\frac{1}{2}}\partial_{t}^{2}\overline{v}\|^{2}+\sum_{j=1}^{n+1}\||\varphi'|^{-\frac{1}{2}}\omega_{n+1}^{\frac{1-2s}{2}}\Omega_{j}\omega_{n+1}^{\frac{2s-1}{2}}\partial_{t}\overline{v}\|^{2}\nonumber \\
 & +\sum_{(j,k)\neq(n+1,n+1)}\||\varphi'|^{-\frac{1}{2}}\omega_{n+1}^{\frac{1-2s}{2}}\Omega_{j}\Omega_{k}\omega_{n+1}^{\frac{2s-1}{2}}\overline{v}\|^{2}\nonumber \\
 & +\tau^{2}\||\varphi'|^{\frac{1}{2}}\partial_{t}\overline{v}\|^{2}+\tau^{2}\sum_{j=1}^{n+1}\||\varphi'|^{\frac{1}{2}}\omega_{n+1}^{\frac{1-2s}{2}}\Omega_{j}\omega_{n+1}^{\frac{2s-1}{2}}\overline{v}\|^{2}+\tau^{4}\|\varphi'|\varphi''|^{\frac{1}{2}}\overline{v}\|^{2}\nonumber \\
\le & C\tau\|\tilde{f}\|^{2}+C\|(|\tilde{V}|^{\frac{1}{2}}+|\partial_{t}\tilde{V}|^{\frac{1}{2}}+|\nabla_{\omega}'\tilde{V}|^{\frac{1}{2}})|\varphi'|^{-\frac{1}{2}}\partial_{t}\omega_{n+1}^{\frac{2s-1}{2}}\overline{v}\|_{0}^{2}\nonumber \\
 & +C\|(|\tilde{V}|^{\frac{1}{2}}+|\partial_{t}\tilde{V}|^{\frac{1}{2}}+|\nabla_{\omega}'\tilde{V}|^{\frac{1}{2}})|\varphi'|^{-\frac{1}{2}}\nabla_{\omega}'\omega_{n+1}^{\frac{2s-1}{2}}\overline{v}\|_{0}^{2}\nonumber \\
 & +C\tau^{2}\|(|\tilde{V}|^{\frac{1}{2}}+|\partial_{t}\tilde{V}|^{\frac{1}{2}}+|\nabla_{\omega}'\tilde{V}|^{\frac{1}{2}})|\varphi''|^{\frac{1}{2}}\omega_{n+1}^{\frac{2s-1}{2}}\overline{v}\|_{0}^{2}.\label{eq:Carl6}
\end{align}
Changing back to the Cartesian coordinate, and we obtain our result. 
\end{proof}

\subsection{A Carleman estimate without differentiabiliy assumptions}

Imitating the splitting arguments in \cite[Theorem~5]{RW19Landis}, we can prove the following Carleman estimate. 
\begin{thm}
\label{thm:Carl-ndiff} Let $s\in(0,1)$ and let $\tilde{u}\in H^{1}(\mathbb{R}_{+}^{n+1},x_{n+1}^{1-2s})$
with ${\rm supp}(\tilde{u})\subset\mathbb{R}_{+}^{n+1}\setminus B_{1}^{+}$
be a solution to 
\begin{align*}
\bigg[\partial_{n+1}x_{n+1}^{1-2s}\partial_{n+1}+x_{n+1}^{1-2s}\sum_{j,k=1}^{n}a_{jk}\partial_{j}\partial_{k}\bigg]\tilde{u} & =f\quad\text{in }\mathbb{R}_{+}^{n+1},\\
\lim_{x_{n+1}\rightarrow0}x_{n+1}^{1-2s}\partial_{n+1}\tilde{u} & =V\tilde{u}\quad\text{on }\mathbb{R}^{n}\times\{0\},
\end{align*}
where $x=(x',x_{n+1})\in\mathbb{R}^{n}\times\mathbb{R}_{+}$, $f\in L^{2}(\mathbb{R}_{+}^{n+1},x_{n+1}^{2s-1})$
with compact support in $\mathbb{R}_{+}^{n+1}$, and $V\in L^{\infty}(\mathbb{R}^{n})$.
Assume that 
\[
\max_{1\le j,k\le n}\sup_{|x'|\ge1}|a_{jk}(x')-\delta_{jk}(x')|+\max_{1\le j,k\le n}\sup_{|x'|\ge1}|x'||\nabla'a_{jk}(x')|\le\epsilon
\]
for some sufficiently small $\epsilon>0$. Let further $\phi(x)=|x|^{\alpha}$
for $\alpha\ge1$. Then there exist constants $C=C(n,s,\alpha)$
and $\tau_{0}=\tau_{0}(n,s,\alpha)$ such that 
\begin{align*}
 & \tau^{3}\|e^{\tau\phi}|x|^{\frac{3\alpha}{2}-1}x_{n+1}^{\frac{1-2s}{2}}\tilde{u}\|_{L^{2}(\mathbb{R}_{+}^{n+1})}^{2}+\tau\|e^{\tau\phi}|x|^{\frac{\alpha}{2}}x_{n+1}^{\frac{1-2s}{2}}\nabla\tilde{u}\|_{L^{2}(\mathbb{R}_{+}^{n+1})}^{2}\\
\le & C\bigg[\|e^{\tau\phi}x_{n+1}^{\frac{2s-1}{2}}|x|f\|_{L^{2}(\mathbb{R}_{+}^{n+1})}^{2}+\tau^{2-2s}\|e^{\tau\phi}V|x|^{(1-\alpha)s}\tilde{u}\|_{L^{2}(\mathbb{R}^{n}\times\{0\})}\bigg]
\end{align*}
for all $\tau\ge\tau_{0}$. 
\end{thm}

\begin{proof}
[Proof of Theorem~{\rm \ref{thm:Carl-ndiff}}] \textbf{Step 1: Changing the coordinates.} As in the proof of Theorem
\ref{thm:Carl-diff}, firstly, we pass to conformal coordinates. With
the notations mentioned before, recall \eqref{eq:Carl1}: 
\[
\bigg[\omega_{n+1}^{1-2s}\partial_{t}^{2}+\sum_{j=1}^{n+1}\Omega_{j}\omega_{n+1}^{1-2s}\Omega_{j}-\omega_{n+1}^{1-2s}\frac{(n-2s)^{2}}{4}\bigg]\overline{u}+R\overline{u}=\tilde{f},
\]
where 
\begin{align*}
R= & \omega_{n+1}^{1-2s}\sum_{j,k=1}^{n}(a_{jk}-\delta_{jk})\bigg[\omega_{j}\omega_{k}\partial_{t}^{2}+\omega_{j}\Omega_{k}\partial_{t}+\omega_{k}\Omega_{j}\partial_{t}+\frac{1}{2}\Omega_{j}\Omega_{k}+\frac{1}{2}\Omega_{k}\Omega_{j}\bigg]\\
 & +\omega_{n+1}^{1-2s}\sum_{j,k=1}^{n}(a_{jk}-\delta_{jk})\bigg[(\delta_{jk}-(n+2-2s)\omega_{j}\omega_{k})\partial_{t}\\
 & \qquad-\frac{n+1-2s}{2}\omega_{j}\Omega_{k}-\frac{n+1-2s}{2}\omega_{k}\Omega_{j}\bigg]\\
 & +\omega_{n+1}^{1-2s}\sum_{j,k=1}^{n}(a_{jk}-\delta_{jk})\bigg[\frac{(n-2s)^{2}}{4}\omega_{j}\omega_{k}-\frac{n-2s}{2}(\delta_{jk}-2\omega_{j}\omega_{k})\bigg].
\end{align*}

\textbf{Step 2: Splitting $\overline{u}$ into elliptic and subelliptic parts.} We split $\overline{u}$ into two parts $\overline{u}=u_{1}+u_{2}$.
Here $u_{1}$ is a solution to 
\begin{align}
\bigg[\omega_{n+1}^{1-2s}\partial_{t}^{2}+\sum_{j=1}^{n+1}\Omega_{j}\omega_{n+1}^{1-2s}\Omega_{j}-\omega_{n+1}^{1-2s}\frac{(n-2s)^{2}}{4}-K^{2}\tau^{2}|\varphi'|^{2}\omega_{n+1}^{1-2s}\bigg]u_{1}+Ru_{1} & =\tilde{f}\quad\text{in }\mathcal{S}_{+}^{n}\times\mathbb{R},\label{eq:Carl-ndiff1}\\
\lim_{\omega_{n+1}\rightarrow0}\omega_{n+1}^{1-2s}\Omega_{n+1}u_{1}=\lim_{\omega_{n+1}\rightarrow0}\omega_{n+1}^{1-2s}\Omega_{n+1}\overline{u} & \quad\text{on }\partial\mathcal{S}_{+}^{n}\times\mathbb{R}.\nonumber 
\end{align}
We remark that existence of unique energy solution to this problem is followed by the Lax-Milgram theorem in $H^{1}(\mathcal{S}_{+}^{n}\times\mathbb{R},\omega_{n+1}^{1-2s})$. 

\textbf{Step 2.1: Obtain an elliptic estimate.} Testing $\tau^{2}e^{2\tau\varphi}|\varphi''|^{2}u_{1}$ in \eqref{eq:Carl-ndiff1}, for $\delta>0$, we reach 
\begin{align*}
 & \tau^{2}\|e^{\tau\varphi}\varphi''\omega_{n+1}^{\frac{1-2s}{2}}\partial_{t}u_{1}\|^{2}+\tau^{2}\|e^{\tau\varphi}\varphi''\omega_{n+1}^{\frac{1-2s}{2}}\nabla_{\mathcal{S}^{n}}u_{1}\|^{2}+\tau^{2}\frac{(n-2s)^{2}}{4}\|e^{\tau\varphi}\varphi''\omega_{n+1}^{\frac{1-2s}{2}}u_{1}\|^{2}\\
 & +K^{2}\tau^{4}\|e^{\tau\varphi}\varphi'\varphi''\omega_{n+1}^{\frac{1-2s}{2}}u_{1}\|^{2}\\
= & -\tau^{2}\langle e^{\tau\varphi}\omega_{n+1}^{\frac{2s-1}{2}}\tilde{f},e^{\tau\varphi}|\varphi''|^{2}\omega_{n+1}^{\frac{1-2s}{2}}u_{1}\rangle-\langle\tau e^{\tau\varphi}\varphi'''\omega_{n+1}^{\frac{1-2s}{2}}\partial_{t}u_{1},\tau^{2}e^{\tau\varphi}\varphi'\varphi''\omega_{n+1}^{\frac{1-2s}{2}}u_{1}\rangle\\
 & +\tau^{2}\langle Ru_{1},e^{2\tau\varphi}|\varphi''|^{2}u_{1}\rangle-2\langle\tau e^{\tau\varphi}\varphi''\omega_{n+1}^{\frac{1-2s}{2}}\partial_{t}u_{1},\tau e^{\tau\varphi}\varphi'''\omega_{n+1}^{\frac{1-2s}{2}}u_{1}\rangle\\
 & -\tau^{2}\langle e^{\tau\varphi}\varphi''e^{\alpha st}\omega_{n+1}^{1-2s}\Omega_{n+1}u_{1},e^{\tau\varphi}\varphi'\varphi''\omega_{n+1}^{\frac{1-2s}{2}}u_{1}\rangle_{0}\\
\le & \|e^{\tau\varphi}\omega_{n+1}^{\frac{2s-1}{2}}\tilde{f}\|^{2}+\tau^{4}\|e^{\tau\varphi}|\varphi''|^{2}\omega_{n+1}^{\frac{1-2s}{2}}u_{1}\|^{2}+\delta\tau^{2}\|e^{\tau\varphi}\varphi'''\omega_{n+1}^{\frac{1-2s}{2}}\partial_{t}u_{1}\|^{2}\\
 & +C_{\delta}\tau^{4}\|e^{\tau\varphi}\varphi'\varphi''\omega_{n+1}^{\frac{1-2s}{2}}u_{1}\|^{2}+\delta\tau^{2}\|e^{\tau\varphi}\varphi'''\omega_{n+1}^{\frac{1-2s}{2}}\partial_{t}u_{1}\|^{2}+C_{\delta}\tau^{2}\|e^{\tau\varphi}\varphi'''\omega_{n+1}^{\frac{1-2s}{2}}u_{1}\|^{2}\\
 & +\tau^{2}\|e^{\tau\varphi}\varphi''e^{\alpha st}\omega_{n+1}^{1-2s}\Omega_{n+1}\overline{u}\|_{0}\|e^{\tau\varphi}\varphi'\varphi''\omega_{n+1}^{\frac{1-2s}{2}}u_{1}\|_{0}+\tau^{2}|\langle Ru_{1},e^{2\tau\varphi}|\varphi''|^{2}u_{1}\rangle|.
\end{align*}
Firstly, we choose small $\delta>0$ and small $\epsilon>0$, then choose large $K>1$, so 
\begin{align}
 & \tau^{2}\|e^{\tau\varphi}\varphi''\omega_{n+1}^{\frac{1-2s}{2}}\partial_{t}u_{1}\|^{2}+\tau^{2}\|e^{\tau\varphi}\varphi''\omega_{n+1}^{\frac{1-2s}{2}}\nabla_{\mathcal{S}^{n}}u_{1}\|^{2}+\tau^{4}\|e^{\tau\varphi}\varphi'\varphi''\omega_{n+1}^{\frac{1-2s}{2}}u_{1}\|^{2}\nonumber \\
\le & C\|e^{\tau\varphi}\omega_{n+1}^{\frac{2s-1}{2}}\tilde{f}\|^{2}+C\tau^{2}\|e^{\tau\varphi}\varphi''e^{\alpha st}\omega_{n+1}^{1-2s}\Omega_{n+1}\overline{u}\|_{0}\|e^{\tau\varphi}\varphi'\varphi''\omega_{n+1}^{\frac{1-2s}{2}}u_{1}\|_{0}\nonumber \\
\le & C\|e^{\tau\varphi}\omega_{n+1}^{\frac{2s-1}{2}}\tilde{f}\|^{2}+C_{\eta}\tau^{2-2s}\|e^{\tau\varphi}\varphi''e^{\alpha st}\omega_{n+1}^{1-2s}\Omega_{n+1}\overline{u}\|_{0}+\eta\tau^{2+2s}\|e^{\tau\varphi}\varphi'\varphi''\omega_{n+1}^{\frac{1-2s}{2}}u_{1}\|_{0}.\label{eq:Carl-ndiff2}
\end{align}
From Proposition \ref{prop:inter1}, we have 
\[
|\varphi''|^{2}e^{2\alpha st}\int_{\partial\mathcal{S}_{+}^{n}}u_{1}^{2}\le C\tilde{\tau}^{2-2s}|\varphi''|^{2}e^{2\alpha st}\int_{\mathcal{S}_{+}^{n}}\omega_{n+1}^{1-2s}u_{1}^{2}+C\tilde{\tau}^{-2s}|\varphi''|^{2}e^{2\alpha st}\int_{\mathcal{S}_{+}^{n}}\omega_{n+1}^{1-2s}|\nabla_{\mathcal{S}^{n}}u_{1}|^{2}.
\]
Choosing $\tilde{\tau}=e^{\alpha t}\tau$, we reach 
\[
|\varphi''|^{2}e^{2\alpha st}\int_{\partial\mathcal{S}_{+}^{n}}u_{1}^{2}\le C\tau^{2-2s}|\varphi''|^{2}e^{2\alpha t}\int_{\mathcal{S}_{+}^{n}}\omega_{n+1}^{1-2s}u_{1}^{2}+C\tau^{-2s}|\varphi''|^{2}\int_{\mathcal{S}_{+}^{n}}\omega_{n+1}^{1-2s}|\nabla_{\mathcal{S}^{n}}u_{1}|^{2}.
\]
Multiplying with $e^{2\tau\varphi}$, using that $\varphi'=\alpha e^{\alpha t}$
and integrating in the radial direction, thus implies 
\[
\tau^{2+2s}\|e^{\tau\varphi}|\varphi''|e^{\alpha st}u_{1}\|_{0}^{2}\le C\tau^{4}\|e^{\tau\varphi}\omega_{n+1}^{\frac{1-2s}{2}}\varphi'\varphi''u_{1}\|^{2}+C\tau^{2}\|e^{\tau\varphi}\omega_{n+1}^{\frac{1-2s}{2}}\varphi''\nabla_{\mathcal{S}^{n}}u_{1}\|^{2}.
\]
Plug the inequality above into \eqref{eq:Carl-ndiff2}, and choose $\eta>0$ small, so 
\begin{align}
 & \tau^{2}\|e^{\tau\varphi}\varphi''\omega_{n+1}^{\frac{1-2s}{2}}\partial_{t}u_{1}\|^{2}+\tau^{2}\|e^{\tau\varphi}\varphi''\omega_{n+1}^{\frac{1-2s}{2}}\nabla_{\mathcal{S}^{n}}u_{1}\|^{2}+\tau^{4}\|e^{\tau\varphi}\varphi'\varphi''\omega_{n+1}^{\frac{1-2s}{2}}u_{1}\|^{2}\nonumber \\
\le & C\|e^{\tau\varphi}\omega_{n+1}^{\frac{2s-1}{2}}\tilde{f}\|^{2}+C\tau^{2-2s}\|e^{\tau\varphi}\varphi''e^{\alpha st}\omega_{n+1}^{1-2s}\Omega_{n+1}\overline{u}\|_{0}.\label{eq:Carl-ndiff3}
\end{align}

\textbf{Step 2.2: Obtaining a sub-elliptic estimate.} Indeed, $u_{2}$ satisfies 
\begin{align*}
\bigg[\omega_{n+1}^{1-2s}\partial_{t}^{2}+\sum_{j=1}^{n+1}\Omega_{j}\omega_{n+1}^{1-2s}\Omega_{j}-\omega_{n+1}^{1-2s}\frac{(n-2s)^{2}}{4}\bigg]u_{2}+Ru_{2}=-K^{2}\tau^{2}|\varphi'|^{2}\theta_{n+1}^{1-2s}u_{1} & \quad\text{in }\mathcal{S}_{+}^{n}\times\mathbb{R},\\
\lim_{\omega_{n+1}\rightarrow0}\omega_{n+1}^{1-2s}\Omega_{n+1}u_{2}=0 & \quad\text{on }\partial\mathcal{S}_{+}^{n}\times\mathbb{R}.
\end{align*}
To compare with \eqref{eq:Carl1}, we should put 
\[
\tilde{f}=-K^{2}\tau|\varphi'|^{2}\omega_{n+1}^{1-2s}u_{1}\quad\text{and}\quad\tilde{V}\equiv0
\]
in \eqref{eq:Carl6}. Omitting the second derivative terms, we obtain
\[
\tau^{3}\|\varphi'|\varphi''|^{\frac{1}{2}}\overline{v}\|^{2}+\tau\||\varphi''|^{\frac{1}{2}}\partial_{t}\overline{v}\|^{2}+\tau\||\varphi'|^{\frac{1}{2}}\omega_{n+1}^{\frac{1-2s}{2}}\nabla_{\mathcal{S}^{n}}\omega_{n+1}^{\frac{2s-1}{2}}\overline{v}\|^{2}\le CK^{4}\tau^{4}\|e^{\tau\varphi}|\varphi'|^{2}\omega_{n+1}^{\frac{1-2s}{2}}u_{1}\|^{2},
\]
that is, 
\begin{align}
 & \tau^{3}\|e^{\tau\varphi}\varphi'|\varphi''|^{\frac{1}{2}}\omega_{n+1}^{\frac{1-2s}{2}}u_{2}\|^{2}+\tau\|e^{\tau\varphi}|\varphi''|^{\frac{1}{2}}\omega_{n+1}^{\frac{1-2s}{2}}\partial_{t}u_{2}\|^{2}+\tau\|e^{\tau\varphi}|\varphi''|^{\frac{1}{2}}\omega_{n+1}^{\frac{1-2s}{2}}\nabla_{\mathcal{S}^{n}}u_{2}\|^{2}\nonumber \\
\le & CK^{4}\tau^{4}\|e^{\tau\varphi}|\varphi'|^{2}\omega_{n+1}^{\frac{1-2s}{2}}u_{1}\|^{2}.\label{eq:Carl-ndiff4}
\end{align}

\textbf{Step 3: Conclusion.} Summing up \eqref{eq:Carl-ndiff3} and \eqref{eq:Carl-ndiff4}, since
$\overline{u}=u_{1}+u_{2}$, so  
\begin{align}
 & \tau^{3}\|e^{\tau\varphi}\varphi'|\varphi''|^{\frac{1}{2}}\omega_{n+1}^{\frac{1-2s}{2}}\overline{u}\|^{2}+\tau\|e^{\tau\varphi}|\varphi''|^{\frac{1}{2}}\omega_{n+1}^{\frac{1-2s}{2}}\partial_{t}\overline{u}\|^{2}+\tau\|e^{\tau\varphi}|\varphi''|^{\frac{1}{2}}\omega_{n+1}^{\frac{1-2s}{2}}\nabla_{\mathcal{S}^{n}}\overline{u}\|^{2} \nonumber \\
\le & C\bigg[\|e^{\tau\varphi}\omega_{n+1}^{\frac{2s-1}{2}}\tilde{f}\|^{2}+\tau^{2-2s}\|e^{\tau\varphi}\varphi''e^{\alpha st}\omega_{n+1}^{1-2s}\Omega_{n+1}\overline{u}\|_{0}^{2}\bigg]. \label{eq:carl-ndff-polar}
\end{align}
Finally, plug in the boundary condition 
\[
\lim_{\omega_{n+1}\rightarrow0}\omega_{n+1}^{1-2s}\Omega_{n+1}\overline{u}=\tilde{V}\overline{u},
\]
and switch back to the Cartesian coordinate, we obtain our result. 
\end{proof}

\section{\label{sec:result}Proofs of Theorem \ref{thm:result1} and Theorem \ref{thm:result2}}
\begin{proof}
[Proof of Theorem~{\rm \ref{thm:result1}}] \textbf{Step 1: Applying Carleman estimate.} Define $w:=\eta_{R}\tilde{u}$,
where $\eta_{R}$ is radial, 
\begin{equation}
\eta_{R}(x)=\begin{cases}
1 & ,2\le|x|\le R,\\
0 & ,|x|\le1\text{ or }|x|\ge2R,
\end{cases} \label{eq:eta-R-cutoff}
\end{equation}
and satisfies $|\nabla\eta_{R}|\le C/R$, $|\nabla^{2}\eta_{R}|\le C/R^{2}$
in $A_{R,2R}^{+}$,
\begin{align*}
|\nabla\eta_{R}| & \le C/R,\quad|\nabla^{2}\eta_{R}|\le C/R^{2}\quad\text{in }A_{R,2R}^{+},\\
|\nabla\eta_{R}| & \le C,\quad|\nabla^{2}\eta_{R}|\le C\quad\text{in }A_{1,2}^{+}.
\end{align*}
Note that 
\[
\bigg[\partial_{n+1}x_{n+1}^{1-2s}\partial_{n+1}+x_{n+1}^{1-2s}\sum_{j,k=1}^{n}a_{jk}\partial_{j}\partial_{k}\bigg]w=f,
\]
where 
\begin{align*}
f= & x_{n+1}^{1-2s}\bigg[(1-2s)x_{n+1}^{-1}\partial_{n+1}\eta_{R}\bigg]\tilde{u}+x_{n+1}^{1-2s}\bigg[\partial_{n+1}^{2}\eta_{R}+\sum_{j,k=1}^{n}a_{jk}\partial_{j}\partial_{k}\eta_{R}\bigg]\tilde{u}\\
 & +2x_{n+1}^{1-2s}\bigg[(\partial_{n+1}\eta_{R})(\partial_{n+1}\tilde{u})+\sum_{j,k=1}^{n}a_{jk}(\partial_{k}\eta_{R})(\partial_{j}\tilde{u})\bigg]\\
 & -x_{n+1}^{1-2s}\sum_{j,k=1}^{n}(\partial_{j}a_{jk})(\partial_{k}\tilde{u})\eta_{R}.
\end{align*}
Since $\eta_{R}$ is radial, then $\partial_{n+1}\eta_{R}=\eta_{R}'\partial_{n+1}|x|=0$
on $\mathbb{R}^{n}\times\{0\}$. Thus, 
\[
\lim_{x_{n+1}\rightarrow0}x_{n+1}^{1-2s}\partial_{n+1}w=\lim_{x_{n+1}\rightarrow0}x_{n+1}^{1-2s}\eta_{R}\partial_{n+1}\tilde{u}=c_{n,s}^{-1}q\eta_{R}u=c_{n,s}^{-1}qw\quad\text{on }\mathbb{R}^{n}\times\{0\}.
\]
Note that $w$ is admissible in the Carleman estimate in Theorem \ref{thm:Carl-diff}.
For $\beta>1$, since $|q|\le1$ and $|x'||\nabla' q|\le1$, we have
\begin{align}
 & \tau^{3}\|e^{\tau\phi}|x|^{\frac{3\beta}{2}-1}x_{n+1}^{\frac{1-2s}{2}}w\|_{L^{2}(\mathbb{R}_{+}^{n+1})}^{2}+\tau\|e^{\tau\phi}|x|^{\frac{\beta}{2}}x_{n+1}^{\frac{1-2s}{2}}\nabla w\|_{L^{2}(\mathbb{R}_{+}^{n+1})}^{2}\nonumber \\
 & +\tau^{-1}\|e^{\tau\phi}|x|^{-\frac{\beta}{2}+1}x_{n+1}^{\frac{1-2s}{2}}\nabla(\nabla'w)\|_{L^{2}(\mathbb{R}_{+}^{n+1})}^{2}\nonumber \\
\le & C\bigg[\|e^{\tau\phi}x_{n+1}^{\frac{2s-1}{2}}|x|f\|_{L^{2}(\mathbb{R}_{+}^{n+1})}^{2}+\tau\|e^{\tau\phi}|x|^{\frac{\beta}{2}}w\|_{L^{2}(\mathbb{R}^{n}\times\{0\})}^{2}\nonumber \\
 & \qquad+\tau^{-1}\|e^{\tau\phi}|x|^{-\frac{\beta}{2}+1}\nabla'w\|_{L^{2}(\mathbb{R}^{n}\times\{0\})}^{2}\bigg].\label{eq:pf1}
\end{align}

\textbf{Step 2: Estimating the bulk contributions.} Since $1\le\frac{|x|}{R}$ in $A_{R,2R}^{+}$ and $1\le|x|$ in $A_{1,2}^{+}$,
then 
\begin{align*}
 & \|e^{\tau\phi}x_{n+1}^{\frac{2s-1}{2}}|x|f\|_{L^{2}(\mathbb{R}_{+}^{n+1})}^{2}\\
\le & C\bigg[R^{-4}\|e^{\tau\phi}x_{n+1}^{\frac{1-2s}{2}}|x|\tilde{u}\|_{L^{2}(A_{R,2R}^{+})}^{2}+R^{-2}\|e^{\tau\phi}x_{n+1}^{\frac{1-2s}{2}}|x|\nabla\tilde{u}\|_{L^{2}(A_{R,2R}^{+})}^{2}\\
 & +\|e^{\tau\phi}x_{n+1}^{\frac{1-2s}{2}}|x|\tilde{u}\|_{L^{2}(A_{1,2}^{+})}^{2}+\|e^{\tau\phi}x_{n+1}^{\frac{1-2s}{2}}|x|\nabla\tilde{u}\|_{L^{2}(A_{1,2}^{+})}^{2}+\|e^{\tau\phi}x_{n+1}^{\frac{1-2s}{2}}\nabla w\|_{L^{2}(\mathbb{R}_{+}^{n+1})}^{2}\bigg].
\end{align*}
Write $\tilde{\phi}(r)=\phi(x)=r^{\beta}$ with $r=|x|$, note that
\begin{align*}
 & R^{-4}\|e^{\tau\phi}x_{n+1}^{\frac{1-2s}{2}}|x|\tilde{u}\|_{L^{2}(A_{R,2R}^{+})}^{2}+R^{-2}\|e^{\tau\phi}x_{n+1}^{\frac{1-2s}{2}}|x|\nabla\tilde{u}\|_{L^{2}(A_{R,2R}^{+})}^{2}\\
\le & C\bigg[R^{-2}e^{\tau\tilde{\phi}(2R)}\|x_{n+1}^{\frac{1-2s}{2}}\tilde{u}\|_{L^{2}(A_{R,2R}^{+})}^{2}+e^{\tau\tilde{\phi}(2R)}\|x_{n+1}^{\frac{1-2s}{2}}\nabla\tilde{u}\|_{L^{2}(A_{R,2R}^{+})}^{2}\bigg].
\end{align*}
Now we estimate $\|x_{n+1}^{\frac{1-2s}{2}}\nabla\tilde{u}\|_{L^{2}(A_{R,2R}^{+})}^{2}$.
Choose $\xi_{R}$ satisfies 
\[
\xi_{R}(x)=\begin{cases}
1 & ,R\le|x|\le2R,\\
0 & ,|x|\le\frac{R}{2}\text{ or }|x|\ge2R,
\end{cases}
\]
with $|\nabla\xi_{R}|\le C/R$ for $x\in A_{\frac{R}{2},R}^{+}$
or $x\in A_{2R,3R}^{+}$. Test $\partial_{n+1}x_{n+1}^{1-2s}\partial_{n+1}\tilde{u}+\sum_{j,k=1}^{n}\partial_{j}a_{jk}\partial_{k}\tilde{u}=0$
by the function $\tilde{u}\xi_{R}^{2}$, we reach 
\[
\|x_{n+1}^{\frac{1-2s}{2}}\nabla\tilde{u}\|_{L^{2}(A_{R,2R}^{+})}^{2}\le C \bigg[ \|x_{n+1}^{\frac{1-2s}{2}}u\|_{L^{2}(A_{\frac{R}{2},3R}')}^{2}+R^{-2}\|x_{n+1}^{\frac{1-2s}{2}}\tilde{u}\|_{L^{2}(A_{\frac{R}{2},3R}^{+})}^{2} \bigg].
\]
So, 
\begin{align*}
 & R^{-4}\|e^{\tau\phi}x_{n+1}^{\frac{1-2s}{2}}|x|\tilde{u}\|_{L^{2}(A_{R,2R}^{+})}^{2}+R^{-2}\|e^{\tau\phi}x_{n+1}^{\frac{1-2s}{2}}|x|\nabla\tilde{u}\|_{L^{2}(A_{R,2R}^{+})}^{2}\\
\le & C e^{\tau\tilde{\phi}(2R)} \bigg[\|x_{n+1}^{\frac{1-2s}{2}}u\|_{L^{2}(A_{\frac{R}{2},3R}')}^{2}+R^{-2}\|x_{n+1}^{\frac{1-2s}{2}}\tilde{u}\|_{L^{2}(A_{R,2R}^{+})}^{2}\bigg].
\end{align*}
Using Proposition~\ref{prop:bulk-decay}, we have 
\[
|\tilde{u}(x)|\le C_{1}e^{-C_{2}R^{\alpha}}\quad\text{for }x\in A_{\frac{R}{2},3R}^{+}.
\]
So, if we choose $\beta=\alpha-\epsilon$ for some $\epsilon\in(0,\alpha-1)$, then we have 
\[
\lim_{R\rightarrow\infty}(R^{-4}\|e^{\tau\phi}x_{n+1}^{\frac{1-2s}{2}}|x|\tilde{u}\|_{L^{2}(A_{R,2R}^{+})}^{2}+R^{-2}\|e^{\tau\phi}x_{n+1}^{\frac{1-2s}{2}}|x|\nabla\tilde{u}\|_{L^{2}(A_{R,2R}^{+})}^{2})=0.
\]
However, \eqref{eq:pf1} writes 
\begin{align*}
 & \tau^{3}\|e^{\tau\phi}|x|^{\frac{3\beta}{2}-1}x_{n+1}^{\frac{1-2s}{2}}w\|_{L^{2}(\mathbb{R}_{+}^{n+1})}^{2}+\tau\|e^{\tau\phi}|x|^{\frac{\beta}{2}}x_{n+1}^{\frac{1-2s}{2}}\nabla w\|_{L^{2}(\mathbb{R}_{+}^{n+1})}^{2}\\
 & +\tau^{-1}\|e^{\tau\phi}|x|^{-\frac{\beta}{2}+1}x_{n+1}^{\frac{1-2s}{2}}\nabla(\nabla'w)\|_{L^{2}(\mathbb{R}_{+}^{n+1})}^{2}\\
\le & C\bigg[R^{-4}\|e^{\tau\phi}x_{n+1}^{\frac{1-2s}{2}}|x|\tilde{u}\|_{L^{2}(A_{R,2R}^{+})}^{2}+R^{-2}\|e^{\tau\phi}x_{n+1}^{\frac{1-2s}{2}}|x|\nabla\tilde{u}\|_{L^{2}(A_{R,2R}^{+})}^{2}\\
 & +\|e^{\tau\phi}x_{n+1}^{\frac{1-2s}{2}}|x|\tilde{u}\|_{L^{2}(A_{1,2}^{+})}^{2}+\|e^{\tau\phi}x_{n+1}^{\frac{1-2s}{2}}|x|\nabla\tilde{u}\|_{L^{2}(A_{1,2}^{+})}^{2}+\|e^{\tau\phi}x_{n+1}^{\frac{1-2s}{2}}\nabla w\|_{L^{2}(\mathbb{R}_{+}^{n+1})}^{2}\\
 & +\tau\|e^{\tau\phi}|x|^{\frac{\beta}{2}}w\|_{L^{2}(\mathbb{R}^{n}\times\{0\})}^{2}+\tau^{-1}\|e^{\tau\phi}|x|^{-\frac{\beta}{2}+1}\nabla'w\|_{L^{2}(\mathbb{R}^{n}\times\{0\})}^{2}\bigg].
\end{align*}
Taking $R\rightarrow\infty$ in \eqref{eq:pf1} and choosing large
$\tau$, we reach 
\begin{align}
 & \tau^{3}\|e^{\tau\phi}|x|^{\frac{3\beta}{2}-1}x_{n+1}^{\frac{1-2s}{2}}w\|_{L^{2}(\mathbb{R}_{+}^{n+1})}^{2}+\tau\|e^{\tau\phi}|x|^{\frac{\beta}{2}}x_{n+1}^{\frac{1-2s}{2}}\nabla w\|_{L^{2}(\mathbb{R}_{+}^{n+1})}^{2}\nonumber \\
 & +\tau^{-1}\|e^{\tau\phi}|x|^{-\frac{\beta}{2}+1}x_{n+1}^{\frac{1-2s}{2}}\nabla(\nabla'w)\|_{L^{2}(\mathbb{R}_{+}^{n+1})}^{2}\nonumber \\
\le & C\bigg[\|e^{\tau\phi}x_{n+1}^{\frac{1-2s}{2}}|x|\tilde{u}\|_{L^{2}(A_{1,2}^{+})}^{2}+\|e^{\tau\phi}x_{n+1}^{\frac{1-2s}{2}}|x|\nabla\tilde{u}\|_{L^{2}(A_{1,2}^{+})}^{2}\nonumber \\
 & +\tau\|e^{\tau\phi}|x|^{\frac{\beta}{2}}w\|_{L^{2}(\mathbb{R}^{n}\times\{0\})}^{2}+\tau^{-1}\|e^{\tau\phi}|x|^{-\frac{\beta}{2}+1}\nabla'w\|_{L^{2}(\mathbb{R}^{n}\times\{0\})}^{2}\bigg].\label{eq:pf2}
\end{align}

\textbf{Step 3: Estimating the boundary contributions.} Using Proposition~\ref{prop:inter1}, we have 
\[
\tilde{\tau}|\varphi''|e^{2st}\|v\|_{L^{2}(\partial\mathcal{S}_{+}^{n})}\le C\bigg[\tilde{\tau}^{2-2s}e^{2st}|\varphi''|\|\omega_{n+1}^{\frac{1-2s}{2}}v\|_{L^{2}(\mathcal{S}_{+}^{n})}^{2}+\tilde{\tau}^{-2s}e^{2st}|\varphi''|\|\omega_{n+1}^{\frac{1-2s}{2}}\nabla_{\omega}v\|_{L^{2}(\mathcal{S}_{+}^{n})}^{2}\bigg].
\]
Setting $e^{2st}\tilde{\tau}^{-2s}=\tau^{-2s}$ (i.e. $\tilde{\tau}=\tau e^{t}$),
our choice of $\varphi$ gives 
\[
\tilde{\tau}^{2-2s}e^{2st}|\varphi''|=\tau^{2-2s}e^{2t}|\varphi''|\le\tau^{2-2s}|\varphi'|^{2}|\varphi''|.
\]
Hence, we reach 
\[
\tau^{2s+1}|\varphi''|\|v\|_{L^{2}(\partial\mathcal{S}_{+}^{n})}^{2}\le C\bigg[\tau^{3}|\varphi''||\varphi'|^{2}\|\omega_{n+1}^{\frac{1-2s}{2}}v\|_{L^{2}(\mathcal{S}_{+}^{n})}^{2}+\tau|\varphi''|\|\omega_{n+1}^{\frac{1-2s}{2}}\nabla_{\omega}v\|_{L^{2}(\mathcal{S}_{+}^{n})}^{2}\bigg].
\]
Multiplying the above inequality by $e^{\tau\varphi}$, and then integrating with respect to the radial variable
$t$, we obtain  
\begin{align*}
 & \tau^{2s+1}\|e^{\tau\varphi}|\varphi''|^{\frac{1}{2}}v\|_{L^{2}(\partial\mathcal{S}_{+}^{n}\times\mathbb{R})}^{2}\\
\le & C\bigg[\tau^{3}\|e^{\tau\varphi}\varphi'|\varphi''|^{\frac{1}{2}}\omega_{n+1}^{\frac{1-2s}{2}}v\|_{L^{2}(\mathcal{S}_{+}^{n}\times\mathbb{R})}^{2}+\tau\|e^{\tau\varphi}|\varphi''|^{\frac{1}{2}}\omega_{n+1}^{\frac{1-2s}{2}}\nabla_{\omega}v\|_{L^{2}(\mathcal{S}_{+}^{n}\times\mathbb{R})}^{2} \bigg],
\end{align*}
that is, 
\begin{align*}
 & \tau^{2s+1}\|e^{\tau\phi}|x|^{\frac{\beta}{2}}w\|_{L^{2}(\mathbb{R}^{n}\times\{0\})}^{2}\\
\le & C\bigg[\tau^{3}\|e^{\tau\phi}|x|^{\frac{3\beta}{2}-1}x_{n+1}^{\frac{1-2s}{2}}w\|_{L^{2}(\mathbb{R}_{+}^{n+1})}^{2}+\tau\|e^{\tau\phi}|x|^{\frac{\beta}{2}}x_{n+1}^{\frac{1-2s}{2}}\nabla w\|_{L^{2}(\mathbb{R}_{+}^{n+1})}^{2}\bigg].
\end{align*}
Similarly, we have 
\begin{align*}
 & \tau^{2s-1}\|e^{\tau\phi}|x|^{-\frac{\beta}{2}+1}\nabla'w\|_{L^{2}(\mathbb{R}^{n}\times\{0\})}^{2}\\
\le & C\bigg[\tau\|e^{\tau\phi}|x|^{\frac{\beta}{2}}x_{n+1}^{\frac{1-2s}{2}}\nabla'w\|_{L^{2}(\mathbb{R}_{+}^{n+1})}^{2}+\tau^{-1}\|e^{\tau\phi}|x|^{-\frac{\beta}{2}+1}x_{n+1}^{\frac{1-2s}{2}}\nabla(\nabla'w)\|_{L^{2}(\mathbb{R}_{+}^{n+1})}^{2}\bigg].
\end{align*}
So, for large $\tau$, the boundary terms of \eqref{eq:pf2} are absorbed,
and we reach 
\begin{align*}
 & \tau^{3}\|e^{\tau\phi}|x|^{\frac{3\beta}{2}-1}x_{n+1}^{\frac{1-2s}{2}}w\|_{L^{2}(B_{6}^{+}\setminus B_{4}^{+})}^{2}+\tau\|e^{\tau\phi}|x|^{\frac{\beta}{2}}x_{n+1}^{\frac{1-2s}{2}}\nabla w\|_{L^{2}(B_{6}^{+}\setminus B_{4}^{+})}^{2}\\
\le & \tau^{3}\|e^{\tau\phi}|x|^{\frac{3\beta}{2}-1}x_{n+1}^{\frac{1-2s}{2}}w\|_{L^{2}(\mathbb{R}_{+}^{n+1})}^{2}+\tau\|e^{\tau\phi}|x|^{\frac{\beta}{2}}x_{n+1}^{\frac{1-2s}{2}}\nabla w\|_{L^{2}(\mathbb{R}_{+}^{n+1})}^{2}\\
 & +\tau^{-1}\|e^{\tau\phi}|x|^{-\frac{\beta}{2}+1}x_{n+1}^{\frac{1-2s}{2}}\nabla(\nabla'w)\|_{L^{2}(\mathbb{R}_{+}^{n+1})}^{2}\\
\le & C\bigg[\|e^{\tau\phi}x_{n+1}^{\frac{1-2s}{2}}|x|\tilde{u}\|_{L^{2}(A_{1,2}^{+})}^{2}+\|e^{\tau\phi}x_{n+1}^{\frac{1-2s}{2}}|x|\nabla\tilde{u}\|_{L^{2}(A_{1,2}^{+})}^{2}\bigg].
\end{align*}
Pulling out the exponential weight in the above estimate yields 
\begin{align*}
 & \tau^{3}e^{\tau\tilde{\phi}(4)}\|x_{n+1}^{\frac{1-2s}{2}}\tilde{u}\|_{L^{2}(B_{6}^{+}\setminus B_{4}^{+})}^{2}+\tau e^{\tau\tilde{\phi}(4)}\|x_{n+1}^{\frac{1-2s}{2}}\nabla\tilde{u}\|_{L^{2}(B_{6}^{+}\setminus B_{4}^{+})}^{2}\\
\le & C\bigg[e^{\tau\tilde{\phi}(2)}\|x_{n+1}^{\frac{1-2s}{2}}\tilde{u}\|_{L^{2}(A_{1,2}^{+})}^{2}+e^{\tau\tilde{\phi}(2)}\|x_{n+1}^{\frac{1-2s}{2}}\nabla\tilde{u}\|_{L^{2}(A_{1,2}^{+})}^{2}\bigg].
\end{align*}

\textbf{Step 4: Conclusion.} Since $\tilde{\phi}(4)\ge\tilde{\phi(2)}$, taking $\tau\rightarrow\infty$
will leads a contradiction, unless $\tilde{u}=0$ in $B_{6}^{+}\setminus B_{4}^{+}$.
Finally, applying the unique continuation property for classical second order elliptic equations (see e.g. \cite[Theorem~1.1]{Reg97StrongUniquenessSecondOrderElliptic}), we conclude that $\tilde{u}\equiv0$. 
\end{proof}
Following exactly the arguments in \cite[Theorem~2]{RW19Landis}, we can obtain Theorem~\ref{thm:result2}. For sake of completeness, here we give a sketch of the proof of Theorem~\ref{thm:result2}. 

\begin{proof}
	[Sketch of the proof of Theorem~{\rm \ref{thm:result2}}] Let $\eta_{R}$
	be the function given in \eqref{eq:eta-R-cutoff}, and write $\overline{w}(t,\theta)=\overline{u}(t,\theta)\eta_{R}(e^{t}\theta)\equiv\tilde{u}(e^{t}\theta)\eta_{R}(e^{t}\theta)$,
	where $(t,\theta)$ is the conformal polar coordinate used in the
	proof of Carleman estimates (Theorem~\ref{thm:Carl-diff} and Theorem~\ref{thm:Carl-ndiff}).
	Pluging $\overline{w}$ into \eqref{eq:carl-ndff-polar} (i.e. the
	Carleman estimate in Theorem~\ref{thm:Carl-ndiff} with conformal
	polar coordinate) with $\varphi(t)=e^{\beta t}$ (that is, $\phi(x)=|x|^{\beta}$)
	with $\frac{4s}{4s-1}<\beta<\alpha$, and taking the limit $R\rightarrow\infty$,
	we obtain \cite[equation~(49)]{RW19Landis}: 
	\begin{align}
		& \tau^{3}\|e^{\tau\varphi}|\varphi'||\varphi''|^{\frac{1}{2}}\omega_{n+1}^{\frac{1-2s}{2}}\overline{w}\|_{L^{2}(\mathcal{S}_{+}^{n}\times\mathbb{R})}^{2}+\tau\|e^{\tau\varphi}|\varphi''|^{\frac{1}{2}}\omega_{n+1}^{\frac{1-2s}{2}}\partial_{t}\overline{w}\|_{L^{2}(\mathcal{S}_{+}^{n}\times\mathbb{R})}^{2}\nonumber \\
		& \quad+\tau\|e^{\tau\varphi}|\varphi''|^{\frac{1}{2}}\omega_{n+1}^{\frac{1-2s}{2}}\nabla_{\mathcal{S}^{n}}\overline{w}\|_{L^{2}(\mathcal{S}_{+}^{n}\times\mathbb{R})}^{2}\nonumber \\
		\le & C\bigg(\|e^{\tau\varphi}\omega_{n+1}^{\frac{2s-1}{2}}\tilde{f}\|_{L^{2}(\mathcal{S}_{+}^{n}\times[1,2])}^{2}+\tau^{2-2s}\|e^{\tau\varphi}|\varphi''|\tilde{q}e^{-\beta st}\overline{w}\|_{L^{2}(\partial\mathcal{S}_{+}^{n}\times\mathbb{R})}^{2}\bigg),\label{eq:Carl-utilize1}
	\end{align}
	with $|\tilde{f}|\le C\omega_{n+1}^{1-2s}(|\partial_{t}\overline{u}|+|\nabla_{\mathcal{S}^{n}}\overline{u}|+|\overline{u}|)$.
	Using the trace estimate in Proposition~\ref{prop:inter1} (by replacing
	$\tau$ by $e^{\beta t}\tau$), the boundary term in (\ref{eq:Carl-utilize1})
	can be absorbed in to the left-hand side of this estimate: 
	\begin{align}
		& \tau^{3}\|e^{\tau\varphi}|\varphi'||\varphi''|^{\frac{1}{2}}\omega_{n+1}^{\frac{1-2s}{2}}\overline{w}\|_{L^{2}(\mathcal{S}_{+}^{n}\times\mathbb{R})}^{2}+\tau\|e^{\tau\varphi}|\varphi''|^{\frac{1}{2}}\omega_{n+1}^{\frac{1-2s}{2}}\partial_{t}\overline{w}\|_{L^{2}(\mathcal{S}_{+}^{n}\times\mathbb{R})}^{2}\nonumber \\
		& \quad+\tau\|e^{\tau\varphi}|\varphi''|^{\frac{1}{2}}\omega_{n+1}^{\frac{1-2s}{2}}\nabla_{\mathcal{S}^{n}}\overline{w}\|_{L^{2}(\mathcal{S}_{+}^{n}\times\mathbb{R})}^{2}\le C\|e^{\tau\varphi}\omega_{n+1}^{\frac{2s-1}{2}}\tilde{f}\|_{L^{2}(\mathcal{S}_{+}^{n}\times[1,2])}^{2}.\label{eq:Carl-utilize2}
	\end{align}
	The observation $2\beta+4s-2\beta s\le\beta+2\beta s$ is helpful.
	Pulling out the weight $e^{\tau\varphi}$ in (\ref{eq:Carl-utilize2})
	leads to 
	\[
	e^{\tau\varphi(4)}\tau^{3}\||\varphi'||\varphi''|^{\frac{1}{2}}\omega_{n+1}^{\frac{1-2s}{2}}\overline{u}\|_{L^{2}(\mathcal{S}_{+}^{n}\times[4,6])}\le Ce^{\tau\varphi(2)}\|\omega_{n+1}^{\frac{2s-1}{2}}\tilde{f}\|_{L^{2}(\mathcal{S}_{+}^{n}\times[1,2])}^{2}.
	\]
	Using the monotonicity of $\varphi$, and passing to the limit $\tau\rightarrow\infty$,
	we know that $\overline{u}=0$ in $\mathcal{S}_{+}^{n}\times(4,6)$,
	i.e. $\tilde{u}=0$ in $B_{6}^{+}\setminus \overline{B_{4}^{+}}$. By unique continuation property, we conclude that $\tilde{u}\equiv0$
	in $\mathbb{R}_{+}^{n+1}$, which conclude the argument. 
\end{proof}

\appendix

\section{Auxiliary Lemmas}

\subsection{Some interpolation inequalities}

The following Hardy inequality can be found in \cite[Lemma~4.6]{RS20Calderon}: 
\begin{lem}
\label{lem:Hardy}If $\alpha\neq\frac{1}{2}$
and if $v$ vanishes for $x_{n+1}$ large, then 
\[
\|x_{n+1}^{-\alpha}u\|_{L^{2}(\mathbb{R}_{+}^{n+1})}^{2}\le\frac{4}{(2\alpha-1)^{2}}\|x_{n+1}^{1-\alpha}\partial_{n+1}u\|_{L^{2}(\mathbb{R}_{+}^{n+1})}^{2}+\frac{2}{2\alpha-1}\|\lim_{x_{n+1}\rightarrow0}x_{n+1}^{\frac{1}{2}-\alpha}u\|_{L^{2}(\mathbb{R}^{n}\times\{0\})}^{2}.
\]
\end{lem}

\begin{proof}
Using integration by parts, we have 
\begin{align*}
\|x_{n+1}^{-\alpha}u\|_{L^{2}(\mathbb{R}_{+}^{n+1})}^{2}= & \int\partial_{n+1}\bigg[\frac{x_{n+1}^{1-2\alpha}}{1-2\alpha}\bigg]u^{2}\\
= & \frac{2}{2\alpha-1}\int x_{n+1}^{1-2\alpha}u\partial_{n+1}u+\frac{1}{2\alpha-1}\int\lim_{x_{n+1}\rightarrow0}x_{n+1}^{1-2\alpha}u^{2}\\
\le & \frac{1}{2}\frac{4}{(2\alpha-1)^{2}}\|x_{n+1}^{1-\alpha}\partial_{n+1}u\|_{L^{2}(\mathbb{R}_{+}^{n+1})}^{2}+\frac{1}{2}\|x_{n+1}^{-\alpha}u\|_{L^{2}(\mathbb{R}_{+}^{n+1})}^{2}\\
 & +\frac{1}{2\alpha-1}\|\lim_{x_{n+1}\rightarrow0}x_{n+1}^{\frac{1}{2}-\alpha}u\|_{L^{2}(\mathbb{R}^{n}\times\{0\})}^{2},
\end{align*}
which gives our desired result. 
\end{proof}
We shall use the following interpolation inequality in \cite{GR19unique,Rul15unique,RW19Landis}: 
\begin{prop}
[Interpolation inequaliy I]\label{prop:inter1}Let $s\in(0,1)$ and
$u:\mathcal{S}_{+}^{n}\rightarrow\mathbb{R}$ with $u\in H^{1}(\mathcal{S}_{+}^{n},\omega_{n+1}^{1-2s})$.
Then there exists a constant $C=C(n,s)$ such that 
\[
\|u\|_{L^{2}(\partial\mathcal{S}_{+}^{n})}\le C\bigg[\tau^{1-s}\|\omega_{n+1}^{\frac{1-2s}{2}}u\|_{L^{2}(\mathcal{S}_{+}^{n})}+\tau^{-s}\|\omega_{n+1}^{\frac{1-2s}{2}}\nabla_{\omega}u\|_{L^{2}(\mathcal{S}_{+}^{n})}\bigg]
\]
for all $\tau>1$. 
\end{prop}

The following trace characterization lemma can be found in \cite[Lemma~4.4]{RS20Calderon}: 
\begin{lem}
\label{lem:traceHs}Let $n\ge1$ and $0<\tilde{s}<1$.
There is a bounded surjective linear map 
\[
T:H^{1}(\mathbb{R}_{+}^{n+1},x_{n+1}^{1-2\tilde{s}})\rightarrow H^{\tilde{s}}(\mathbb{R}^{n}\times\{0\})
\]
so that $u(\bullet,x_{n+1})\rightarrow Tu$ in $L^{2}(\mathbb{R}^{n})$
as $x_{n+1}\rightarrow0$. 
\end{lem}

We need the following interpolation inequality in \cite[Proposition 5.11: Step 1]{RS20Calderon}: 
\begin{lem}
[Interpolation inequality II(a)]\label{lem:inter2(a)}For any $w\in H^{1}(\mathbb{R}_{+}^{n+1},x_{n+1}^{2s-1})$
and any $\mu>0$, the following interpolation inequality holds: 
\[
\|w\|_{L^{2}(\mathbb{R}^{n}\times\{0\})}\le C\bigg[\mu^{1-s}(\|x_{n+1}^{\frac{2s-1}{2}}w\|_{L^{2}(\mathbb{R}_{+}^{n+1})}+\|x_{n+1}^{\frac{2s-1}{2}}\nabla w\|_{L^{2}(\mathbb{R}_{+}^{n+1})})+\mu^{-s}\|w\|_{H^{-s}(\mathbb{R}^{n}\times\{0\})}\bigg].
\]
\end{lem}

\begin{proof}
Let $\langle\bullet\rangle:=\sqrt{1+|\bullet|^{2}}$. Note that 
\begin{align*}
\|w\|_{L^{2}(\mathbb{R}^{n}\times\{0\})} & =\bigg[\int_{\mathbb{R}^{n}\times\{0\}}(\langle\xi\rangle^{2-2s}|\hat{w}|^{2})^{s}(\langle\xi\rangle^{-2s}|\hat{w}|^{2})^{1-s}\,d\xi\bigg]^{\frac{1}{2}}\\
 & \le(\mu^{1-s}\|w\|_{H^{1-s}(\mathbb{R}^{n}\times\{0\})})^{s}(\mu^{-s}\|w\|_{H^{-s}(\mathbb{R}^{n}\times\{0\})})^{1-s}
\end{align*}
and hence our result follows by Lemma~\ref{lem:traceHs} with $\tilde{s}=1-s$. 
\end{proof}
Slightly modify the proof, we can obtain the following: 
\begin{lem}
[Interpolation inequality II(b)]\label{lem:inter2(b)}For any $w\in H^{1}(\mathbb{R}_{+}^{n+1},x_{n+1}^{2s-1})$
and any $\mu>0$, the following interpolation inequality holds: 
\[
\|w\|_{L^{2}(\mathbb{R}^{n}\times\{0\})}\le C\bigg[\mu^{1-s}(\|x_{n+1}^{\frac{2s-1}{2}}w\|_{L^{2}(\mathbb{R}_{+}^{n+1})}+\|x_{n+1}^{\frac{2s-1}{2}}\nabla w\|_{L^{2}(\mathbb{R}_{+}^{n+1})})+\mu^{-2s}\|w\|_{H^{-2s}(\mathbb{R}^{n}\times\{0\})}\bigg].
\]
\end{lem}

\begin{proof}
Using Lemma \ref{lem:traceHs} with $\tilde{s}=1-s$, we have 
\begin{align*}
\|w\|_{L^{2}(\mathbb{R}^{n}\times\{0\})} & \le C\|w\|_{H^{1-s}(\mathbb{R}^{n}\times\{0\})}^{\frac{2s}{1+s}}\|w\|_{H^{-2s}(\mathbb{R}^{n}\times\{0\})}^{\frac{1-s}{1+s}}\\
 & \le C\bigg(\|x_{n+1}^{\frac{2s-1}{2}}w\|_{L^{2}(\mathbb{R}_{+}^{n+1})}+\|x_{n+1}^{\frac{2s-1}{2}}\nabla w\|_{L^{2}(\mathbb{R}_{+}^{n+1})}\bigg)^{\frac{2s}{1+s}}\|w\|_{H^{-2s}(\mathbb{R}^{n}\times\{0\})}^{\frac{1-s}{1+s}}\\
 & \le C\bigg[\mu^{1-s}\bigg(\|x_{n+1}^{\frac{2s-1}{2}}w\|_{L^{2}(\mathbb{R}_{+}^{n+1})}+\|x_{n+1}^{\frac{2s-1}{2}}\nabla w\|_{L^{2}(\mathbb{R}_{+}^{n+1})}\bigg)+\mu^{-2s}\|w\|_{H^{-2s}(\mathbb{R}^{n}\times\{0\})}^{\frac{1-s}{1+s}}\bigg],
\end{align*}
which is our desired result. 
\end{proof}

\subsection{Caccioppoli inequality}

We need a generalized the Caccioppoli inequality in \cite[Lemma~4.5]{RS20Calderon}: 
\begin{lem}
\label{lem:Caccio}Let $s\in(0,1)$ and $u\in H^{1}(B_{2r}^{+},x_{n+1}^{1-2s})$
be a solution to 
\[
\bigg[\partial_{n+1}x_{n+1}^{1-2s}\partial_{n+1}+x_{n+1}^{1-2s}P\bigg]\tilde{u}=-x_{n+1}^{1-2s}\sum_{j=1}^{n}\partial_{j}f_{j}\quad\text{in }B_{2r}^{+}.
\]
Then there exists a constant $C=C(n,\lambda)$ such that 
\begin{align*}
 & \|x_{n+1}^{\frac{1-2s}{2}}\nabla\tilde{u}\|_{L^{2}(B_{r}^{+})}^{2}\\
\le & C\bigg[r^{-2}\|x_{n+1}^{\frac{1-2s}{2}}\tilde{u}\|_{L^{2}(B_{2r}^{+})}^{2}+\sum_{j=1}^{n}\|x_{n+1}^{\frac{1-2s}{2}}f_{j}\|_{L^{2}(B_{2r}^{+})}^{2}+\|\lim_{x_{n+1}\rightarrow0}x_{n+1}^{1-2s}\partial_{n+1}\tilde{u}\|_{L^{2}(B_{2r}')}\|u\|_{L^{2}(B_{2r}')}\bigg].
\end{align*}
\end{lem}

\begin{proof}
Let $\eta:B_{2r}^{+}\rightarrow\mathbb{R}$ be a smooth, radial cut-off
function such that $0\le\eta\le1$, $\eta=1$ on $B_{r}^{+}$, ${\rm supp}(\eta)\subset B_{2r}^{+}$,
and $|\nabla\eta|\le C/r$ for some constant $C$. Note that 
\begin{align}
 & 2\sum_{j=1}^{n}\int_{\mathbb{R}_{+}^{n+1}}\bigg(x_{n+1}^{\frac{1-2s}{2}}\eta f_{j}\bigg)\bigg(x_{n+1}^{\frac{1-2s}{2}}(\partial_{j}\eta)\tilde{u}\bigg)+\sum_{j=1}^{n}\int_{\mathbb{R}_{+}^{n+1}}\bigg(x_{n+1}^{\frac{1-2s}{2}}\eta f_{j}\bigg)\bigg(x_{n+1}^{\frac{1-2s}{2}}\eta\partial_{j}\tilde{u}\bigg)\nonumber \\
= & -\sum_{j=1}^{n}\int_{\mathbb{R}_{+}^{n+1}}x_{n+1}^{1-2s}(\partial_{j}f_{j})(\eta^{2}\tilde{u})\nonumber \\
= & \int_{\mathbb{R}_{+}^{n+1}}\bigg(\partial_{n+1}x_{n+1}^{1-2s}\partial_{n+1}\tilde{u}+x_{n+1}^{1-2s}\sum_{i,j=1}^{n}\partial_{i}a_{ij}\partial_{j}\tilde{u}\bigg)(\eta^{2}\tilde{u})\nonumber \\
= & -\int_{\mathbb{R}^{n}\times\{0\}}\eta^{2}\tilde{u}\lim_{x_{n+1}\rightarrow0}x_{n+1}^{1-2s}\partial_{n+1}\tilde{u}-\int_{\mathbb{R}_{+}^{n+1}}(x_{n+1}^{1-2s}\partial_{n+1}\tilde{u})\partial_{n+1}(\eta^{2}\tilde{u})\nonumber \\
 & -\int_{\mathbb{R}_{+}^{n+1}}x_{n+1}^{1-2s}\sum_{i,j=1}^{n}a_{ij}\partial_{j}\tilde{u}\partial_{i}(\eta^{2}\tilde{u})\nonumber \\
= & -\int_{\mathbb{R}^{n}\times\{0\}}\eta^{2}\tilde{u}\lim_{x_{n+1}\rightarrow0}x_{n+1}^{1-2s}\partial_{n+1}\tilde{u}-2\int_{\mathbb{R}_{+}^{n+1}}(x_{n+1}^{1-2s}\partial_{n+1}\tilde{u})\eta\partial_{n+1}\eta\tilde{u}\nonumber \\
 & -\int_{\mathbb{R}_{+}^{n+1}}\eta^{2}(x_{n+1}^{1-2s}\partial_{n+1}\tilde{u})\partial_{n+1}\tilde{u}-2\int_{\mathbb{R}_{+}^{n+1}}x_{n+1}^{1-2s}\sum_{i,j=1}^{n}a_{ij}(\eta\partial_{j}\tilde{u})(\partial_{i}\eta\tilde{u})\nonumber \\
 & -\int_{\mathbb{R}_{+}^{n+1}}\eta^{2}x_{n+1}^{1-2s}\bigg(\sum_{i,j=1}^{n}a_{ij}\partial_{j}\tilde{u}\partial_{i}\tilde{u}\bigg)\nonumber \\
= & -\int_{\mathbb{R}^{n}\times\{0\}}\lim_{x_{n+1}\rightarrow0}\eta^{2}\tilde{u}x_{n+1}^{1-2s}\partial_{n+1}\tilde{u}-2\langle\eta\nabla\tilde{u},\tilde{u}\nabla\eta\rangle-\|\eta\nabla\tilde{u}\|^{2}\label{eq:Cacc1}
\end{align}
where $\tilde{A}=\begin{pmatrix}\begin{array}{cc}
A & 0\\
0 & 1
\end{array}\end{pmatrix}$. Here we use the notation 
\[
\langle\bullet,\bullet\rangle=\langle\bullet,\bullet\rangle_{L^{2}(\mathbb{R}_{+}^{n},x_{n+1}^{1-2s}\tilde{A})}\quad\text{and}\quad\|\bullet\|=\|\bullet\|_{L^{2}(\mathbb{R}_{+}^{n},x_{n+1}^{1-2s}\tilde{A})}.
\]
By \eqref{eq:unif-ellip}, indeed 
\[
\|\eta\nabla\tilde{u}\|^{2}\ge\lambda\|\eta x_{n+1}^{\frac{1-2s}{2}}\nabla\tilde{u}\|_{L^{2}(\mathbb{R}_{+}^{n+1})}^{2}\ge\lambda\|x_{n+1}^{\frac{1-2s}{2}}\nabla\tilde{u}\|_{L^{2}(B_{r}^{+})}^{2}.
\]
Also, by \eqref{eq:unif-ellip}, for $\delta>0$, we have 
\begin{align*}
2\langle\eta\nabla\tilde{u},\tilde{u}\nabla\eta\rangle & \le\delta\|\eta\nabla\tilde{u}\|^{2}+\delta^{-1}\|\tilde{u}\nabla\eta\|^{2}\\
 & \le\delta\lambda^{-1}\|\eta x_{n+1}^{\frac{1-2s}{2}}\nabla u\|_{L^{2}(\mathbb{R}_{+}^{n+1})}^{2}+\delta^{-1}\lambda^{-1}\|\nabla\eta x_{n+1}^{\frac{1-2s}{2}}u\|_{L^{2}(\mathbb{R}_{+}^{n+1})}^{2}.
\end{align*}
Moreover, we have 
\[
\bigg|\int_{\mathbb{R}^{n}\times\{0\}}\lim_{x_{n+1}\rightarrow0}\eta^{2}\tilde{u}x_{n+1}^{1-2s}\partial_{n+1}\tilde{u}\bigg|\le\|\lim_{x_{n+1}\rightarrow0}x_{n+1}^{1-2s}\partial_{n+1}\tilde{u}\|_{L^{2}(B_{2r}')}\|\eta^{2}\tilde{u}\|_{L^{2}(B_{2r}')}.
\]
Plug the inequalities above into \eqref{eq:Cacc1}, with small $\delta>0$,
we obtain our desired result. 
\end{proof}

\subsection{$L^{\infty}$-$L^{2}$ type interior inequality}

Following the arguments in \cite[Proposition 3.1]{TX11HarnackFractionalLaplace} (see also \cite[Proposition 2.6]{JLX14Nirenberg} or \cite[Proposition 3.2]{FF14unique}),
we can obtain the following: 
\begin{lem}
\label{lem:Linfty-L2}Let $s\in(0,1)$ and $u\in H^{1}(B_{2r}^{+},x_{n+1}^{1-2s})$
be a solution to 
\begin{align*}
	\bigg[\partial_{n+1}x_{n+1}^{1-2s}\partial_{n+1}+x_{n+1}^{1-2s}P\bigg]\tilde{u} & =0\quad\text{in }\mathbb{R}_{+}^{n+1},\\
	\tilde{u} & =u\quad\text{on }\mathbb{R}^{n}\times\{0\},\\
	\lim_{x_{n+1}\rightarrow0}x_{n+1}^{1-2s}\partial_{n+1}\tilde{u}(x) & =Vu\quad\text{on }\mathbb{R}^{n}\times\{0\},
\end{align*}
with \eqref{eq:unif-ellip} and $|V| \le 1$. Then there exists a constant $C=C(n,\lambda)$
such that 
\[
\|\tilde{u}\|_{L^{\infty}(B_{1/2}^{+})}\le C\bigg[\|x_{n+1}^{\frac{1-2s}{2}}\tilde{u}\|_{L^{2}(B_{1}^{+})}+\|x_{n+1}^{\frac{1-2s}{2}}\nabla\tilde{u}\|_{L^{2}(B_{1}^{+})}\bigg].
\]
\end{lem}

\section*{Acknowledgments}

I would like to thank Prof. Jenn-Nan Wang for suggesting the problem and for many helpful discussions. This research is partially supported by MOST 105-2115-M-002-014-MY3, MOST 108-2115-M-002-002-MY3, and MOST 109-2115-M-002-001-MY3. 

\bibliographystyle{alpha}
\bibliography{ref}

\end{document}